\documentclass{amsart}
\usepackage{amsthm,amsmath,amssymb,amscd, mathrsfs}
\usepackage[latin1]{inputenc}
\usepackage[all]{xy}
\numberwithin{equation}{subsection}

 \newtheorem{ithm}[subsection]{Theorem}
 \newtheorem{thm}[equation]{Theorem}
 \newtheorem{prop}[equation]{Proposition}
 \newtheorem{lemma}[equation]{Lemma}
 
 \newtheorem{cor}[equation]{Corollary}

 \theoremstyle{definition}
 \newtheorem{definition}[equation]{Definition}
 \newtheorem{ex}[equation]{Example}

 \theoremstyle{remark}
 \newtheorem{remark}[equation]{Remark}
 \newtheorem{para}[equation]{}

 \def\Spec{{\rm Spec\,}}
 \def\Spf{{\rm Spf}}

 \def\End{{\rm End}}

\def\an{{\rm an}}

 \def\ad{{\rm ad}}
\def\der{{\rm der}}
 \def\det{{\rm det}}
 \def\ab{{\rm ab}}

 \def\dom{{\rm dom}}
\def\an{{\rm an}}
\def\SL{{\rm SL}}
\def\GL{{\rm GL}}
\def\PGL{{\rm PGL}}
\def\GSp{{\rm GSp}}

\def \iso {\overset \sim \longrightarrow}

\def\O{{\mathcal{O}}}
\def\e{{\epsilon}}
\def\ep{{\psi}}
\def\R{{\mathscr R}}

\def\cR{{\mathcal R}}
\def\RR{{\mathbb R}}
\def\Lie{{\rm Lie}}
\def\Z{\mathbb{Z}}
\def\Q{\mathbb{Q}}
\def\DD{\mathbb{D}}
\def\GG{\mathbb{G}}
\def\Res{{\rm Res}}

\def\ilim{\underleftarrow{\text {\rm lim}}}
\def\lps{[\![}
\def\rps{]\!]}
\def\th{{\rm th}}
\def\Gal{{\rm Gal}}
\def\Ind{{\rm Ind}}
\def\W{{Q}}
\def\Weyl{{W}}

\newcommand{\simto}{\stackrel{\sim}{\to}}
\begin{document}
\title{Connected components of affine Deligne-Lusztig varieties in mixed characteristic}
\author{Miaofen Chen}
\email{mfchen@math.ecnu.edu.cn}
\address{Department of Mathematics\\
Shanghai Key Laboratory of PMMP\\
East China Normal University\\
No. 500, Dong Chuan Road\\
Shanghai, 200241, P.R.China.}
\author{Mark Kisin}
\email{kisin@math.harvard.edu}
\address{Department of Mathematics\\
Harvard University\\
1 Oxford St\\
Cambridge, MA 02138, USA.
}
\author{Eva Viehmann}
\email{viehmann@ma.tum.de}
\address{Fakult\"at f\"ur Mathematik der
Technischen Universit\"at M\"unchen -- M11,
Boltzmannstr. 3,
85748 Garching, Germany.}
\thanks{The second author was partially
supported by NSF grant DMS-1001139. The first and third authors
were partially supported by the SFB/TR45 ``Periods, Moduli Spaces and
Arithmetic of Algebraic Varieties'' of the DFG and by ERC starting
grant 277889 ``Moduli spaces of local $G$-shtukas''. The first author was also partially supported by NSFC grant No. 11301185 and SRFDP grant No. 20130076120002. The third author
was also partially supported by a Heisenberg fellowship of the DFG.}
\date{}
\begin{abstract}{We determine the set of connected components of minuscule affine Deligne-Lusztig varieties for hyperspecial maximal compact subgroups of unramified connected reductive groups. Partial results are also obtained for non-minuscule closed affine Deligne-Lusztig varieties. We consider both the function field case and its analog in mixed characteristic. In particular, we determine the set of connected components of unramified Rapoport-Zink spaces. }
\end{abstract}
\maketitle
\tableofcontents

\section{Introduction}\label{secdlvcontext}

Let $k$ be a finite field with $q=p^r$ elements and let $\overline{k}$ be an algebraic closure of $k$. Let $F=k((t))$ or $F=W(k)[1/p]$. Let accordingly $L=\overline{k}((t))$ or $L=W(\overline k)[1/p]$. Let $\mathcal{O}_F$ and $\mathcal{O}_L$ be the valuation rings of $F$ and $L.$
We denote by $\e$ the uniformizer $t$ or $p$. We write $\sigma:x\mapsto x^q$  for the Frobenius of $\overline{k}$ over $k$ and also the induced Frobenius of $L$ over $F$.

Let $G$ be a connected reductive group over $\mathcal{O}_F$. We denote by $G_F$ the generic fibre of $G,$ and write $K = G(\O_L).$ Since $k$ is finite $G$ is automatically quasi-split. Let $B \subset G$ be a Borel subgroup and $T \subset B$ the centralizer of a maximal split torus in $B.$
We denote by $X_*(T)$ the set of cocharacters of $T,$ defined over $\O_L.$

For $b\in G(L)$ and a dominant cocharacter $\mu\in X_*(T)$ the affine Deligne-Lusztig variety $X^G_{\mu}(b)=X_{\mu}(b)$ (which is in fact in general just a set of points) is defined as
\begin{equation*}
X_{\mu}(b) =\{g\in G(L)/K\mid g^{-1}b\sigma(g)\in K\e^{\mu}K\}.
\end{equation*}
Left multiplication by $g\in G(L)$ induces a bijection $X_{\mu}(b) \iso X_{\mu}(gb\sigma(g)^{-1})$. Thus the isomorphism class of the affine Deligne-Lusztig variety only depends on the $\sigma$-conjugacy class $[b]$ of $b,$ and not on $b.$

When $F$ has mixed characteristic, and $\mu$ is minuscule the sets $X_{\mu}(b)$ are closely related to the $\overline k$-points on Shimura varieties which lie in a fixed isogeny class, and in special cases to
$\overline k$-valued points of a moduli space of $p$-divisible groups as defined by Rapoport and Zink \cite{RZ}.

 If $F$ is a function field, then $X_{\mu}(b)$ is the set of $\overline k$-valued points of a locally closed, locally of finite type subscheme of the affine Grassmannian $LG/K$ where $LG$ denotes the loop group of $G$ (compare \cite{Rapoport1}, \cite{GHKR}).
 If $F$ has mixed characteristic, there is, in general, no known scheme structure on the affine Deligne-Lusztig varieties
\footnote{In fact in this case $X_{\mu}(b)$ is defined as a functor, not on $\bar k$-algebras, but rather on certain $p$-adically complete
$W(\bar k)$-algebras equipped with a lift of Frobenius. For this reason, what we have denoted $X_{\mu}(b)$ in the introduction is denoted
$X_{\mu}(b)(W(\bar k))$ in the body of the paper.}.
Nevertheless, they admit some kind of geometric structure,
and in particular a meaningful notion of a set of connected components $\pi_0(X_{\mu}(b))$ which is compatible with the corresponding notion for
Rapoport-Zink spaces.

The aim of this paper is to compute the set of connected components of $X_{\mu}(b)$ for any $b$ when $\mu$ is minuscule. To state our main results, we begin by recalling when
$X_{\mu}(b) \neq \emptyset.$ This condition is a relation between $\mu$ and the $\sigma$-conjugacy class of $b$. Let $B(G)$ denote the set of $\sigma$-conjugacy classes of all elements of $G(L)$. They are described by two invariants. Write $\pi_1(G)$ for the quotient of $X_*(T)$ by the coroot lattice of $G$.
Recall that there is the Kottwitz homomorphism (compare \cite{RR}) $w_G: G(L) \rightarrow \pi_1(G)$ which for $\mu \in X_*(T)$ sends an element
$g \in K \epsilon^{\mu}K \subset G(L)$ to the class of $\mu.$ We denote by
$\kappa_G$ the composite of $w_G$ with the projection $\pi_1(G) \rightarrow \pi_1(G)_{\Gamma},$
where $\Gamma = \Gal(\bar k/k)$ acts in the natural way on $L$ and hence on $\pi_1(G).$
Let $\nu_{\dom} \in X_*(T)_{\Q}$ be the dominant cocharacter conjugate to the Newton cocharacter of $b$. Then $\nu_{\dom}$ is $\Gamma$-invariant and together with $\kappa_G(b)$ determines the $\sigma$-conjugacy class.

Let $\bar \mu \in X_*(T)_{\Q}$ denote the average of the $\Gamma$-conjugates of
$\mu.$ Then the set $X_{\mu}(b)$ is non-empty if and only if $\kappa_G(b)= [\mu]$ in $\pi_1(G)_{\Gamma},$
and $\bar \mu-\nu_{\dom}$ is a linear combination of positive coroots with non-negative rational
coefficients - see \cite{KottRapo}, \cite{Wint}, \cite{GHKR}, Prop. 5.6.1, and \cite{gashi}.
We denote by $B(G,\mu)$ the set of $\sigma$-conjugacy classes $[b] \in B(G)$ satisfying these conditions,
and we assume from now on that $[b] \in B(G,\mu).$

Then $w_G(b)-\mu = (1-\sigma)(c_{b,\mu})$ for an element $c_{b,\mu}\in\pi_1(G)$ whose
$\pi_1(G)^{\Gamma}$-coset is uniquely determined by this condition.
The following is one of our main results.

\begin{ithm}\label{thmzshk}
Assume that $G^{\ad}$ is simple and that $\mu$ is minuscule, and
suppose that $(\mu,b)$ is Hodge-Newton indecomposable in $G.$ Then $w_G$
induces a bijection $$\pi_0(X_{\mu}(b))\cong c_{b,\mu}\pi_1(G)^{\Gamma}$$
unless $[b] = [\e^{\mu}]$ with $\mu$ central, in which case
$$X_{\mu}(b) \cong G(F)/G(\mathcal{O}_F)$$ is discrete.
\end{ithm}

Here $G^{\ad}$ denotes the adjoint group of $G$. The definition of Hodge-Newton indecomposablility will be recalled below in \S \ref{secpf1}.
In fact, without assuming that $G^{\ad}$ is simple, we show that $w_G$ induces an isomorphism as in the first case of the theorem provided
$(\mu,b)$ is {\em Hodge-Newton irreducible}, a condition slightly stronger than Hodge-Newton indecomposability, which is also
recalled in \S 2.5. When $G^{\ad}$ is simple, a Hodge-Newton indecomposable pair $(\mu, b)$ is Hodge-Newton irreducible
unless $[b] = [\e^{\mu}]$ with $\mu$ central.
At the end of \S \ref{secpf1} we also give the easy direct calculation showing the last assertion of the theorem.

The theorem describes $\pi_0(X_{\mu}(b))$ (for $\mu$ minuscule) when $G^{\ad}$ is simple and
$(\mu,b)$ Hodge-Newton is indecomposable in $G.$ The general case without these assumptions (but with $\mu$ still being minuscule) can be reduced to this one. Indeed, for any element $b \in G(L)$ there exists a $b' \in G(L)$ that is
$\sigma$-conjugate to $b,$ and a standard Levi subgroup $M \subset G$ such that $b' \in M(L)$
and $(\mu,b')$ is Hodge-Newton indecomposable in $M,$ and such that the natural map
$X^M_{\mu}(b') \rightarrow X^G_{\mu}(b')$ is a bijection.

To reduce to the case when $G$ is adjoint and simple, we again denote by $b$ and $\mu$ the images of $b$ and $\mu$ in $G^{\ad}.$
Then the sets of connected components of $X^G_{\mu}(b)$ and $X^{G^{\ad}}_{\mu}(b)$ are closely related.
More precisely, we prove in \S \ref{secred}
that the diagram
$$\begin{CD}
  \pi_0(X_{\mu}^G(b))  @>>> \pi_0(X_{\mu}^{G^{\ad}}(b)) \\
^{w_G}@VVV^{w_{G^{\ad}}} @VVV      \\
c_{b,\mu}\pi_1(G)^{\Gamma}   @>>> c_{b,\mu}\pi_1(G^{\ad})^{\Gamma}
\end{CD}$$
is Cartesian. Furthermore, affine Deligne-Lusztig varieties for products of groups are products of the affine Deligne-Lusztig varieties for the individual factors. This reduces the description of $\pi_0(X_{\mu}(b))$ from the general case to the case where $G$ itself is simple.


In the course of the proof we obtain the following theorem (which is also a consequence of Theorem \ref{thmzshk}).
It is less precise but has the advantage that it does not require any additional assumptions. Define an $F$-group $J_b$ by setting
$$ J_b(R) : =  \{g \in G(R\otimes_FL): \sigma(g) = b^{-1}gb \}.$$ for $R$ an $F$-algebra.
There is an inclusion $J_b \subset G,$ defined over $L,$
which is given on $R$-points ($R$ an $L$-algebra) by the natural map $G(R\otimes_FL) \rightarrow G(R).$
\begin{ithm}\label{prop1} If $\mu$ is minuscule then
$J_b(F)$ acts transitively on $\pi_0(X_{\preceq\mu}(b))$.
\end{ithm}

In fact we will show in Theorem \ref{prop1'} that already the action of a certain subgroup of $J_b(F)$ is transitive.

Our description of the connected components is used in an essential way in the work of one of us [Ki] on the Langlands-Rapoport conjecture for mod $p$ points on Shimura varieties. Our results also allow us to get a description of the set of connected components of (simple) unramified Rapoport-Zink spaces of PEL type.

More precisely, suppose $(G, b, \mu)$ is a (simple) unramified Rapoport-Zink datum of EL type or unitary/symplectic PEL type (for the precise definition, see Section \ref{sec_app_RZ}). To this kind of datum we can associate a Rapoport-Zink space $\breve{\mathcal{M}}=\breve{\mathcal{M}}(G, b, \mu)$  which is a formal scheme locally formally of finite type over $\Spf\mathcal{O}_L$ (cf. \cite{RZ}). By the Dieudonn\'e-Manin classification of isocrystals over $\bar{\mathbb{F}}_p$, there exists a natural bijection $\theta: \breve{\mathcal{M}}(G, b, \mu)(\bar{\mathbb{F}}_p)\simeq X^G_\mu(b)$. Let $\breve{\mathcal{M}}^{\an}$ be the generic fiber of $\breve{\mathcal{M}}$ as a Berkovich analytic space. There exists a tower of finite \'etale covers $(\breve{\mathcal{M}}_{\tilde K})_{\tilde K\subset G(\mathbb{Z}_p)}$ on $\breve{\mathcal{M}}^{\an}$ parametrizing the $\tilde K$-level structures on the Tate-module of the universal $p$-divisible group,
where $\tilde K$ runs through the open subgroups in $G(\mathbb{Z}_p)$. Let $\mathbb{C}_p$ be the completion of an algebraic closure of $\Q_p,$
write $\pi_0(\breve{\mathcal{M}}_{\tilde K}\hat{\otimes}\mathbb{C}_p)$ for the set of geometrically connected components of $\breve{\mathcal{M}}_{\tilde K}.$
The group $J_b(\Q_p)\times G(\Q_p) \times \Gal(\bar L/L)$ acts naturally on $\pi_0(\breve{\mathcal{M}}_{\tilde K}\hat{\otimes}\mathbb{C}_p),$ where $\bar L$
is the intgeral closure of $L$ in $\mathbb{C}_p.$ Moreover, there is a natural map
$$\delta = (\delta_{J_b},\delta_G,\chi_{\delta_G,\mu}): J_b(\Q_p)\times G(\Q_p) \times \Gal(\bar L/L) \rightarrow G^{\ab}(\Q_p),$$
where the maps $\delta_{J_b}$ and $\delta_G$ are the natural ones, and $\chi_{\delta_G,\mu}$ is given by the Artin reciprocity map
and the reflex norm of $\mu.$
Then our main result implies the following theorem (see \ref{geo_conn_comp} below, cf. \cite{chen} Theorem 6.3.1).

\begin{ithm} If $(b, \mu)$ is Hodge-Newton irreducible, then the action of $J_b(\Q_p)\times G(\Q_p) \times \Gal(\bar L/L)$
on $\pi_0(\breve{\mathcal{M}}_{\tilde K}\hat{\otimes}\mathbb{C}_p)$ factors through $\delta,$ and makes $\pi_0(\breve{\mathcal{M}}_{\tilde K}\hat{\otimes}\mathbb{C}_p)$
into a $G^{\ab}(\Q_p)/\delta(\tilde K)$-torsor. In particular, there
exist bijections
$$ \pi_0(\breve{\mathcal{M}}_{\tilde K}\hat{\otimes}\mathbb{C}_p) \iso G^{\ab}(\mathbb{Q}_p)/\delta(\tilde K)$$
which are compatible when $\tilde K$ varies.
\end{ithm}

For dominant elements $\mu,\mu'\in X_*(T)$ we say that $\mu'\preceq \mu$ if $\mu-\mu'$ is a non-negative {\em integral} linear combination of positive coroots. The closed affine Deligne-Lusztig variety is defined as
\begin{equation*}
X_{\preceq \mu}(b)=\bigcup_{\mu'\preceq \mu}X_{\mu'}(b).
\end{equation*}

If $\mu$ is minuscule, $X_{\mu}(b) \cong X_{\preceq\mu}(b).$ We conjecture
that Theorem \ref{thmzshk} remains true without the assumption that $\mu$ is minuscule
if we replace $X_{\mu}(b)$ by $X_{\preceq\mu}(b)$ in the statement. For split groups this is proved
in \cite{conncomp} in the function field case. For split groups in mixed characteristic it can be deduced by combining the arguments in \cite{conncomp} with the theory of connected components of affine Deligne-Lusztig varieties in mixed characteristic developed in the present paper.

The proofs of the theorems are organized as follows: In Section 2 we collect some
foundational results including the behavior of the Cartan decomposition in a family, the definition of the affine Grassmannian and affine Deligne-Lusztig varieties in mixed characteristic. We also make the reductions discussed above, first to the case where $(\mu,b)$ is Hodge-Newton indecomposable, and then to the case when $G$ is adjoint and simple.
In Section \ref{secsb} we prove Theorem \ref{thmzshk} for the case that $b$ is superbasic, i.e.~under the assumption that $b$ is not $\sigma$-conjugate to an element of any proper Levi subgroup of $G$. In Proposition \ref{(1.5.5)} we show that each connected component contains an element of $J_b(F)X_{\preceq\mu'}^M(b)$ where $M$ is a standard Levi subgroup such that $b$ is superbasic in $M$ and $\mu'$ is an $M$-dominant cocharacter with $\mu'_{\dom}\preceq\mu$. Until this point we do not assume that $\mu$ is minuscule. Finally in Section \ref{seccp} we assume that $\mu$ is minuscule and we connect suitable representatives of the connected components of all $X_{\preceq\mu'}^M(b)$ by one-dimensional subvarieties in $X_{\mu}^G(b).$ Here the reader may wish to first consider the case when $G$ is a split
group, as this substantially simplifies the arguments.

Apart from this introduction we only consider the arithmetic case. Proofs for the function field case are completely analogous, but simpler.

\noindent{\it Acknowledgement.} The authors would like thank Robert Kottwitz, Dennis Gaitsgory and Jilong Tong for useful discussions, and Xuhua He and Rong Zhou for useful comments on a previous version of the manuscript. We thank the referee for helpful comments.

\section{Affine Deligne-Lusztig varieties in mixed characteristic}\label{1}

\subsection{The Cartan decomposition in families}

\begin{para}\label{Notation}
Let $F = W(k)[1/p]$ with $k$ a finite field with $q = p^r$ elements.
Fix an algebraic closure $\bar k$ of $k,$ and let $L = W(\bar k)[1/p].$
Write $\Gamma = \Gal(\bar k/k).$ Then $\Gamma$ has a canonical topological generator
$\sigma$ given by $x \mapsto x^q,$ and acts in the natural way on $L.$
Let $G,B,T$ be as above, and write $\Weyl = \Weyl_G$ for the Weyl group of $T$ in $G.$

We have the Cartan decomposition \cite{Bruhat-Tits} 4.4.3
$$ G(L) = \coprod_{\mu} G(\O_L)p^\mu G(\O_L) $$
where $\mu$ runs over the dominant elements of $X_*(T).$
In particular, $\mu \mapsto p^{\mu}$ induces a bijection
\begin{equation}\label{cartan}
X_*(T)/{\Weyl} \iso G(\O_L)\backslash G(L)/G(\O_L).
\end{equation}


We write $\mu_{G-\dom}$ (or $\mu_{\dom}$ if the group is clear) for the dominant element in the orbit of $\mu\in X_*(T)$ under $\Weyl.$
For $\mu_1,\mu_2 \in X_*(T),$ we write $\mu_1\preceq \mu_2$ if $\mu_2 - \mu_1$ is a linear combination of positive coroots with {\em integral}, non-negative coefficients.
For $\nu_1,\nu_2 \in X_*(T)_{\RR}$ we write $\nu_1 \leq \nu_2$ if $\nu_2 - \nu_1$ is a linear combination of positive coroots with {\em real}, non-negative coefficients.
\end{para}

\begin{para}\label{1.1.2} Let $R$ be a $\bar k$-algebra. A {\it frame} for $R$ is a $p$-torsion free, $p$-adically complete and separated
$\O_L$-algebra $\R$ equipped with an isomorphism $\R/p\R \iso R,$ and a lift (again denoted $\sigma$)
of the $q$-Frobenius $\sigma$ on $R$ to $\R.$
When $q=p$ this is a special case of Zink's definition \cite{Zi3}, Definition 1.
If $\theta:R \rightarrow R'$ is a map of $\bar k$-algebras, then a frame for $\theta$
is a morphism of $\O_L$-algebras $\tilde \theta: \R \rightarrow \R'$ from a frame of $R$
to a frame of $R',$ which lifts $\theta,$ and is compatible with $\sigma.$

Let $\kappa$ be a perfect field. Any map $s: R \rightarrow \kappa$ admits a unique
$\sigma$-equivariant map $\R \rightarrow W(\kappa),$ which we will often again denote by $s.$
\end{para}

\begin{lemma} \label{1.1.5} Let $\R$ be a frame for $R.$ Then any \'etale morphism
$R \rightarrow R'$ admits a canonical frame $\R \rightarrow \R'.$
\end{lemma}
\begin{proof} Since the \'etale site is invariant under nilpotent thickenings, $R'$ lifts canonically
to an \'etale $\R/p^n\R$ algebra $\R'_n,$ and we set $\R' = \ilim \R'_n.$

Similarly, the canonical isomorphism $R'\otimes_{R,\sigma}R \underset{\sigma\otimes 1}{\iso} R'$ lifts to a unique isomorphism
$ \R'_n\otimes_{\R,\sigma}\R \iso \R'_n,$ and the composite
$$ \R'_n \overset{a\mapsto a\otimes 1}\longrightarrow \R'_n\otimes_{\R,\sigma}\R \iso \R'_n$$
lifts $\sigma$ on $\R'_n.$ Passing to the limit with $n$ we get a lift of $\sigma$ on $\R'.$
\end{proof}

\begin{para}
Fix a frame $\R$ for $R,$ and let $g \in G(\R_L).$
For a dominant $\mu \in X_*(T)$ let
$$ S_{\mu}(g) = \{s \in \Spec R: s(g) \in G(W(\bar \kappa(s)))p^{\mu} G(W(\bar \kappa(s)))\} $$
where $\bar\kappa(s)$ denotes an algebraic closure of $\kappa(s),$ and set
$$ S_{\preceq \mu}(g) = \cup_{\mu' \preceq \mu} S_{\mu'}(g),$$
where $\mu'$ runs over dominant cocharacters $\preceq \mu.$
\end{para}

\begin{lemma}\label{1.1.3} Let $R$ be a Noetherian, formally smooth $\bar k$-algebra,
$\R$ a frame for $R,$ and $g \in G(\R_L).$
\begin{enumerate}
\item
The subset $S_{\preceq \mu}(g) \subset S = \Spec R$ is Zariski closed.
\item The subset $S_{\mu}(g)$ is locally closed and is closed if $\mu$ is minuscule.
\item The function $s \mapsto [\mu_{s(g)}] \in \pi_1(G)$ is locally constant on $s \in \Spec R.$
\end{enumerate}
\end{lemma}
\begin{proof}  We begin by checking that $S_{\leq \mu}(g) = \{ s: \mu_{s(g)} \leq  \mu\}$ is closed in $S.$
By \cite{RR}, 2.2(iv) we have $\mu_{G,s(g)} \leq \mu$ if and only if for every representation
$\rho: G_L \rightarrow \GL(V)$ on an $L$-vector space $V,$ we have $\rho\circ\mu_{G,s(g)} \leq \rho\circ \mu.$

Choosing a suitable $\O_L$-lattice $\W \subset V,$ we may assume
that $\rho$ is induced by a map $G \rightarrow \GL(\W)$ over $\O_L$ (cf.~the proof of \cite{Ki4}, 2.3.1). Let $T' \subset \GL(\W)$ be a maximal $L$-split torus containing the image of $T.$
Then $\rho\circ\mu_{G,s(g)} = \mu_{GL,s(\rho(g))}$
in $X_*(T')/\Weyl_{\GL}$ where $\Weyl_{\GL}$ is the Weyl group of $T'$ in $\GL(\W).$
By \cite{Ka}, Cor.~2.3.2 the set of points at which the Hodge polygon of a $\sigma$-isocrystal lies on or above a given polygon and has the same endpoints, is closed in $S.$ Hence $S_{\leq \rho\circ\mu}(\rho(g)) \subset S$ is closed, and hence
$S_{\leq \mu}(g) \subset S$ is closed.

It follows in particular, that the function $s \mapsto [\mu_{G,s(g)}] \in \pi_1(G)\otimes_{\mathbb Z}\Q$ is locally constant
on $S,$ which proves (3) when $\pi_1(G)$ has no torsion. To prove (3) in general,
let $\tilde G$ be the universal cover of $G^{\der}$ and let $G' = \tilde G \times T.$
The kernel of the natural map $G' \rightarrow G$ is a maximal torus $T' \subset \tilde G.$
The obstruction to lifting $g$ to a point of $G'_L(\R_L)$ lies in $H^1(\Spec \R_L, T').$
Since $T'$ is a split torus this obstruction corresponds to a finite collection of line bundles over
$\Spec \R_L.$ Since $\R$ is regular any line bundle on $\Spec \R_L$ extends to a line bundle on $\Spec \R.$
Hence after replacing $S$ by a Zariski covering by affines, and $\R$ by the corresponding frame (see Lemma \ref{1.1.5}), we may assume
that $g$ lifts to a point  $g'\in G'_L(\R_L).$  By what we already saw, the function
$s \mapsto [\mu_{G,s(g')}] \in \pi_1(G')$ is locally constant, so $s \mapsto [\mu_{G,s(g)}] \in \pi_1(G)$ is locally constant.

To prove (1) and (2) we may assume that $S$ is connected. Then $[\mu_{s(g)}] \in \pi_1(G)$ does not
depend on $s,$ and $S_{\preceq\mu}(g)$ is empty unless $[\mu]$ is equal to this constant class.
If this condition holds, then $\mu_{s(g)} \preceq \mu$ if and only if
$\mu_{s(g)} \leq \mu.$ Thus, $S_{\preceq \mu}(g) = \{ s: \mu_{s(g)} \leq  \mu\},$ which we saw is closed.
This proves (1) and that $S_{\mu}(g)$ is locally closed. If $\mu$ is minuscule and $\mu' \preceq \mu$ is dominant with $[\mu] = [\mu']$ in $\pi_1(G),$ then  $\mu' = \mu,$ so (2) follows.
\end{proof}

\begin{para}\label{1.1.4} Suppose that $g \in G(\R_L)$ and $S_{\mu}(g) = S=\Spec R.$ Then a natural question is
whether $G(\R)p^\mu G(\R) \subset G(\R_L)$ contains $g.$ We will show that this is so \'etale
locally on $R,$ when $R$ is a reduced, Noetherian $\bar k$-algebra.
This will be used in \S 2.5 below. To do this we need some preparation.

By an {\it \'etale covering}, we mean a faithfully flat, \'etale morphism $R \rightarrow R'.$
We begin with the following simple lemma which allows us to work with frames \'etale locally on $R,$ and will allow us to often replace $R$ by an \'etale covering in arguments.
\end{para}

\begin{lemma} \label{1.1.6} Let $R$ be a reduced $\bar k$-algebra, and $\R$ a frame for $R.$
Suppose that $g \in G(\R_L)$ and $S_{\mu}(g) = S.$ If $\kappa \supset \bar k$
is a perfect field of characteristic $p,$ and $L'/W(\kappa)[1/p]$ a finite extension with ring of integers $\O_{L'},$
then for any map of $\O_L$-algebras $\xi: \R \rightarrow \O_{L'},$ we have
$$ \xi(g) \in G(\O_{L'})p^{\mu}G(\O_{L'}).$$
\end{lemma}
\begin{proof} As in the proof of Lemma \ref{1.1.3}, it suffices to consider the case $G = \GL(\W)$ for a finite free
$\O_L$-module $\W.$

For $\xi$ as in the lemma, let $i_{\xi}$ denote the greatest number in $e(L')^{-1}\mathbb Z$ such that
$\xi(g)(\W\otimes_{\O_L}\O_{L'}) \subset \pi_{L'}^{e(L')i_{\xi}}\W\otimes_{\O_L}\O_{L'},$
where $\pi_{L'}$ is uniformizer for $L'$ and $e(L')$ is the absolute ramification degree of $L'.$
Our assumptions imply that if $\xi$ is a map
$s: \R \rightarrow \O_L = W(\bar k)$ induced by a closed point $s: R \rightarrow \bar k,$
then $i_{\xi}$ has a value $i_0 \in \mathbb Z$ which does not depend on $s.$

We claim that $i_{\xi} = i_0$ for any $\xi.$ To see this we may multiply $\mu$ by a central
character and $g$ by a scalar, and assume that $i_0 \geq 0,$ and that $g$ stabilizes $\W\otimes_{\O_L}\R.$
If $i_0 > 0,$ then $g$ induces an endomorphism of
$\W\otimes_{\O_L}R$ which vanishes at every closed point of $R,$ and hence is identically $0$
as $R$ is reduced. Hence $g(\W\otimes_{\O_L}\R) \subset p(\W\otimes_{\O_L}\R).$ Thus, after
again multiplying $\mu$ by a central character, we may assume that $i_0 = 0$ and $g$ leaves
$\W\otimes_{\O_L}\R$ stable. This implies that $g$ induces an endomorphism of $\W\otimes_{\O_L}R,$
which is non-zero at every closed point, and hence $i_{\xi}=0$ for all $\xi.$

Now the lemma follows by applying the claim just proved to the exterior powers of $\W.$
\end{proof}

\begin{para}\label{(1.1.7)} Suppose that $\W$ is a finite free $\O_L$-module equipped with an action of $G.$
For $\mu\in X_*(G)$ we denote by $\mu^\W$ the $\GL(\W)$-valued cocharacter given by $z \mapsto z^i\mu(z),$ where $i$
is the integer such that the eigenvalues of $p^i\mu(p)$ acting on $\W$ are non-negative powers of
$p,$ and include $1.$ Let $P_{\mu}(\W)\subset G\times_{\O_F}G$ be the subgroup whose
points are pairs $(g_1,g_2)$ such that $g_1\mu^\W(p) = \mu^\W(p)g_2$ in $\End~\W.$
Note that this need not be a flat subgroup, in general.

Similarly, if $\alpha$ is a collection of finite free $\O_F$-modules equipped with an action of $G,$
then we denote by $P_{\mu}(\alpha)$ the intersection of the $P_{\mu}(\W) \subset G\times_{\O_F}G$ for
$\W \in \alpha.$ Note that the generic fibre of $P_{\mu}(\alpha)$ may be identified with $G$ via the
embedding
$$G \rightarrow G\times G: \quad g \mapsto (g,\mu(p)^{-1}g\mu(p)).$$
\end{para}

\begin{lemma} \label{1.1.8} Let $G \hookrightarrow \GL(\W)$ be a faithful representation of $G$ on a finite
free module $\W,$ let $\alpha = \{\wedge^i \W\}_{i \geq 1}$, and let $\mu\in X_*(G)$. Then $P_{\mu}(\alpha)$ is a smooth model of
$G,$ and may be identified with the closure of the embedding $G \rightarrow G\times G$ above.
\end{lemma}
\begin{proof} Let $P_{\mu} \subset G$ denote the parabolic defined by $\mu,$ so that $\Lie~P_{\mu} \subset \Lie~G$
is the submodule on which $\mu$ acts by non-negative weights. Similarly, let $P_{\mu}^{\circ}$ denote the opposite
parabolic and $M_{\mu}$ the common reductive quotient of $P_{\mu}$ and $P_{\mu}^{\circ}.$
We will use a subscript of $k$ to denote the special fibre of an $\O_F$-scheme.

Write $\W_k = \oplus \W_i$ where $\mu^\W$ acts on $\W_i$ with weight $n_i$ and $ 0 = n_0 < n_1 < \dots,$
and for $i \geq 0,$ let $d_i = \dim_k \W_i$ and $e_i = \sum_{j=0}^i d_j.$
The condition $g_1\mu^\W(p) = \mu^\W(p)g_2$ implies that if $(g_1,g_2) \in P_{\mu}(\alpha)$ then
$g_1$ leaves $\W_0$ stable, $g_2$ leaves $\oplus_{i > 0} \W_i$ stable and $g_1,g_2$ induce
the same endomorphism of $\W_0 = \W_k/\oplus_{i> 0} \W_i.$

Note that
$$ (\wedge^{e_i+1} \W)_0 = (\otimes_{j=0}^i \wedge^{d_j} \W_j)\otimes \W_{i+1},$$
where $(\wedge^{e_i+1} \W)_0$ denotes the summand of $\wedge^{e_i+1} \W$ on which $\mu^{\wedge^{e_i+1} \W}$ acts
with weight $0.$ Hence for $i \geq 0,$ $g_1$ leaves $\oplus_{j \leq i} \W_j$ stable,
$g_2$ leaves $\oplus_{j > i} \W_j$ stable and $g_1,g_2$ induce the same endomorphism of $\W_i.$

It follows that $P_{\mu}(\alpha)_k$ is contained in $P_{\mu,k}^{\circ}\times_{M_{\mu,k}}P_{\mu,k}.$
Thus, if $P_{\mu}'$ denotes the closure of $G\hookrightarrow G\times G,$ under the embedding above,
then we have
$$P_{\mu,k}' \subset P_{\mu}(\alpha)_k \subset P_{\mu,k}^{\circ}\times_{M_{\mu,k}}P_{\mu,k}.$$
Since $P_{\mu,k}^{\circ}\times_{M_{\mu,k}}P_{\mu,k}$ is a smooth connected group scheme with the same dimension
as $P_{\mu,k}',$ the above inclusions must be equalities, which proves the lemma.
\end{proof}

\begin{prop} \label{1.1.9} Let $R$ be a reduced, Noetherian $\bar k$-algebra, $\R$ a frame
for $R,$ and $g \in G(\R_L).$ Suppose that $S_{\mu}(g) = S.$ Then there exists an
\'etale covering $R \rightarrow R'$ such that $g \in G(\R')p^{\mu}G(\R'),$ where $\R'$ is the
canonical frame for $R'$ produced in \ref{1.1.5}.
\end{prop}
\begin{proof} Let $G \hookrightarrow \GL(\W)$ and $\alpha$ be as in \ref{1.1.8}. Consider the map
\begin{equation}\label{1.1.10}
 G\times G \rightarrow \oplus_{i \geq 1} \End_{\O_F} \wedge^i \W; \quad (g_1,g_2)
\mapsto (g_1\mu^{\wedge^i\W}(p)g_2)_{i \geq 1}.
\end{equation}

Note that by Lemma \ref{1.1.8} the non-empty fibres of (\ref{1.1.10}) are torsors under the smooth group scheme $P_{\mu}(\alpha).$
More precisely, for any $\O_F$-scheme $T$ the map on $T$-valued points induced by (\ref{1.1.10}) has fibres
which are either empty or torsors under $P_{\mu}(\alpha)(T).$ Hence the pullback of (\ref{1.1.10}) by the image of any
point in $G\times G(T)$ is a $P_\mu(\alpha)$-torsor.

Let $\gamma_i =\mu^{\wedge^i \W}(p)\mu(p)^{-1}g$ for each $i \geq 1.$
Then $\gamma = (\gamma_i)_{i \geq 1}$ is an $\R$-point of $\oplus_{i \geq 1} \End_{\O_F} \wedge^i \W.$
By \ref{1.1.6} for any perfect field $\kappa \supset \bar k,$ any finite extension $L'/W(\kappa)[1/p],$
and any map of $\O_L$-algebras $\xi:\R \rightarrow \O_{L'},$
$\xi^*(\gamma)$ lifts to a point of $G\times G(\O_{L'}),$ and hence for any such $\xi$
the pullback of (\ref{1.1.10}) by $\xi^*(\gamma)$ is a $P_{\mu}(\alpha)$-torsor, and in particular, flat.
It follows from \ref{1.1.11} below, that the pullback of (\ref{1.1.10}) by $\gamma$ is a (flat) $P_{\mu}(\alpha)$-torsor.

Finally, the lemma follows, since the above torsor can be trivialized over some \'etale covering of $R.$
\end{proof}

\begin{lemma} \label{1.1.11} Let $\R$ be a $p$-adically complete and separated,
$p$-torsion free $\O_L$-algebra, such that $\R/p\R$ is reduced and Noetherian,
and $X$ a finite type $\R$-scheme.
Suppose that for any perfect field $\kappa \supset \bar k,$ any finite extension $L'/W(\kappa)[1/p],$
and any map of $\O_L$-algebras $\xi:\R \rightarrow \O_{L'},$ the fibre $X_{\xi}$ is flat over $\O_{L'}.$
Then $X$ is a flat $\R$-scheme.
\end{lemma}
\begin{proof} It suffices to check that $X$ is flat at every closed point $x \in \Spec R.$
Let $\widehat \R_x$ denote the completion of $\R$ at $x.$ By \cite{RG}, 4.2.8 $X\otimes_{\R}\widehat \R_x$
is flat, provided $\cap _{\xi} \ker \xi = 0$ where $\xi$ runs over all maps $\widehat \R_x \rightarrow \O_{L'}$
with $L'$ as in the lemma.

To see this, we first note that $\R$ is reduced. Indeed, if $\alpha \in \R$ is a nilpotent element,
then $\alpha^n = 0$ for some $n,$ so that $\alpha \in p \R,$ as $\R/p\R$ is reduced.
Since $\R$ is $p$-torsion free, we can apply the same argument to $p^{-1}\alpha,$ and we find
that $\alpha$ is infinitely divisible by $p$ in $\R.$ As $\R$ is $p$-adically separated, this is a
contradiction, unless $\alpha = 0.$

By \cite{EGA} IV 10.5.8, $\widehat \R_x[1/p]$ is a Jacobson ring.
Let $y \in \Spec \widehat \R_x[1/p]$ be a closed point, and $L_y$ the quotient of
$\widehat \R_x[1/p]$ by the corresponding maximal ideal.
Then $L_y$ is equipped with a discrete valuation, and the corresponding valuation ring $\O_{L_y}$
is a finite $\widehat \R_x$-algebra (see \cite{EGA} IV 10.5.10, and its proof). In particular,
if $\bar\kappa(x)$ is an algebraic closure of $\kappa(x),$ then $L_y$ admits an embedding
into a finite extension $L'/W(\bar\kappa(x))[1/p].$
Since any map $\xi: \R \rightarrow L'$ factors through $\O_{L'},$ we see that
$\cap _{\xi} \ker \xi = 0$ as required.
\end{proof}

\begin{cor} \label{1.1.12} Let $R$ be a Noetherian, formally smooth $\bar k$-algebra, $\R$ a frame
for $R,$ and $g \in G(\R_L).$ Suppose that $\mu$ is minuscule and that $S_{\mu}(g)$
contains the generic points of $\Spec R.$ Then there exists an
\'etale covering $R \rightarrow R'$ such that $g \in G(\R')p^{\mu}G(\R'),$ where $\R'$ is the
canonical frame for $R'$ produced in \ref{1.1.5}.
\end{cor}
\begin{proof} Since $S_{\mu}(g)$ contains the generic points of $\Spec R,$ and $\mu$ is minuscule,
we have
$$ S_{\mu}(g) = S_{\preceq \mu}(g) = \Spec R$$
by Lemma \ref{1.1.3}, and the corollary follows from \ref{1.1.9}.
\end{proof}

\subsection{The affine Grassmannian in mixed characteristic}

\begin{para} Let $\R$ be a $p$-torsion free, $p$-adically complete and separated $\O_L$-algebra.
Let  $X(\R) = X_G(\R)$ denote the set of isomorphism classes
of pairs $(\mathcal T, \tau)$ where $\mathcal T$ is a $G$-torsor over $\Spec \R,$
and $\tau$ is a trivialization of $\mathcal T$ over $\Spec \R_L.$

Let $S$ be a flat $p$-adic formal scheme over $\O_L,$ and let $S_0$ be the reduced subscheme of $S.$
An \'etale morphism $U_0 \rightarrow S_0$ lifts canonically to a formally \'etale morphism of $p$-adic formal schemes
$U \rightarrow S.$ We call such a morphism a {\em formal \'etale neighborhood} of $S.$ We call such a morphism
a covering if $U_0$ is a covering of $S_0.$ We say that $U$ is a  {\em formal affine \'etale neighborhood} if in addition $U$ is formal affine (or equivalently $U_0$ is affine).

In particular, $X_G:\Spf \R \mapsto X(\R)$ defines a functor on formal affine \'etale neighborhods of $S.$
Equivalently, we may view $X_G$ as a functor on affine \'etale neighborhoods of $S_0.$

Given a section $(\mathcal T, \tau)$ of $X_G(S)$ there is a formal \'etale covering
$\Spf \R \rightarrow S$ over which $\mathcal T$ becomes trivial.
To $g \in G(\R_L),$ we can associate the trivial $G$-torsor over $\Spec \R$ given by $G$ itself,
equipped with the trivialization over $\Spec \R_L$ corresponding to left multiplication by $g.$
Two elements $g,g' \in G(\R_L)$ give rise to the same torsor with trivialization over
$\Spec \R_L$ if and only if they have the same image in $G(\R_L)/G(\R).$
The set of elements of $X_G(\Spf \R)$ where the underlying $G$-torsor over
$\Spec \R$ is trivial is in natural bijection with $G(\R_L)/G(\R).$ Thus, the functor $X_G$ is an analogue
of the affine Grassmannian in mixed characteristic.
\footnote{This definition works well for our purposes, but has the aethstetic disadvantage that it depends on
$\R$ and not just on $R = \R/p\R.$ Haboush \cite{Haboush} (see also Kreidl \cite{Kreidl}  and Lusztig \cite{Lusztig}) has proposed
an approach to the affine Grassmannian in mixed characterstic which uses Witt vectors and the Greenberg
functor, and does not depend on the choice of lifting. However this works well only when $R$ is perfect.
Since perfect rings are typically not Noetherian many of our commutative algebra arguments would break down in this setting.
}

We will often reduce questions about $G$-bundles to questions about vector bundles.
For this we will need the following
\end{para}

\begin{lemma}\label{tannakaoverdvr} Let $Y$ be a flat $\O_L$-scheme.
Let $\mathcal F$ denote the
category of exact, faithful tensor functors from representations of $G$
on finite free $\O_L$-modules to vector bundles on $Y.$

If $P$ is a $G$-bundle on $Y,$
and $V$ is a representation of $G$ on a finite free $\O_L$ module, write
$F_P(V) = G\backslash (P\times V).$ Then $P \mapsto F_P$ is an equivalence between the category of
$G$-bundles on $Y,$ and the category $\mathcal F.$
\end{lemma}
\begin{proof} See \cite{Broshi} Thm~2.1.5.5 (cf.~also \cite{Nori}).
\end{proof}

\begin{lemma}\label{affgrassdefn} The functor $X_G$ extends to a sheaf (again denoted $X_G$)
on the \'etale topology of $S_0.$
\end{lemma}
\begin{proof} We extend $X_G$ to a presheaf $X_G^-$ on the \'etale topology of $S_0$ by
setting $X_G^-(U_0) = \ilim X_G(V_0)$ where $V_0$ runs over affine
\'etale neighborhoods of $V_0 \rightarrow U_0,$
(cf.~ \cite{EGA} \S 0, 3.2) and we let $X_G^+$ denote the sheafification of $X_G^-.$
Note we do not claim that $X_G^-$ is a sheaf,
but only that its values agree with those of $X_G^+$ on {\em affine} \'etale neighborhoods.

We have to show that for any formal affine \'etale neighborhoods $\Spf \R \rightarrow S,$
$X_G(\R) = X_G^+(\R).$
By definition, an element of $X_G^+(\R)$ is defined by giving a collection $\{\Spf \R_i\}_i$ of formal affine \'etale neighborhoods of $\Spf \R,$ whose union is a covering of $\Spf \R,$ an element
$(\mathcal T_i, \tau_i)$ of $X_G(\R_i)$ for each $i,$ and isomorphisms
$(\mathcal T_i, \tau_i) \iso (\mathcal T_j, \tau_j)$ over $\Spec \R_{ij}$ satisfying the cocycle
condition. Here $\Spf \R_{i,j} = \Spf \R_i\times_{\Spf \R}\Spf \R_j.$ We have to show that any
such collection of data arises from an element $(\mathcal T,\tau)$ in $X_G(\R),$ which is
unique up to a unique isomorphism.

By Lemma \ref{tannakaoverdvr} it suffices to prove the analogous statement for vector bundles of some fixed rank $d.$
Thus let $\{(V_i,\tau_i)\}_i$ be a collection consisting of vector bundles $V_i$ of rank $d,$ over
$\Spec \R_i$ together with trivializations $\tau_i$ over $\Spec \R_{i,L}.$
Suppose we are given isomorphisms
$\{(V_i,\tau_i) \iso (V_j,\tau_j)\}_{i,j}$ over $\Spec \R_{i,j,L}$ for all $i,j$ satisfying the cocycle condition.
We have to show this data arises from a vector bundle $V$ over $\Spec \R$
together with a trivialization over $\Spec \R_L,$ determined up to unique isomorphism.

By \'etale descent, for $n \geq 0$ this data gives rise to a uniquely determined
vector bundle $V_n$ on $\Spec \R/p^n\R,$ and hence
to a vector bundle on $\Spec \R.$
To construct the trivialization $\tau,$ we may first assume the above covering consists of
finitely many formal affine \'etale neighborhoods, since $\Spf \R$ is quasi-compact.
Now choose a sufficiently large integer $n$ such that
for each $i,$ $p^n\tau_i$ and $p^n\tau_i^{-1}$ induce maps $V_i \rightarrow\R_i^d$
and $\R_i^d \rightarrow V_i$ whose
composite is multiplication by $p^{2n}.$ As above, by \'etale descent these maps give rise
to maps $\R^d \rightarrow V$ and  $V  \rightarrow \R^d$ whose composite is multiplication by $p^{2n}.$
Inverting $p$ and dividing the resulting maps by $p^n$ produces the required trivialization $\tau.$
\end{proof}

\begin{para} Now suppose that $S = \Spf \R$ is formal affine, and locally Noetherian.
We will give a description of $X_G(\R)$ using the \'etale topology on $\Spec \R,$
which will be useful for computations.

Let $j: \Spec \R_L \hookrightarrow \Spec \R$ and $i: \Spec \R/p\R \rightarrow \Spec \R$ denote the inclusions. We again
write $X_G$ for the \'etale sheaf $i_*X_G,$ on $\Spec \R.$

Let $U$ be an \'etale neighborhood of $\Spec \R.$
Using the fact that a $G$-torsor over $\Spec \R$ is \'etale locally trivial one sees that $(j_*G/G)(U)$  is in bijection with the set of isomorphism class of pairs consisting of a $G$-torsor over $U,$ equipped with a trivialization over
$U\otimes_{\O_L}L.$ Thus, we have a natural map of \'etale sheaves $j_*G/G \rightarrow X_G.$
\end{para}

\begin{lemma}\label{affgrasspresent} The map $j_*G/G \rightarrow X_G$ is an isomorphism.
\end{lemma}
\begin{proof} We first consider the case $G = \GL_n.$
Let $U = \Spec \cR'$ be an \'etale neighborhood of $S,$
$U_0 = U\otimes_{\Z}\Z/p\Z,$ and $\widehat U = \Spf \R'$ the $p$-adic completion of $U.$
Let $\tilde U = \Spec \tilde \cR'$ be the localization of $U$ along $U_0,$ so $\tilde \cR'$ is obtained from $\cR'$
by inverting all elements which map to a unit in $\cR'/p\cR'.$ Note that any maximal ideal of $\tilde \cR'$ contains $(p)$ so that $\R',$ which is the $p$-adic completion of $\tilde \cR',$ is a faithfully flat
$\tilde \cR'$-algebra.

Since $X_G$ is a sheaf, it suffices to show that for any $\cR'$ as above,
the map $\GL_n(\tilde \cR'_L)/\GL_n(\tilde \cR') \rightarrow \GL_n(\R'_L)/\GL_n(\R')$ is a bijection.
The injectivity follows from the fact that $\R'$ is faithfully flat over $\tilde \cR',$ which implies that
$\tilde \cR'_L/\tilde \cR'$ injects into $\R'_L/\R'.$ For the surjectivity, suppose that $g \in \GL_n(\R'_L).$
Let $s$ be an integer such that $g, g^{-1} \in M_n(p^{-s}\R').$
For any $m > 0$ there exists $h \in M_n(p^{-s}\tilde \cR')$ such that $g - h = p^m\delta$ for some
$\delta \in M_n(\R').$ For $m$ sufficiently large $h \in \GL_n(\tilde \cR')$ and $1+p^mh^{-1}\delta \in \GL_n(\R').$
As $g =h(1+p^mh^{-1}\delta),$ this proves the surjectivity.

Now suppose that $G$ is arbitrary, and let $P$ be a $G$-bundle over $\Spec \R'$
equipped with a trivialization over $\Spec \R'_L.$ Then $P$ gives rise to an exact, faithful
tensor functor $F_P$ which associates to each $\O_L$-representation $V$ of $G$
the vector bundle $F_P(V) = G\backslash V\times P,$ together with an isomorphism
$\tau_V:V\otimes \R'_L \iso F_P(V)\otimes_{\O_L}L.$ By the case of vector bundles proved above,
($F_P(V),\tau_V)$ arises from a pair $(\tilde F_P(V), \tilde \tau_V)$ consisting of a vector bundle
$\tilde F_P(V)$ over $\Spec \cR'$ together with an isomorphism
$\tilde\tau_V: V\otimes \cR'_L \iso \tilde F_P(V)\otimes_{\O_L}L,$ and this pair is unique up to canonical isomorphism. In particular, $\tilde F_P(V)$ is a faithful tensor functor.
Moreover, $\tilde F_P(V)$ is exact: over $(p)$ this follows from the fact that $\R'$ is a faithful
$\tilde \cR'$ algebra, and after inverting $p$ it is forced by the existence of the isomorphisms $\tilde \tau_V.$ Using Lemma \ref{tannakaoverdvr}, we obtain the required $G$-bundle over $\Spec \cR'$ equipped with a trivialization over $\Spec \cR'_L.$
\end{proof}

\begin{para}
The following lemma, in the case when $R$ is a Dedekind domain, shows that $X_G$ satisfies an extension property which is analoguous to
of the valuative criterion for properness.
\end{para}

\begin{lemma}\label{valuative} Suppose that $R = \R/p\R$ is a Noetherian, formally smooth domain over $\bar k.$
Let $f \in \R\backslash p\R,$ and $\widehat \R_f$ the $p$-adic completion of $\R_f = \R[f^{-1}].$
Denote by $r_{G,f}$ the natural functor from the category of $G$-torsors on $\Spec \R$ equipped with a
trivialization over $\Spec \R_L,$ to the category of $G$-torsors on
$\Spec \widehat \R_f$ equipped with a trivialization over $\Spec \widehat \R_{f,L}.$

Then
\begin{enumerate}
\item $r_{G,f}$ is fully faithful, and an equivalence if $R$ is a Dedekind domain.
In particular, the natural map
$$X_G(\R) \rightarrow X_G(\widehat \R_f)$$
is an injection, and a bijection if $R$ is a Dedekind domain.
\item If $M \subset G$ is a reductive, closed $\O_F$-subgroup, the diagram
$$\xymatrix{
X_M(\R) \ar[r]\ar[d] & X_M(\widehat \R_f) \ar[d] \\
X_G(\R) \ar[r] & X_G(\widehat \R_f)}
$$
is Cartesian.
\end{enumerate}
\end{lemma}
\begin{proof}  We first prove that the functor is fully faithful. By Lemma \ref{tannakaoverdvr}
it suffices to show this for vector bundles, and for this it is enough to check that
$\widehat \R_f \cap \R_L = \R.$ Let $\tilde \R_f$ be the localization of $\R_f$ along $(p).$
Then $\widehat \R_f \cap \tilde \R_{f,L} = \tilde \R_f,$ since $\widehat \R_f$ is a fully faithful
$\tilde \R_f$-algebra, and $\tilde \R_f \cap \R_L = \R.$

Now suppose that $R$ is a Dedekind domain
By Lemma \ref{affgrasspresent} to show essential surjectivity, it suffices to show that a $G$-bundle
$P$ over $\Spec \R_f$ equipped with a trivialization over $\Spec \R_{f,L}$ extends uniquely to a
$G$-bundle over $\Spec \R.$
Using the trivialization, we may extend $P$ to a $G$-bundle over the complement of a set of codimension $2$
in $\Spec \R,$ equipped with a trivialization over $\Spec \R_L.$
By \cite{CS}, Thm.~6.13, since $G$ is reductive over $\O_L,$ any such bundle extends to a $G$-bundle over
$\Spec \R.$ This proves (1).

To prove (2), it suffices, by (1), to show that if $(\mathcal T_M,\tau) \in X_M(\widehat \R_f)$ lifts to an element of $X_G(\widehat \R)$
then it lifts to an element of $X_M(\widehat \R).$ Using the full faithfulness in (1) again, it suffices to prove this
with $R$ replaced by an \'etale covering. Thus we may assume that $(\mathcal T_M,\tau)$ is given by an element in $g \in G(\widehat \R_L).$
By Lemma \ref{affgrasspresent}, $\mathcal T_M$ arises from an $M$-bundle
on $\Spec \R_f,$ and we extend it to an $M$-bundle $\mathcal T_M'$ on $U : = \Spec \R_f \cup \Spec \R_L,$ equipped with a trivialization
over $\Spec \R_L.$ Since $\mathcal T_M'$ arises from $g,$ the $G$-torsor induced by $\mathcal T_M'$ is trivial. Thus it corresponds to
a section in $G/M(U).$ The complement of $U$ in $\Spec \R$ has codimension $\geq 2.$
Since $M$ is reductive, $G/M$ is a smooth, affine scheme. It follows that any section in $G/M(U)$ extends to $\Spec \R.$
 This shows that $\mathcal T'_M$ extends to an $M$-bundle of $\Spec \R,$ and proves (2).
\end{proof}

\begin{para} Now suppose that $\R$ has the structure of a frame for $R = \R/p\R.$
If $\Spf \R' \rightarrow \Spf \R$ is a formal affine \'etale neighborhood, then as remarked in \ref{1.1.5},
$\Spf \R'$ has a canonical structure of frame for $R' = \R'/p\R'.$
Thus given any $s \in \Spec R',$ and $g \in X_G(\R'),$ we may consider the induced element
$s(g) \in X_G(W(\bar\kappa(s))) = G(W(\bar\kappa(s))[1/p])/G(W(\bar\kappa(s))).$
\end{para}

\begin{lemma} Let $R$ be a formally smooth, Noetherian $\bar k$-algebra, and $\R$ a frame for $R.$
We regard $\pi_1(G)$ as a constant \'etale sheaf on $\Spec R$ with value $\pi_1(G).$
Then there is a canonical map $w_G: X_G \rightarrow \pi_1(G)$ of \'etale sheaves on $\Spec R,$
such that for any \'etale covering $\Spec R' \rightarrow \Spec R,$ $s \in \Spec R',$
and $g \in X_G(\R')$ we have
$$ [\mu_{s(g)}] = w_G(g)_{s} \in \pi_1(G).$$
\end{lemma}
\begin{proof} This follows immediately from Lemma \ref{1.1.3}.
\end{proof}

\subsection{Affine Deligne-Lusztig varieties}\label{affdlvar}

\begin{para} Let $\R$ be a $p$-torsion free, $p$-adically complete and separated $\O_L$-algebra.
Recall that for $g \in G(\R_L)$ and $\mu \in X_*(T)$ dominant we defined
$$ S_{\mu}(g) = \{s \in \Spec R: s(g) \in G(W(\bar \kappa(s)))p^{\mu} G(W(\bar \kappa(s)))\}
$$
where $\bar\kappa(s)$ denotes an algebraic closure of
$\kappa(s).$
Note that the condition on $g$ in the definition of $S_{\mu}(g)$ depends only on
the image of $g$ in $G(\R_L)/G(\R).$ We may therefore define $S_{\mu}(g)$
and $S_{\preceq \mu}(g)$ in the same way for any $g \in G(\R_L)/G(\R).$

Now let $R$ be a $\bar k$-algebra, $S = \Spec R$ and $\R$ a frame for $R.$
For $b \in G(L)$ we set
$$ X_{\preceq \mu}(b)(\R) = \{g \in X_G(\R): S_{\preceq \mu}(g^{-1}b\sigma(g)) = S\},$$
and we define $X_{\mu}(b)(\R)$ in an analogous way,
replacing $S_{\preceq \mu}$ by $S_{\mu}.$
If $\Spf \R' \rightarrow \Spf \R$ is a formal affine \'etale neighborhood, then as remarked above,
$\Spf \R'$ has a canonical structure of frame for $R' = \R'/p\R'.$ Thus we may consider
$X_{\preceq\mu}(b)(\R')$ (resp. $X_{\mu}(b)(\R')$). Note that the above definition probably needs to be refined if one wants to obtain a good notion of non-reduced structure on affine Deligne-Lusztig sets. However, for our study of connected components this is not relevant.

For $g_0\in G(L)$ we have natural bijections $X_{\preceq \mu}(b)(\R)\rightarrow X_{\preceq \mu}(g_0^{-1}b\sigma(g_0))(\R)$ with $g\mapsto g_0^{-1}g$. Therefore, all of the following notions for these sets and in particular the set of connected components of $X_{\preceq \mu}(b)$ only depend on the $\sigma$-conjugacy class of $b$.

In the analogous situation, when $F$ has characteristic $p,$ any $\bar k$-algebra $R$ admits the canonical frame
$R\lps t \rps.$ Thus $X_{\mu}(b)$ can be regarded as a functor on $\bar k$-algebras, by setting $X_{\mu}(b)(R)$ to be the set $X_{\mu}(b)(R\lps t \rps)$ defined as above. In fact, in this setting, $X_{\mu}(b)$ is a scheme in characteristic $p.$
Although one would like to have a similar interpretation in mixed characteristic there is no canonical frame, and we do not know of any way to formalize this heuristic.

We will sometimes write simply $X_{\preceq \mu}(b)$ for $X_{\preceq\mu}(b)(W(\bar k)).$
When we want to make the group $G$ explicit we will write $X_{\preceq \mu}^G(b)$ for $X_{\preceq \mu}(b).$
\end{para}
\begin{lemma} The functors $X_{\preceq\mu}(b)$ and $X_{\mu}(b)$ are subsheaves of $X_G$ in the \'etale topology of $\Spec R.$
\end{lemma}
\begin{proof} This follows from Lemma \ref{affgrassdefn} together with the fact that the conditions defining $X_{\preceq\mu}(b)$ and $X_{\mu}(b)$ are local for the \'etale topology on $\Spec R.$
\end{proof}

\begin{lemma}\label{valuativeII} Suppose that $R = \R/p\R$ is Noetherian and formally smooth over $\bar k.$
Let $f \in \R\backslash p\R,$ and $\widehat \R_f$ the $p$-adic completion of $\R_f.$ Then the diagram
$$\xymatrix{
X_{\preceq \mu}(b)(\R) \ar[r]\ar[d] & X_{\preceq \mu}(b)(\widehat \R_f) \ar[d] \\
X_G(\R) \ar[r] & X_G(\widehat \R_f)}
$$ is Cartesian, and similarly with $X_{\mu}$ in place of $X_{\preceq \mu}$ if $\mu$ is minuscule. In particular,
\begin{enumerate}
\item The natural map
$X_{\preceq \mu}(b)(\R) \rightarrow X_{\preceq\mu}(b)(\widehat \R_f)$ is injective, and is bijective if $R$ is a Dedekind domain.
\item If $\mu$ is minuscule the natural map
$X_{\mu}(b)(\R) \rightarrow X_{\mu}(b)(\widehat \R_f)$ is injective, and is bijective if $R$ is a Dedekind domain.
\item If $M \subset G$ is a closed, reductive $\O_F$-subgroup with $b \in M(L),$ then the diagram
$$\xymatrix{
X^M_{\preceq \mu}(b)(\R) \ar[r]\ar[d] & X^M_{\preceq \mu}(b)(\widehat \R_f) \ar[d] \\
X^G_{\preceq \mu}(b)(\R) \ar[r] & X^G_{\preceq \mu}(b)(\widehat \R_f)}
$$
is Cartesian, and similarly with $X_{\mu}$ in place of $X_{\preceq \mu}$ if $\mu$ is minuscule.
\end{enumerate}
\end{lemma}
\begin{proof} Let $g \in X_{\preceq \mu}(\widehat \R_f),$ and suppose that
$g$ arises from an element $\tilde g \in X_G(\R).$ By Lemma \ref{1.1.3},
the condition $S_{\preceq \mu}(g^{-1}b\sigma(g)) = \Spec R[1/f]$ implies
$S_{\preceq \mu}(\tilde g^{-1}b\sigma(\tilde g)) = S,$ so $\tilde g \in X_{\preceq\mu}(\R).$
Similarly, if $\mu$ is minuscule and $g \in X_{\mu}(\widehat \R_f),$ then
$\tilde g \in X_{\mu}(\R).$ It follows that the first diagram in the lemma is Cartesian.
This implies the other claims in the lemma, using Lemma \ref{valuative}.
\end{proof}
\begin{para} Let $\DD$ denote the pro-torus with character group $\Q.$
Recall the Newton  cocharacter
$$\nu = \nu_b: \DD \rightarrow G$$ defined by Kottwitz \cite{Kottwitz1}, 4.2.
If $G = GL(\W)$ for an $F$-vector space $\W,$ then $\nu$ is the cocharacter which induces the slope decomposition of $b\sigma$ acting on $\W.$
In general $\nu$ is determined by requiring that it be functorial in the group $G.$
We denote by $M_b \subset G$ the centralizer of $\nu_b.$
A $\sigma$-conjugacy class is called basic if the associated Newton cocharacter is central in $G$. Let $\nu_{\dom} \in X_*(T)_{\Q}^{\Gamma}$ be the dominant cocharacter conjugate to the Newton cocharacter of $b.$

The group $\Gamma$ acts on $X_*(T)$ through a finite quotient, and we denote by
$$\bar \mu
= [\Gamma:\Gamma_{\mu}]^{-1} \sum_{\tau \in \Gamma/\Gamma_{\mu}} \tau(\mu) \in X_*(T)_{\Q}$$
the average of the $\Gamma$-conjugates of $\mu.$
As mentioned in the introduction,
the set $X_{\mu}(b)$ is non-empty if and only if $[b] \in B(G,\mu).$ That is, $\kappa_G(b)= [\mu]$ in $\pi_1(G)_{\Gamma},$
and $\bar \mu-\nu_{\dom}$
is a linear combination of positive coroots with non-negative rational coefficients.
We assume from now on that this condition holds.

For any $\bar b \in G^{\ad}(L),$ we define an $F$-group $J_{\bar b}$ by setting $$ J_{\bar b}(R) = J^G_{\bar b}(R) : =  \{g \in G(R\otimes_FL): \sigma(g) = \bar b^{-1}g\bar b \},$$
for $R$ an $F$-algebra. There is an inclusion $J_{\bar b} \subset G,$ defined over $L,$
which is given on $R$-points ($R$ an $L$-algebra) by the natural map $G(R\otimes_FL) \rightarrow G(R),$ and which identifies $J_{\bar b}$ with the preimage of $M_{\bar b}$ in $G.$
The group $J_{\bar b}$ is an inner form of $M_{\bar b}$ \cite{Kottwitz2}, 3.3, \cite{RZ}, 1.12. 

If $b \in G(L)$ we write $J_b = J_{\bar b}$ where $\bar b$ denotes the image of $b$ in $G^{\ad}(L).$
Then $J_b(F)$ acts naturally on $X_{\preceq \mu}(b)$ and $X_{\mu}(b).$
\end{para}

\begin{para}\label{conn_comp_ADLV}
Let $g_0, g_1 \in X_{\preceq \mu}(b)(W(\bar k)),$ and $R$ a smooth $\bar k$-algebra with connected spectrum,
equipped with a frame $\R.$ We say that $g_0$ is connected to $g_1$ via $\R$
if there exists
$g \in X_{\preceq \mu}(b)(\R)$ and $s_0,s_1 \in (\Spec R)(\bar k)$ such that
$s_0(g) = g_0$ and $s_1(g) = g_1.$ We denote by $\sim$ the smallest equivalence relation on
$X_{\preceq \mu}(b)(W(\bar k))$ such that $g_0 \sim g_1$ if $g_0$ is connected to $g_1$ via some $\R$ as above,
and we write $\pi_0(X_{\preceq \mu}(b))$ for the set of equivalence classes under $\sim.$

We could have defined a notion of connected components without assuming that $R$ is smooth.
However the stronger notion of connectedness is useful in the applications in \cite{Ki4} and, happily, this
condition is also convenient in several of our arguments. On the other hand, we conjecture that the two definitions
of connected components are equivalent. This follows {\it a posteriori} from our main result,
when $\mu$ is minuscule $G^{\ad}$ is simple, $(\mu,b)$ is Hodge-Newton indecomposable and $G^{\der}$ is
simply connected (so that $\pi_1(G)$ has no torsion). To see this one uses the first part of the
proof of Lemma \ref{1.1.3} which shows (without assuming $R$ formally smooth) that
$s \mapsto [\mu_{s(g)}] \in \pi_1(G)\otimes_{\mathbb Z}\mathbb Q$ is locally constant on $\Spec R.$
We believe that all of Lemma \ref{1.1.3} remain true without assuming $R$ formally smooth, in
which case the two notions of connected component would agree without assuming $G^{\der}$ simply connected.

The natural action of $J_b(F)$ on $X_{\preceq \mu}(b)$ clearly induces an action on $\pi_0(X_{\preceq \mu}(b)).$
Note that we also have an action of $J_b(F)$ on $\pi_1(G)$  by left multiplication via
$J_b(F) \overset {w_{J_b}}\rightarrow \pi_1(J_b) \rightarrow \pi_1(G).$
\end{para}
\begin{lemma}\label{constructionofmap}
\begin{enumerate}
\item The homomorphism $w_G:G(L)\rightarrow \pi_1(G)$ induces a well-defined map
$w_G:\pi_0(X_{\preceq\mu}(b))\rightarrow \pi_1(G),$ which is compatible with the action
of $J_b(F).$
\item Let $c_{b,\mu}$ be as in Theorem \ref{thmzshk}. Then the image of the map defined above is contained in $c_{b,\mu}\pi_1(G)^{\Gamma}$.
\end{enumerate}
\end{lemma}
\begin{proof}
The first assertion of (1) follows from Lemma \ref{1.1.3}, where the claim regarding
the action of $J_b(F)$ is clear.

For (2) let $g\in X_{\preceq\mu}(b)$. As $K$ is in the kernel of $w_G$, this implies $w_G(g^{-1}b\sigma(g)) = [\mu] \in \pi_1(G)$. Hence $-w_G(g)+\sigma(w_G(g))=[\mu]-w_G(b)$.
By definition of $c_{b,\mu}$ this implies the claim.
\end{proof}

\subsection{Reduction to adjoint groups}\label{secred}

 We continue to use the notation above. In particular, $\R$ is a frame for $R = \R/p\R,$
and we continue to assume that $[b] \in B(G,\mu).$

\begin{lemma}\label{cartesianI} Let $G \rightarrow G'$ be a morphism of reductive
groups over $\O_L$ which takes $Z_G$ to $Z_{G'}$ and induces an isomorphism on adjoint groups.
Suppose that $R$ is Noetherian and formally smooth over $\bar k.$
Then the diagram of \'etale sheaves on $\Spec R$
$$\begin{CD}
  X_G  @>>> X_{G'} \\
^{w_G}@VVV^{w_{G'}} @VVV      \\
\pi_1(G)   @>>> \pi_1(G')
\end{CD}$$
is Cartesian.
\end{lemma}
\begin{proof} Using Lemma \ref{affgrasspresent} we identify the top line of the diagram with the map
$j_*G/G \rightarrow j_*G'/G'.$
Let $Z = \ker(G \rightarrow G')$ and let $G''$ be the pushout of $G$ by an embedding $Z \hookrightarrow T$
where $T$ is a $\O_L$-torus. Then we have maps $ G \rightarrow G'' \rightarrow G',$ where the first map is an
embedding, and the second map has kernel a torus. Hence it suffices to prove the lemma in the two
cases when $G \rightarrow G'$ is faithfully flat with $Z$ a torus, or an embedding.

For the first case, we begin by computing the fibre of this map over the identity.
Let $g$ be a local section in this fibre. Since any $G$-torsor is \'etale
locally trivial, $g$ admits a local lift to a section $\tilde g$ of $j_*G.$ Since the image of $\tilde g$ is trivial
in $j_*G'/G'$ for any point $s \in \Spec R,$ we obtain that $\mu_{s(\tilde g)}$ is in $X_*(Z).$
Hence, this cocharacter is a locally constant function on $\Spec R$ by Lemma \ref{1.1.3}.
It follows by \ref{1.1.9} that $\tilde g$ is \'etale locally of the form $p^{\mu_g}h$ with $\mu_g \in X_*(Z)$ and
$h$ a section of $G.$ Hence $g$ is in the image of
$$ X_*(Z) \rightarrow j_*G/G \quad \mu \mapsto p^{\mu}.$$
This map is injective (for example by the pointwise Cartan decomposition) and is equal to the fibre of
$j_*G/G \rightarrow j_*G'/G'$ over the identity.
In particular, we see that the non-empty fibres of both the horizontal maps in the diagram are $X_*(Z)$-torsors.

Since $Z$ is a torus the map $\pi_1(G) \rightarrow \pi_1(G')$ is surjective.
Hence it suffices to show that a local section of $j_*G'$ lifts to $j_*G.$
Note that $R^1j_*\GG_m = 0.$ Indeed if $\mathfrak p \supset (p)$ is a prime of $\R$, then
a line bundle $\mathcal L$ on $\Spec \R_{\mathfrak p}[1/p]$ extends to $\Spec \R_{\mathfrak p}:$
Our assumptions imply that $\R_{\mathfrak p}$ is a regular local ring. Thus, we may
first extend $\mathcal L$ as a coherent sheaf, and then take the determinant of the extension.
Hence $R^1j_*Z = 0,$ which shows that $j_*G \rightarrow j_*G'$ is surjective.

For the case of an embedding, we have to show that if $g$ is a local section of $j_*G'/G'$ whose
image in $\pi_1(G')$ is in $\pi_1(G),$ then $g$ lifts locally to $j_*G/G.$ We may assume that $g$
lifts to a section $\tilde g$ of $j_*G'.$ Let $T' \subset G'$ be a maximal
(necessarily split) torus, and $T \subset G$ its preimage. Using that $R^1j_*\GG_m = 0$ we have
$j_*(G'/G) = j_*(T'/T) = j_*T'/j_*T.$ Hence, after modifying $\tilde g$ by an element of $j_*G,$
we may assume that $\tilde g \in j_*T'.$ Since the map $j_*T/T \rightarrow j_*T'/ T'$
may be identified with $X_*(T) \rightarrow X_*(T'),$ and the cokernel of the latter map
is $X_*(G'/G),$ it follows that $\tilde g$ lifts to an element of $j_*T.$
\end{proof}
\begin{cor}\label{cartesianII}  Let $Z \subset Z_G$ be a closed $\O_L$-subgroup, and $G' = G/Z.$
Write $T' = T/Z,$ $b' \in G'(\O_L)$ and $\mu' \in X_*(T')$ for the elements induced by $b$ and $\mu.$
Suppose that $R$ is Noetherian and formally smooth over $\bar k.$ Then the diagrams of \'etale sheaves on $\Spec R$
$$\begin{CD}
  X_{\mu}(b)  @>>> X_{\mu'}(b') \\
^{w_G}@VVV^{w_{G'}} @VVV      \\
c_{b,\mu}\pi_1(G)^{\Gamma}   @>>> c_{b',\mu'}\pi_1(G')^{\Gamma}
\end{CD}$$
and
$$\begin{CD}
  X_{\preceq\mu}(b)  @>>> X_{\preceq\mu'}(b') \\
^{w_G}@VVV^{w_{G'}} @VVV      \\
c_{b,\mu}\pi_1(G)^{\Gamma}   @>>> c_{b',\mu'}\pi_1(G')^{\Gamma}
\end{CD}$$
are Cartesian.
\end{cor}
\begin{proof}  It follows from Lemma \ref{cartesianI} that the non-empty fibres of all the horizontal maps in both diagrams
are torsors under $X_*(Z)^{\Gamma}.$ Hence it suffices to show that a local section $g$ of $X_{\mu'}(b')$
(resp. $ X_{\preceq\mu'}(b')$) whose image in $c_{b',\mu'}\pi_1(G')^{\Gamma}$ lifts to $c_{b,\mu}\pi_1(G)^{\Gamma},$
lifts \'etale locally to  $X_{\mu}(b)$ (resp. $ X_{\preceq\mu}(b)$).

By Lemma \ref{cartesianI} $g$ lifts to a local section $\tilde g$ of $X_G.$
By assumption, there exists $\chi \in X_*(Z),$
such that $w_G(\tilde g)+\chi \in c_{b,\mu}\pi_1(G)^{\Gamma}.$
Hence after replacing $\tilde g$ by $\tilde g p^{\chi},$
we may assume $w_G(\tilde g) \in c_{b,\mu}\pi_1(G)^{\Gamma}.$ To check that
$\tilde g \in X_{\mu}(b)$ (resp. $ X_{\preceq\mu}(b)$), it suffices to pull back to geometric points, and consider
the special case $\R = W(\bar\kappa)[1/p]$ for an algebraically closed field $\bar\kappa.$ In this case,
we have $\mu_{\tilde g^{-1}b\sigma(\tilde g)} + \alpha = \mu$
(resp. $\mu_{\tilde g^{-1}b\sigma(\tilde g)} + \alpha \preceq \mu$) for some $\alpha \in X_*(Z).$
Since $w_G(\tilde g) \in c_{b,\mu}\pi_1(G)^{\Gamma},$ the image of $\alpha$ in $\pi_1(G)$ is trivial,
and $\alpha = 0.$
\end{proof}

\begin{cor}\label{cartesianIII} With the notation above, the diagrams
$$\begin{CD}
  \pi_0(X_{\mu}(b))  @>>> \pi_0(X_{\mu'}(b')) \\
^{w_G}@VVV^{w_{G'}} @VVV      \\
c_{b,\mu}\pi_1(G)^{\Gamma}   @>>> c_{b',\mu'}\pi_1(G')^{\Gamma}
\end{CD}$$
and
$$\begin{CD}
  \pi_0(X_{\preceq\mu}(b))  @>>> \pi_0(X_{\preceq\mu'}(b')) \\
^{w_G}@VVV^{w_{G'}} @VVV      \\
c_{b,\mu}\pi_1(G)^{\Gamma}   @>>> c_{b',\mu'}\pi_1(G')^{\Gamma}
\end{CD}$$
are Cartesian.
\end{cor}
\begin{proof} The vertical maps are given by Lemma \ref{constructionofmap}, which also implies that
$Z(F) \subset J_b(F)$ acts on the fibres of the top horizontal maps via $Z(F) \rightarrow X_*(Z)^{\Gamma}.$
Thus the non-empty fibres of all the horizontal maps are $X_*(Z)^{\Gamma}$-torsors. That the diagrams are Cartesian
now follows from Corollary \ref{cartesianII}.
\end{proof}


\subsection{Hodge-Newton indecomposability}\label{secpf1}

\begin{para} Let $b \in G(L),$ and $M_b \subset G$ the centralizer of $\nu_b,$ as above.
\end{para}

\begin{lemma}\label{(1.2.5)}
\begin{enumerate}
\item If $b' = gb\sigma(g)^{-1}$ for $g \in G(L),$ then $\nu_{b'} = g\nu_bg^{-1}.$
\item There exists a $b'$ in the $\sigma$-conjugacy class of $b$ such that $\nu_{b'} \in X_*(T)\otimes_{\Z}\Q,$
 is dominant and $\sigma$-invariant, and $b' \in M_{b'}.$
\end{enumerate}
\end{lemma}
\begin{proof} (1) is clear from the definition of $\nu.$

Applying this with $g = b^{-1},$ we find that $\sigma(\nu_b) = \nu_{\sigma(b)} = b^{-1}\nu_bb$ is conjugate to
$\nu_b,$ so the $G(L)$-conjugacy class of $\nu$ is stable by $\sigma.$ Since $G$ is quasi-split, this implies that $\nu$ is conjugate to a dominant $\sigma$-invariant cocharacter in $X_*(T)\otimes_{\Z}\Q$ (\cite{Kottwitz4}, 1.1.3(a)), which shows
there is a $b'$ with $\nu_{b'} \in X_*(T)\otimes_{\Z}\Q$ and $\sigma$-invariant. Then $\nu_{b'} = \sigma(\nu_{b'}) = b^{\prime-1}\nu_{b'} b',$ so $b' \in M_{b'}.$
\end{proof}

\begin{para} By the Lemma, after replacing $b$ by an element in its $\sigma$-conjugacy class we may assume that
$\nu = \nu_b \in X_*(T)$ is dominant, and thus defined over $F$ (so that $M_b$ is also defined over $F$), and
that $b \in M_b(L).$ In particular $b$ is then basic as an element of $M_b(L).$ We assume that $b$ has been chosen with these properties.
\end{para}

\begin{prop}\label{prophndecomp}
Let $M \supset M_b$ be a standard Levi defined over $F.$ Assume that $\kappa_M(b)= [\mu]\in \pi_1(M)_{\Gamma}.$ Then
the natural inclusion $X_{\mu}^{M}(b)(W(\bar k))\hookrightarrow X_{\mu}^G(b)(W(\bar k))$
is a bijection, and similarly for the closed affine Deligne-Lusztig varieties. Furthermore, it induces bijections between the corresponding sets of connected components.
\end{prop}
\begin{proof}
The bijection between the two Deligne-Lusztig sets is shown in \cite{hn}, Theorem 6, i. Note that that theorem has a slightly different assumption on $M$, which is incorrect. The present assertion is the corrected statement and follows from the proof of \cite{hn}, which in turn is nothing but a variant of the original proof of Kottwitz in \cite{Kottwitz3}.

It remains to show that if $g\in X_{\preceq\mu}^G(b)(\R),$ where $\R$
is a frame for a smooth connected $\bar k$-algebra $R,$
and if $g_1=g(s_1), g_2=g(s_2)$ for two $\bar k$-valued points $s_1, s_2$ of $R,$ then the corresponding elements of $X_{\preceq\mu}^M(b)$
are in the same connected component. The strategy is to show that $g$ is induced by a connecting family in $X_{\preceq\mu}^M(b)$.
We may replace $R = \R/p\R$ by an \'etale covering, and assume that $g$ arises from an element $g \in G(\R_L).$

Let $\R_n$ denote $\R$ regarded as an $\R$-algebra via
$\R \overset {\sigma^n} \rightarrow \R.$ Let $\eta$ denote the generic point of $\Spec R,$
set $\R_{\eta,\infty} = \lim_n \R_{n,\eta},$ and let $\widehat \R_{\eta,\infty}$ be the
$p$-adic completion of $\R_{\eta,\infty}.$ Then $\widehat \R_{\eta,\infty}$ is a frame for a perfect closure
$R_{\eta,\infty}$ of $R_{\eta}.$

By the Iwasawa decomposition we have $g \in M(\widehat\R_{\eta,\infty,L})N(\widehat\R_{\eta,\infty,L})G(\widehat\R_{\eta,\infty})$.
By the (pointwise) Hodge-Newton decomposition the factor in $N$ may be assumed to be $1$.
Write $g = m_{\eta}h_{\eta}$ where $m_{\eta} \in M(\widehat\R_{\eta,\infty,L})$ and $h_{\eta} \in G(\widehat\R_{\eta,\infty}).$
Using the Cartan decomposition, and the formal smoothness of $M$ we may approximate $m_{\eta}$ by an element
of $\R_{n,\eta},$ for some $n,$ and assume that $m_{\eta} \in M(\R_{n,\eta,L})$ and $h_{\eta} \in G(\R_{n,\eta}).$

It follows that there exists an $f \in \R_n\backslash p\R_n$ such that as a section of $X_{\preceq\mu}^G(b)(\widehat \R_{n,f}),$
$g$ arises from an element $m_f \in X_{\preceq\mu}^M(b)(\widehat \R_{n,f}).$ Hence $g$ arises from an element
 $m \in X_{\preceq\mu}^M(b)(\widehat \R_n)$ by Lemma \ref{valuativeII}. This shows that $s_1$ and $s_2$ are connected via
$\R_n.$
\end{proof}

\begin{para}
We now suppose that $[b] \in B(G,\mu),$ and we continue to assume that $b \in M_b(L)$ and that $\nu_b$ is dominant.\footnote{We emphasize that one gets the correct notion of HN indecomposability only if $b$ is chosen so that $\nu_b$ is dominant.}
We say that the pair $(\mu,b)$ is indecomposable with respect to the Hodge-Newton decomposition if for all proper standard Levi subgroups $M \supset M_b$ that are defined over $F$, we have $\kappa_{M}(b)\neq \mu$ in $\pi_1(M)_{\Gamma}$. Given $G$, $\mu$, and $[b]$, we may always pass to a Levi subgroup $M$ of $G$ defined over $F$ in which $(\mu,b)$ is indecomposable. Lemma \ref{prophndecomp} shows that to describe the connected components
of affine Deligne-Lusztig varieties it is sufficient to consider pairs $(\mu,b)$ which are indecomposable with respect to the
Hodge-Newton decomposition. For a pair $(\mu,b)$ that is indecomposable with respect to the
Hodge-Newton decomposition, we say that it is irreducible with respect to the Hodge-Newton decomposition (or HN-irreducible for short) if $\kappa_{M}(b)\neq \mu$ for every proper standard Levi $M$ in $G$ containing an element $b\in[b]$ such that the $M$-dominant Newton point of $b$ is $G$-dominant.

The following theorem gives a stronger characterization of indecomposability that is used in Section \ref{seccp}.
\end{para}

\begin{thm}\label{thmzshkII}
Let $G$, $\mu$, and $b$ be as above and assume that $G^{\ad}$ is simple. Then the following conditions are equivalent:
\begin{enumerate}
\item The pair $(\mu, b)$ is HN-irreducible.
\item For any proper standard Levi subgroup $M$ of $G$, we do not have $\nu_b\leq \bar\mu$ in the positive Weyl chamber of $M$ in
$X_*(A)\otimes \mathbb{Q}$, where $A\subset T$ is the maximal split torus.
\item All the coefficients of simple coroots of $G$ in $\bar\mu-\nu_b$ are strictly positive.
\end{enumerate}
If these conditions are not satisfied then either $(\mu, b)$ is already HN-decomposable or $b$ is $\sigma$-conjugate to $p^{\mu}$ and $\mu$ is central.
\end{thm}
\begin{proof}
Conditions (2) and (3) are clearly equivalent. For any standard proper Levi subgroup $M$ with $b\in M(L)$,  we have $\kappa_M(b)-\mu=\nu_b-\bar\mu\in \pi_1(M)_{\Gamma}\otimes \mathbb{Q}$. Therefore (3) implies (1).

We now assume that (3) is not satisfied, i.e.~the coefficient of some simple coroot $\alpha^\vee_0$ vanishes.

\noindent{\bf Claim.} $(\mu,b)$ is HN-decomposable or $\nu_b=\bar\mu$.

We first show that this claim implies the last assertion of the theorem. Suppose that
 $(\mu,b)$ is HN-indecomposable, so that $\nu_b = \bar \mu.$

Since $\mu - \kappa_{M_b}(b) = \bar \mu - \nu_b = 0$ in $\pi_1(M_b)_{\Gamma}\otimes \mathbb Q,$
and $\mu = \kappa_G(b),$ it follows by  Corollary \ref{cor:invariants} below that $\kappa_{M_b}(b) = \mu.$ Hence $M_b = G,$
since we are assuming $(\mu,b)$ is HN-indecomposable.
Thus $\langle\alpha,\bar\mu\rangle=n^{-1}\sum_{i=1}^n \langle\alpha,\sigma^i\mu\rangle=0$ for every positive root $\alpha$ of $G$ and some $n$ with $\sigma^n(\mu)=\mu$. As $B$ is defined over $F$ and $\mu$ is dominant, each of the summands is non-negative. Hence all of them are zero, and $\mu$ is central.

In particular we see that $p^{\mu}\in[b]\cap T(L)\subsetneq G(L)$ with $\kappa_T(p^{\mu})=\mu$, hence $(\mu,b)$ is not HN-irreducible.

It remains to prove the claim. Let us assume that $(\mu,b)$ is HN-indecomposable, because otherwise the claim holds. We want to use induction on the distance between a simple root $\alpha$ and the Galois orbit of $\alpha_0$ in the Dynkin diagram of $G$ to show that also the coefficient of $\alpha^{\vee}$ in $\bar\mu-\nu_b$ is $0$. As $\bar\mu-\nu_b$ is $\Gamma$-invariant, our assumption on $\alpha_0$ shows that the coefficients of all $\alpha^{\vee}$ for $\alpha\in\Gamma\alpha_0$ vanish. Assume that the statement is shown for some simple root $\alpha$. Let $\Omega=\Gamma\alpha$ and let $M_{\Omega}$ be the standard Levi subgroup corresponding to the set of simple roots $\{\gamma: \text{ simple root}, \ \gamma\notin \Omega\}$. If $\alpha$ is not a simple root in $M_b$ then $M_{\Omega}\supset M_b\ni b$. As $(\mu,b)$ is HN-indecomposable, $\mu-\kappa_{M_{\Omega}}(b)=\lambda\alpha^{\vee}\in \pi_1(M_{\Omega})_{\Gamma}$ with $\lambda>0$ in contradiction to our assumption. Thus $\alpha$ is a simple root in $M_b$. As $\mu$ is dominant, this implies
\begin{equation}\label{glmunu}
\langle\alpha,\bar\mu-\nu_b\rangle= \langle\alpha,\bar\mu\rangle+0\geq 0.
\end{equation}
On the other hand \begin{equation*}
\langle\alpha,\bar\mu-\nu_b\rangle= \langle\alpha,\sum_{\beta\text{ simple}}\lambda_{\beta}\beta^{\vee}\rangle=\sum_{\beta\text{ neighbor of $\alpha$}}\lambda_{\beta}\langle\alpha,\beta^{\vee}\rangle.
\end{equation*} As all $\lambda_{\beta}$ are non-negative, this can only be non-negative if $\lambda_{\beta}=0$ for all neighbors $\beta$ of $\alpha$. This finishes the induction and shows that $\nu_b=\bar\mu$.
\end{proof}

\begin{remark}

Using Corollary \ref{cor:invariants}, as in the proof of the Lemma, we obtain the following fact.
Let $[b]\in B(G)$ and $\nu_b$ its Newton point. Let $M$ be a standard Levi subgroup with $M(L)\cap [b]\neq\emptyset$.
Then $\kappa_M$ is constant on
$$\{x\in[b]\cap M(L)\mid \nu^M_{x}=\nu_b \in \pi_1(M)\otimes \mathbb Q\}.$$
Here $\nu^M$ denotes the Newton point for an element of $M$, an $M$-dominant element of $X_*(T)_{\mathbb{Q}}$.
\end{remark}

\begin{remark} In \cite{chen}, we take the second condition in Theorem \ref{thmzshkII} as the definition of HN-irreducibility (cf. \cite{chen} definition 5.0.4).
\end{remark}

\begin{remark}\label{remspecialcase}
In the particular case of the above theorem where $b$ is $\sigma$-conjugate to $p^{\mu}$ and $\mu$ is central we have
\begin{eqnarray*}X_{\mu}(b)&=&\{g\in G(L)/ K\mid g^{-1}b\sigma(g)\in Kp^{\mu}K\}\\
&=& \{g\in G(L)/ K\mid g^{-1}\sigma(g)\in K \}\\
&=&G(F)/G(\O_F)
\end{eqnarray*} where the third equality follows from Lang's Lemma.
\end{remark}

\begin{lemma}\label{cartesianIV} Let $G$ be a reductive group over $\O_F,$ let $T$ be the centralizer
of a maximal split torus, and let $T^{\ad} = T/Z_G.$ Then the following
diagram is Cartesian with surjective vertical maps
$$\begin{CD}
 X_*(T)^{\Gamma}  @>>> X_*(T^{\ad})^{\Gamma} \\
^{w_G}@VVV^{w_{G^{\ad}}} @VVV      \\
\pi_1(G)^{\Gamma}   @>>>\pi_1(G^{\ad})^{\Gamma}.
\end{CD}$$
\end{lemma}
\begin{proof}
Let $\tilde G$ denote the simply connected cover of $G^{\ad},$ and $\tilde T$ the preimage of
$T$ in $\tilde G.$ The fibres of both horizontal maps are torsors under $X_*(Z_G),$ and the fibres of both
vertical maps are torsors under $X_*(\tilde T)^{\Gamma}.$
Using this, one sees easily that it suffices to show that the
vertical maps are surjective. Thus it remains to check that $H^1(\Gamma, X_*(\tilde T)) = 0.$

Suppose that $r$ is a non-negative integer, and consider any continuous action of $\Gamma$ on
$\mathbb Z^r,$ which permutes the standard basis vectors. We claim that $H^1(\Gamma, \mathbb Z^r) = 0.$
It suffices to consider the case when $\Gamma$ permutes the basis vectors transitively. If
$\Gamma'$ is the stabilizer of one of the basis vectors, then $\mathbb Z^r$ can be identified
with $\Ind_{\Gamma'}^{\Gamma} \mathbb Z,$ and claim follows since $H^1(\Gamma', \mathbb Z) = 0.$

Applying this to $X_*(\tilde T)$ with its basis of simple coroots proves the lemma.
\end{proof}

\begin{cor}\label{cor:invariants} Let $M \subset G$ be a standard Levi. Then
\begin{enumerate}
\item The map $\pi_1(M)^{\Gamma} \rightarrow \pi_1(G)^{\Gamma}$ is surjective,
and its kernel is spanned by the sum of $\Gamma$-orbits
of coroots of $G.$
\item $\ker(\pi_1(M)_{\Gamma} \rightarrow \pi_1(G)_{\Gamma})$ is torsion free.
\end{enumerate}
\end{cor}
\begin{proof} The first claim in (1) follows from Lemma \ref{cartesianIV}, and (2) then follows
by the snake Lemma. To see the second claim in (1), let $\tilde T$ be as in Lemma \ref{cartesianIV}, and
let $\tilde T_M \subset \tilde T$ be the analogous torus for $M$ in place of $G.$ Then the kernel
of the map in (1) is $(X_*(\tilde T)/X_*(\tilde T_M))^\Gamma.$ By what we saw in Lemma \ref{cartesianIV},
$X_*(\tilde T_M)$ and $X_*(\tilde T)$ are a sum of induced modules. It follows that
$(X_*(\tilde T)/X_*(\tilde T_M))^\Gamma = X_*(\tilde T)^{\Gamma}/X_*(\tilde T_M)^\Gamma,$ and that
$X_*(\tilde T)^{\Gamma}$ is spanned by the sum of $\Gamma$-orbits in $X_*(\tilde T)^{\Gamma}.$
\end{proof}

\section{The superbasic case}\label{secsb}

\subsection{Superbasic $\sigma$-conjugacy classes}\label{secsbI}

As recalled above, an element $b \in G(L)$ is called {\it basic} if $\nu_b$ factors through the
center of $G.$ This condition depends only on the $\sigma$-conjugacy class of $b.$
We say that $b$ is {\it superbasic} if no $\sigma$-conjugate of $b$ is contained in a proper Levi
subgroup of $G$ defined over $F.$ Since all maximal $F$-split tori of $G$ are conjugate over $F,$
this is equivalent to asking that no $\sigma$-conjugate of $b$ is contained in a proper Levi
subgroup of $G$ defined over $F,$ and containing $T.$
If $b$ is superbasic, then $M_b = G,$ by Lemma \ref{(1.2.5)}(2), and $\nu_b$ is central, so $b$ is basic.

\begin{lemma}\label{(1.2.6)} If $b \in G(L)$ is superbasic, then $J_b$ is anisotropic modulo center, and in particular the simple factors of $G^{\ad}$ are of the form $\Res_{E_i/F} \PGL_{h_i}$ for some unramified extension $E_i/F$ and $h_i \geq 2.$
\end{lemma}
This is analogous to \cite{GHKR}, 5.9.1. We are grateful to R.~Kottwitz for explaining how to adapt the proof of loc.~cit.~to the quasi-split setting.
\begin{proof}
A cocharacter $\ep \in X_*(J_b)^{\Gamma}$ may be regarded as a cocharacter of $G$ such that
$\sigma(\ep) = b^{-1}\ep b.$ Then as above, $\ep$ is conjugate by a $g \in G(L)$ to a dominant
cocharacter $\ep' \in X_*(T)$ defined over $F.$ That is, $\sigma(g^{-1}\ep' g) = b^{-1}g^{-1}\ep' gb,$
which implies that $gb\sigma(g^{-1})$ commutes with $\ep'.$ Since $gb\sigma(g^{-1})$ is not contained
in a proper Levi subgroup of $G$ containing $T,$ $\ep$ must be central.

The fact that $J_b$ is anisotropic modulo center implies that all the factors of $J_b^{\ad}$ are isomorphic to the group of units of a division algebra over an extension of $F$ modulo its center \cite{Ti} \S4. Since $G^{\ad} = M_b^{\ad}$ is an inner form of $J_b^{\ad},$ which is quasi-split, its simple factors have the form
$\Res_{E_i/F} \PGL_{h_i}$ for some finite extensions $E_i/F.$ As $G$ is unramified, $E_i$ must be an unramified extension of $F.$
\end{proof}

\begin{para}\label{para:superbasicsetup} For every $[b]\in B(G)$ there exists a standard parabolic subgroup $P$ of $G$ defined over $F$ with Levi factor $M$ containing $T$, unipotent radical $N$ and the following properties. There exists $b' \in [b]\cap M(L)$ such that $b'$ is superbasic in $M,$ i.e. no $\sigma$-conjugate of $b$ lies in a proper Levi subgroup of $M$. Thus we may assume that $b \in M(L)$ is superbasic.
\end{para}
\subsection{The superbasic case for $\GL_h$}\label{(1.3)}
Let $E/F$ be a finite unramified extension and suppose $G = \Res_{\O_E/\O_F} \GL_h,$
with $T$ the standard diagonal torus and $B$ the Borel subgroup of upper triangular matrices.
In this subsection we will prove Theorem \ref{thmzshk} for this $G$ when $b$ is superbasic.
For the rest of this subsection, we suppose that $b$ is a superbasic element of $G(L).$

Let $n =[E:F].$ The $F$-algebra embeddings $E \hookrightarrow L$ are permuted
cyclically by Frobenius, so over $\O_E$ we may identify $G$ with $(\GL_h)^n,$
such that $\sigma$ acts on $G(L) = \GL_h(L)^n,$ by
$$ \sigma(g_1, \dots, g_n) = (\sigma(g_n), \sigma(g_1), \dots, \sigma(g_{n-1})) .$$

We get an analogous decomposition of $X_*(T),$ and for $r=1, \dots ,n,$ we denote by $\mu_r$
the projection of $\mu$ onto the $r^{\th}$ factor of $X_*(T).$
Let $\mu_{r,\min} \in X_*(T)$ denote the unique dominant minuscule cocharacter with
$\mu_{r,\min} \preceq \mu_r$ (that is with $\det(\mu_{r,\min}(p)) = \det(\mu_r(p))$, compare \eqref{minuscule_pi1} below)
and set $\mu_{\min} = (\mu_{r,\min})_r.$

Let $h \geq 1$ be an integer and $e_1, \dots , e_h$ the standard basis of $L^h.$
We define $e_i$ for $i \in \mathbb Z$ so that $e_{i+h} = pe_i.$
Let $s \in GL_h(F)$ be defined by $s(e_i) = e_{i+1}$ for all $i.$

Note that for $i \in \Z,$ $s^i = {}^i\mu_{\min}(p)w^i$ where $w$ is the Weyl
group element given by $w(e_i) = e_{i+1}$ for $i = 1, \dots, h-1$
and $w(e_h) = e_1,$ and ${}^i\mu_{\min}$ is the unique dominant minuscule cocharacter
of $\GL_h$ such that $\det({}^i\mu_{\min}(p)) = p^i.$

\begin{lemma}\label{(1.3.1)} If $X_{\preceq \mu}(b) \neq \emptyset,$ $b$ is $\sigma$-conjugate to
$b_{\min} = (s^{m_r}) \in G(L),$ where $m_r \in \Z$ satisfies ${}^{m_r}\mu_{\min} = \mu_{r,\min}.$
Moreover, we have $(\sum_r m_r,h)=1.$
\end{lemma}
\begin{proof} Recall from \cite{Kottwitz1}, Proposition 5.6 that $\kappa_G$ induces a bijection between the set of
basic $\sigma$-conjugacy classes in $G(L)$ and $\pi_1(G)_{\Gamma}.$
The Newton cocharacter of $(s^{m_r})$ is the central cocharacter of
$GL_h \subset G$ corresponding to the rational number  $n^{-1}h^{-1}\sum m_r.$
In particular $(s^{m_r})$ is basic. As $X_{\preceq \mu}(b) \neq \emptyset,$ we have $\kappa_G(b) = \mu$
in $\pi_1(G)_{\Gamma}.$ Furthermore $\mu$ and $(s^{m_r})$ both have image $\sum_r m_r$
in $\pi_1(G)_{\Gamma} \iso \Z.$ Thus $b$ and $(s^{m_r})$ are $\sigma$-conjugate.

If $(\sum_r m_r,h) \neq 1,$ then  there exist integers $m_r'$ with
$\sum_r m_r' = \sum_r m_r,$ and such that $\gcd(m_1',\dots,m_r',h) > 1.$ Then the same argument as above shows that $b$ is
$\sigma$-conjugate to $(s^{m_r'}).$ The latter element is contained in a proper Levi subgroup of $G,$
defined over $F,$ which contradicts the fact that $b$ is superbasic.
\end{proof}

\begin{para}\label{(1.3.2)} Let $i, \delta \in \mathbb Z.$
If $\delta \neq 0,$  set $\R_{\delta} = \O_L\langle x \rangle,$ the $p$-adic
completion of $\O_L[x].$ Similarly, if $\delta = 0,$ we set $\R_{\delta}$ equal to the $p$-adic completion of $\O_L[x,(1+x)^{-1}].$
Let $a_{i,\delta} \in \GL_h(\R_{\delta})$ which sends $e_j$ to $e_j +xe_{j+\delta}$ if $h|(j-i)$ and
fixes $e_j$ otherwise.
\end{para}

\begin{lemma} \label{(1.3.3)} Let $g \in \GL_h(L)$ and let $\delta_g \in \mathbb Z$ be minimal such that
$a_{i,\delta}(x)\circ g \in g\GL_h(\R_{\delta})$ for all $\delta > \delta_g$ and $i.$
Then
\begin{enumerate}
\item Either $\delta_g \geq 1$ or $\delta_g = -1.$
\item If $\delta_g = -1$ then $g\GL_h(\O_L)$ contains an element of the form $s^j$ for some $j \in \mathbb Z.$
\item If $\delta_g \geq 1$ then there exists a unique $i_g \in \{1, \dots, h\}$ with
$a_{i_g,\delta_g}(x)\circ g \notin g \GL_h(\R_{\delta}).$
\item If $i,i' \in \{1. \dots, h\},$ and $\delta \geq \delta' > 0,$ then $a_{i,\delta}(x)a_{i,\delta}(x')a_{i,\delta}(-x-x')$ and
the commutator $[a_{i,\delta}(x), a_{i',\delta'}(x')]$
can be written as a (possibly infinite, $p$-adically convergent) product of terms of the form $a_{i_j,\delta_j}(x_j)$
with $\delta_j > \delta.$
\end{enumerate}
\end{lemma}
\begin{proof}This is a translation of \cite{conncomp}, Lemma 2.
The proof given in {\it loc.~cit} goes over verbatim, except that
the elements $\beta_j \in \bar k$ which appear in it should be replaced by Teichm\"uller representatives in $W(\bar k).$
Note that in {\it loc.~cit} the definition of $\delta_g$ and condition (3) are formulated by asking that
$a_{i,\delta}(x)\circ g$ is contained (resp. not contained) in $g\GL_h(\O_L)$ for every specialization
of $x$ at a point of $\bar k.$ This is equivalent to the formulation here, for example using Lemma \ref{1.1.9}
\end{proof}
\begin{lemma}\label{(1.3.4)} Let $s \in \GL_h(F) \subset G(F)$ be as above, and
suppose that $b = b_{\min}.$
Then $\langle s \rangle \subset J_b(F)$ acts transitively on $\pi_0(X_{\preceq \mu}(b)).$
\end{lemma}
\begin{proof} For $r=1, \dots ,n $ let $\delta_{g_r}$ be the integer obtained by applying Lemma \ref{(1.3.3)} to $g_r,$ and
if  $\delta_{g_r} \geq 1, $ let $i_{g_r}$ be the integer produced by (3) of that lemma. Suppose that $g,g' \in X_{\preceq \mu}(b),$
and that $\delta_{g_r} = \delta_{g'_r} = -1$ for all $r.$ We claim that $g$ and $g'$ are in
the same $\langle s \rangle $-orbit. By Lemma \ref{(1.3.3)}(2) we may assume that for $r=1,2, \dots, n$ we have
$g_r = s^{j_r}$ and $g'_r = s^{j'_r}$ for some $j_r,j'_r \in \mathbb Z.$ Note that $\sigma(s) = s \in J_b(F),$ so that
$$ s^{j_{r-1}-j_r} b_r = s^{-j_r}b_rs^{j_{r-1}} \in \GL_h(\O_L) p^{\mu'_r} \GL_h(\O_L) $$
for some $\mu'_r \preceq \mu_r.$ Here we set $j_{-1} = j_n,$ and we have again written $b_r$ for the
image of $b$ under the $r^{\text{th}}$ projection $G(L) \rightarrow \GL_h(L).$
Hence $v_p(\det(s^{j_r-j_{r-1}} b_r)) = v_p(\det(b_r))+j_r-j_{r-1}$ depends only
on $\mu_r$ and not on $g.$ It follows that $j = j_r-j'_r$ is independent of $r,$
so that $ g = s^jg'.$

Note that if $h=1,$ then $\delta_{g_r} = -1$ for all $r$ for any $g,$ so we are done in this case (which can of course
be easily checked directly). If $h > 1,$ it remains to show that given $g \in X_{\preceq \mu}(b)$
with $\delta_{g_r} > 0$ for some $r,$ there exists $g' \in X_{\preceq \mu}(b)$ in the same connected
component as $g,$ with
$\delta_{g'_r} \leq \delta_{g_r}$ for $r=1,\dots, n$ and such that this inequality is strict for some $r.$

Let $\R = \O_L \langle x \rangle$ equipped with the lift of Frobenius given by $x \mapsto x^q.$
Choose $r_0$ such that $\delta_{g_{r_0}}$ is maximal among the $\delta_{g_r}$ and set $\delta = \delta_{g_{r_0}} > 0.$
(In the following it will be convenient to view the indices $r$ in $\Z/n\Z.$)
Define $a = (a_r) \in (GL_h)^n(\R)$ as follows: If not all the $\delta_{g_r}$ are equal $\delta,$ let
$r_1 < r_0$ be an integer with $\delta_{g_{r_1}} < \delta.$ Then for $r=r_1, \dots, r_1+h-1$ we set
$a_r = \sigma^{r-r_1}(a_{j_r,\delta}(x)),$ where
$j_{r_1} = i_{g_{r_0}}- m_{r_0}- \dots - m_{r_1+1}$ and $j_r = j_{r_1}+m_{r_1+1}+ \dots + m_r$
for $r=r_1+1, \dots, r_1+n-1.$ If all the $\delta_{g_r}=\delta$ we choose $r_1=r_0$ so that $h\nmid m_{r_0}$
and set $a_r = \sigma^{r-r_0}(a_{j_r,\delta}(x)),$ where $j_{r_0} = i_{g_{r_0}}$ and
$j_r = i_{g_{r_0}}+ m_{r_0+1} + \dots +m_r$ for $r=r_0+1,\dots ,r_0+n-1.$

Then, as in \cite{conncomp}, p.~322, for $r\neq r_1$, we have, using Lemma \ref{(1.3.3)},
\begin{align}
\nonumber&\GL_h(\R)g_r^{-1}a_{r}^{-1}b_r\sigma(a_{r-1})\sigma(g_{r-1})\GL_h(\R) \\
\nonumber&\quad=\GL_h(\R)g_r^{-1}\sigma^{r-r_1}(a_{j_r,\delta}(x)^{-1})b_r\sigma^{r-r_1}(a_{j_{r-1},\delta}(x))\sigma(g_{r-1})\GL_h(\R) \\
\nonumber&\quad= \GL_h(\R)g_r^{-1}\sigma^{r-r_1}(a_{j_r,\delta}(x)^{-1}a_{j_{r-1}+m_r,\delta}(x))b_r\sigma(g_{r-1})\GL_h(\R) \\
\nonumber&\quad= \GL_h(\R)g_r^{-1}\sigma^{r-r_1}(a_{j_r,\delta}(-x)a_{j_{r-1}+m_r,\delta}(x))b_r\sigma(g_{r-1})\GL_h(\R).
\end{align}
From the definition of the $a_r$ and $j_r,$ it follows that this is equal to
$$\GL_h(\R)g_r^{-1}b\sigma(g_{r-1})\GL_h(\R).$$ For $r=r_1$ a similar calculation shows
\begin{align}
\nonumber&\GL_h(\R)g_{r_1}^{-1}a_{r_1}^{-1}b_{r_1}\sigma(a_{r_1-1})\sigma(g_{r_1-1})\GL_h(\R) \\
\label{1.3.5'}&\quad= \GL_h(\R)g_{r_1}^{-1}a_{j_{r_1},\delta}(-x)\sigma^n(a_{j_{r_1-1}+m_{r_1},\delta}(x))b_{r_1}\sigma(g_{r_1-1})\GL_h(\R).
\end{align}
We claim that this is again equal to $\GL_h(\R)g_{r_1}^{-1}b_{r_1}\sigma(g_{r_1-1})\GL_h(\R).$
If not all the $\delta_{g_r}$ are equal to $\delta,$ this follows from $\delta > \delta_{g_{r_1}}.$
If all the $\delta_{g_r} = \delta,$ then using Lemma \ref{(1.3.3)} (4), the expression \eqref{1.3.5'} is equal to
\begin{align*}
\nonumber&\GL_h(\R)g_{r_1}^{-1}\sigma^n(a_{j_{r_1-1}+m_{r_1},\delta}(x))a_{j_{r_1},\delta}(-x)b_{r_1}\sigma(g_{r_1-1})\GL_h(\R) \\
&\quad= \GL_h(\R)g_{r_1}^{-1}\sigma^n(a_{j_{r_1-1}+m_{r_1},\delta}(x))b_{r_1}a_{j_{r_1}-m_{r_1},\delta}(-x)\sigma(g_{r_1-1})\GL_h(\R)
\end{align*}

Now $j_{r_1}-m_{r_1} = i_{g_{r_1}}-m_{r_1} \neq i_{g_{r_1}}$ in $\Z/h\Z$ as $h\nmid m_{r_1},$
while $j_{r_1-1}+m_{r_1} = i_{g_{r_1}} + \sum_r m_r \neq i_{g_{r_1}}.$ Hence the uniqueness of $i_{g_{r_1}}$
in Lemma \ref{(1.3.3)}(3) implies the claim in this case also. It follows that $ag \in X_{\preceq \mu}(b)(\R)$.

Let $\R'$ and $\R''$ denote the $p$-adic completions of $\O_L[y]$ and $\O_L[x,x^{-1}]$ respectively,
equipped with the lifts of Frobenius $\sigma$ given by $y \mapsto y^q$ and $x \mapsto x^q.$ We consider $\R'$
as subring of $\R''$ via $y \mapsto x^{-1}.$ We may consider $ag \in  X_{\preceq \mu}(b)(\R'').$ Then
by Lemma \ref{valuativeII}, $ag$ is induced by an element $\gamma \in X_{\preceq \mu}(b)(\R').$

Now $(a\circ g)|_{x=0} = g,$ and a computation as in \cite{conncomp}, proof of Proposition 1 for superbasic $b$,
using Lemma \ref{(1.3.3)} (4), shows that $g' = \gamma|_{y=0}$ satisfies $\delta_{g'_{r_0}} < \delta_{g_{r_0}}$ and $\delta_{g'_r} \leq \delta_{g_r}$
for $r\neq r_0.$ Since $g$ and $g'$ are in the same connected component of $X_{\preceq \mu}(b)(W(\bar k)),$ the Lemma follows.
\end{proof}

\begin{para}\label{para:defnw'} It will be convenient to formulate a slight variant of Lemma \ref{(1.3.4)}.
Recall the element $w$ defined at the beginning of this subsection, which permutes the chosen basis
$e_1, \dots, e_h$ cyclicly. Then $\det(w) = (-1)^{h-1}.$ Let $w' = tw$ where
$t(e_1) = (-1)^{h-1}(e_1)$ and $t(e_i) = e_i$ for $i > 1.$ Then $w' \in \SL_h(F).$
We set $s' = ts = {}^1\mu_{\min}(p)w',$ and $b_{\min}' = ((s')^{m_r})_r \in G(L).$
\end{para}

\begin{cor}\label{cor:superbasicgl} If $b = b_{\min}'$ then $b$ is superbasic in $G,$
and $\langle s' \rangle \subset J_b(F)$ acts transitively on $\pi_0(X_{\preceq \mu}(b)).$
\end{cor}
\begin{proof} The same argument as in Lemma \ref{(1.3.1)} shows that $b_{\min}'$ is superbasic
in $G(L)$ and $\sigma$-conjugate to $b_{\min}.$ By Lemma \ref{(1.3.4)},
$\pi_0(X_{\preceq \mu}(b))$ maps isomorphically to $\pi_1(G)_{\Gamma} = \mathbb Z.$
Since $s'$ maps to a generator of  $\pi_1(G)_{\Gamma}$, $\langle s' \rangle$ acts
transitively on $\pi_0(X_{\preceq \mu}(b)).$
\end{proof}

\subsection{The superbasic case in general}\label{(1.4)}
We return to the notation and assumptions introduced in subsection \ref{secsbI}.

\begin{prop}\label{prop:generalsuperbasictrans} Suppose that $b \in G(L)$ is superbasic. Then
$$ \pi_0(X_{\preceq \mu}^G(b)) \iso c_{b,\mu} \pi_1(G)^{\Gamma},$$
 and $J_b(F)$ acts transitively on $\pi_0(X_{\preceq \mu}^G(b)).$
\end{prop}
\begin{proof}
By Lemma \ref{(1.2.6)}, $G^{\ad}$ is isomorphic to
$\prod_{i\in I} \Res_{E_i/F}\PGL_{h_i}$ with $E_i/F$ some finite unramified extension of degree $n_i,$ and $h_i > 1.$
Fix such an isomorphism. Let $\mu_{\min} \in X_*(T)$ denote the unique dominant minuscule cocharacter whose image in
$\pi_1(G)$ is equal to that of $\mu.$ The induced cocharacter of $\prod_{i\in I} \Res_{E_i/F}\PGL_{h_i}$
has the form $({}^{m_{i,r}}\mu_{\min})_{i,r}$ where $i$ runs over elements of $I,$ $1\leq r \leq n_i,$
and  $m_{i,r} \in \mathbb Z.$

Write $w'_{h_i}$ and $s'_{h_i}$ for the elements introduced in \ref{para:defnw'} above, when $h = h_i.$
Since $w'_{h_i} \in SL_{h_i}(F),$ we may regard $((w'_{h_i})^{m_{i,r}})_{i,r} \in \tilde G^{\ad}$ where $\tilde G^{\ad}$ denotes
the simply connected cover of $G^{\ad}.$ In particular, we may regard $((w'_{h_i})^{m_{i,r}})_{i,r}$ and hence
$b'_{\min} : = \mu_{\min}(p)((w'_{h_i})^{m_{i,r}})_{i,r}$ as elements of $G(L).$ The image of $b'_{\min}$ in $G^{\ad}$
is $((s'_{h_i})^{m_{i,r}})_{i,r}.$ Hence $b'_{\min}$ is basic, and the same argument as  in the proof of Lemma \ref{(1.3.1)} shows that $b$ is
$\sigma$-conjugate to $b'_{\min}.$ Thus we may assume that $b = b'_{\min}.$

Let $b^{\ad}$ be the image of $b$ in $G^{\ad}$ and $b_{\GL} = ((s')^{m_{i,r}})_{i,r} \in \prod_{i\in I} \Res_{E_i/F}\GL_{h_i}.$
Similary, let $\mu^{\ad}$ be the cocharacter of $G^{\ad}$ induced by $\mu.$ Let $\mu_{\GL}$ be the cocharacter of
$\prod_{i\in I} \Res_{E_i/F}\GL_{h_i}$ lifting $\mu^{\ad}$ whose image in $\pi_1(\prod_{i\in I} \Res_{E_i/F}\GL_{h_i})$
is equal to $({}^{m_{i,r}}\mu_{\min})_{i,r}.$

By Corollary \ref{cor:superbasicgl}, $\prod_{i \in I} \langle s'_{h_i} \rangle$ acts transitively on $\pi_0(X_{\preceq \mu_{\GL}}(b_{\GL})),$
and in particular the first claim of the Proposition holds for $(\mu_{\GL}, b_{\GL}).$ It follows from
Corollary \ref{cartesianIII} that $\prod_{i \in I} \langle s'_{h_i} \rangle$ acts transitively on $\pi_0(X_{\preceq \mu^{\ad}}(b^{\ad})),$
and that the first claim of the Proposition holds for $(\mu^{\ad}, b^{\ad}).$

Using Corollary \ref{cartesianIII} again, we see that the first claim of the Proposition holds, and that, since $Z_G(F) \subset J_b(F),$
to prove the second claim it suffices to show that, if the image of $((s'_{h_i})^{j_i})_{i\in I}$ in $\pi_1(G^{\ad})^{\Gamma}$
lifts to $\pi_1(G)^{\Gamma}$ for some integers $j_i,$ then $((s'_{h_i})^{j_i})_{i \in I} \in J_{b^{\ad}}(F) \cap G^{\ad}(F)$ lifts to an element of $G(F).$ But
$((s'_{h_i})^{j_i})_{i \in I} = ({}^{j_i}\mu_{\min}(p)(w'_{h_i})^{j_i}),$ so it suffices to show that (the image of) $({}^{j_i}\mu_{\min}(p))_{i \in I}$ lifts to $G(F).$
This follows, for example, from Lemma \ref{cartesianIV}.
\end{proof}

\subsection{Reduction to the superbasic case}\label{secredsb}

Let $[b]\in B(G,\mu)$ and $M \subset G$ a smallest standard Levi subgroup of $G,$ defined over $F$ and containing $T,$
and which contains an element of $[b].$ Fix a representative $b\in M(L)$ of $[b],$ so that $b$ is superbasic in $M(L).$
Let $P \supset B$ be the parabolic with reductive quotient $M,$ and $N \subset P$ its unipotent radical.

Let $\bar I_{\mu,b}$ be the set of $M$-conjugacy classes of cocharacters $\mu':\mathbb{G}_{m}\rightarrow  M$ (defined over some finite extension of $F$) such that  $\mu':\mathbb{G}_m\rightarrow G$ satisfies $\mu'\preceq\mu$ and such that $b\in B(M,\mu')$. We identify an element of $\bar I_{\mu,b}$ with its $M$-dominant representative in $X_*(T)$.
Note that in general (even for minuscule $\mu$) this set is non-empty and finite, but may have more than one element.
For each $\mu'\in \bar I_{\mu,b}$ we have a canonical inclusion $X_{\preceq\mu'}^M(b)\rightarrow X_{\preceq\mu}^G(b)$. The following proposition is the main goal of this subsection.
\begin{prop}\label{(1.5.5)}
Each connected component of $X_{\preceq\mu}^G(b)$ contains an element $jg$ where $j\in J_b(F)\cap N(L)$ and $g\in X_{\preceq\mu'}^M(b)$ for some
 $\mu'\in \bar I_{\mu,b}$.
\end{prop}
The proof of this is very similar to \cite{conncomp}, proof of Proposition 1.

\begin{para}\label{(1.5.1)} For any $l \geq 0,$ let $b^{(l)} = b\sigma(b)\dotsm \sigma^l(b).$
By \cite{Kottwitz1}, 4.3, after replacing $b$ by a $\sigma$-conjugate in $M,$ we may assume that
for some $l_0 >0,$ $b^{(l_0)} = p^{l_0\nu},$ where $\nu = \nu_b$ is defined over $F,$ as before.

Let $\{\alpha_i\}_{i=1}^r$ denote the roots of $T$ in $N.$
We denote by $U_{\alpha_i} \subset N,$ the corresponding root subgroup.
It will be convenient to identify $U_{\alpha_i}$ with $\GG_a.$ Then
for an $F$-algebra $R$ and $\beta \in R,$ we can regard $\beta$
as a point $U_{\alpha_i}(\beta) \in U_{\alpha_i}(R).$

For $j \geq 1$ let $N[j] \subset N$ denote the subgroup generated by those
$U_{\alpha_i}$ for which the sum of the coefficients of $\alpha_i,$ expressed
as a linear combination of simple roots of $A$ in $N,$ is $\geq j.$
Then for $j,j' \geq 1,$ $[N[j],N[j']] \subset N[j+j'].$
The filtration $N \supset N[1] \supset N[2] \dots $ may be refined into a filtration
$N \supset N_1 \supset N_2 \dots$ such that $N_i/N_{i+1}$ is one dimensional.
After reordering the $\alpha_i$ we may assume that $N_i$ is generated by $U_{\alpha_{i'}}$ for $i' \geq i.$

Now suppose that $R$ is a $\bar k$-algebra, $\R$ a frame for $R,$ and $y \in N(\R_L).$
We set
$$ f_b(y) = y^{-1}b\sigma(y)b^{-1}.$$
Then $f_b(y) \in N(\R_L).$
\end{para}

\begin{lemma}\label{(1.5.2)} Let $R$ be a smooth $\bar k$-algebra, $\R$ a frame for $R,$ and $\beta \in \R_L.$ Assume that there is an element $x \in (\Spec R)(\bar k)$ with $\beta(x) = 0$. If $i \geq 1,$ and $j$ is maximal such that $N[j] \supset N_i$ then for $n \geq 1$ a positive integer,
there exists a finite \'etale covering $R \rightarrow R',$ with frame $\R \rightarrow \R',$
and $z \in N[j](\R'_L)$ such that
\begin{enumerate}
\item $f_b(z) \in U_{\alpha_i}(\beta+\varepsilon)N_{i+1}(\R'_L)$ for some $\varepsilon \in p^n\R'.$
\item there exists $x' \in (\Spec R')(\bar k)$ mapping to $x$ such that $z(x') = 1.$
\end{enumerate}
\end{lemma}
\begin{proof} This is analogous to the argument of \cite{conncomp}, p.~324--325.

Suppose first that $\langle \alpha_i, \nu \rangle > 0,$ and set
$$ z(l) = b^{(l)}\sigma^l(U_{\alpha_i}(-\beta))(b^{(l)})^{-1} \dots b\sigma(U_{\alpha_i}(-\beta))b^{-1}U_{\alpha_i}(-\beta).$$
Note that conjugation by $b^{(l_0)}$ acts on $U_{\alpha_i}$ by $p^{\langle \alpha_i, \nu \rangle}.$
Using this one sees as in {\it loc.~cit} that the sequence $z(l)$ converges to an element $z \in U_{\alpha_i}(\R_L)$
such that $f_b(z) = U_{\alpha_i}(\beta).$ Thus we may take $R' = R.$

Suppose that $\langle \alpha_i, \nu \rangle = 0.$ Let $R'$ be finite \'etale over $R,$ and $z_0 \in \R'_L.$
Set
$$ z = b^{(l_0-1)}\sigma^{l_0-1}(U_{\alpha_i}(z_0))(b^{(l_0-1)})^{-1}\dots b\sigma(U_{\alpha_i}(z_0))b^{-1}U_{\alpha_i}(z_0).$$
Then we have
\begin{eqnarray*} f_b(z) = z^{-1}b\sigma(z)b^{-1} &=& z^{-1}b^{(l_0)}\sigma^{l_0}(U_{\alpha_i}(z_0))(b^{(l_0)})^{-1}zU_{\alpha_i}(-z_0) \\
&=& z^{-1}\sigma^{l_0}(U_{\alpha_i}(z_0))zU_{\alpha_i}(-z_0).
\end{eqnarray*}
Since all the terms in the product defining $z$ are in $N[j],$ we have $z \in N[j].$ Assume that $l_0$ is such that $\sigma^{l_0}$ acts trivially on $X^*(T)$. Then the final
term is equal to $U_{\alpha_i}(\sigma^{l_0}(z_0)-z_0)\ \ \mod\ N[j+1]$, and $z$ will have the desired property if
$z_0$ satisfies
$$ \sigma^{l_0}(z_0) - z_0 = \beta \quad \text{\rm mod } p^{n}\R'.$$

To show this equation has a solution for some $\R'/\R$ finite \'etale we may replace $\beta$ and $n$ by
$p^m\beta$ and $n+m$ respectively and assume that $\beta \in \R.$ Then one sees by induction on $n,$
that the above equation has a solution over a finite \'etale covering of $\R.$
\end{proof}

\begin{lemma}\label{(1.5.3)} Let $R$ be a smooth $\bar k$-algebra with frame $\R,$ and $x_1 \in (\Spec R)(\bar k).$
Suppose $y \in N(\R_L),$ and $z_1 \in N(L)$ satisfy $f_b(z_1) = y(x_1).$
Then for any bounded open subgroup $K' \subset N(L)$ there exists a finite \'etale covering $R \rightarrow R',$
with canonical frame $\R\rightarrow \R',$ and $z \in N(\R'_L)$ such that
\begin{enumerate}
\item For every $\bar k$-valued point $x$ of $R',$
$$ f_b(z(x)) K' = y(x)K'.$$
\item There exists a point $x_1' \in (\Spec R')(\bar k)$ over $x_1$ such that $z(x_1') = z_1.$
\end{enumerate}
\end{lemma}
\begin{proof} We remind the reader that in the statement of the lemma and below,  a map $R \rightarrow \bar k$ and the induced map
$\R \rightarrow W(\bar k)$ are denoted by the same symbol.

We will construct a finite \'etale covering $R \rightarrow R_i$ with canonical frame
$\R \rightarrow \R_i,$ together with
a point $x_{1,i} \in (\Spec R_i)(\bar k)$ over $x_1$
and elements $z_i \in N(\R_L)$ and
$\delta_i \in N_i(\R_L)$ such that for every $x \in (\Spec R_i)(\bar k)$
$$ f_b(z_i(x))\delta_i(x)K' = y(x)K',$$
$z_i(x_{1,i}) = z_1,$ and $\delta_i(x_{1,i}) = 1.$

When $i=1,$ then $N_1 = N,$ and the element $z_1 \in N(L) \subset N(\R_L)$ satisfies these conditions, with $\delta_1 = f_b(z_1)^{-1}y.$
Suppose that $z_i, \delta_i$ and $x_{1,i}$ with these properties have already been constructed.
Let $j$ be maximal such that $N_i \subset N[j].$
Then $\delta_i \in U_{\alpha_i}(\beta)N_{i+1}(\R_{i,L})$ for some $\beta \in \R_{i,L}.$ By Lemma \ref{(1.5.2)}, for any $n \geq 0,$
there exists a finite \'etale faithful $R_i$-algebra $R_{i+1}$ and elements $\tilde z \in N[j](\R_{i+1,L})$ and
$\varepsilon_i \in p^n\R_{i+1}$ such that
$$ f_b(\tilde z) \in U_{\alpha_i}(\beta + \varepsilon_i)N_{i+1}(\R_{i+1,L}).$$
Note that $\delta_i(x_{1,i}) = 1$ implies $\beta(x_{1,i})= 0,$ so by Lemma \ref{(1.5.2)}(2)
we may assume that there is a point $x_{1,i+1} \in (\Spec R_{i+1})(\bar k)$ over $x_{1,i}$
such that $\tilde z(x_{1,i+1}) = 1.$

Let $z_{i+1} = z_i\tilde z.$
Since $\tilde z, b\sigma(\tilde z)b^{-1} \in N[j](\R_{i+1,L}),$ and $[N[j],f_b(z_i)] \subset N[j+1]$ we have
$$ f_b(z_{i+1}) = \tilde z^{-1}f_b(z_i)(b\sigma(\tilde z)b^{-1}) = f_b(z_i)f_b(\tilde z)\gamma_{i+1}$$
for some $\gamma_{i+1} \in N_{i+1}(\R_{i+1,L}).$
Hence
$$ f_b(z_{i+1}) = f_b(z_i)f_b(\tilde z)\gamma_{i+1} = f_b(z_i)\delta_i[U_{\alpha_i}(\varepsilon)\delta_{i+1}^{-1}]$$
for some $\delta_{i+1} \in N_{i+1}(\R_{i+1,L}).$
Now choose $n$ so that $U_{\alpha_i}(p^n\mathcal{O}_L) \subset K'.$
Then for every $x \in (\Spec R_{i+1})(\bar k)$ we have
$$ f_b(z_{i+1})(x)\delta_{i+1}(x)K' = f_b(z_i)\delta_i(x)U_{\alpha_i}(\varepsilon(x))K' = y(x)K'.$$
Moreover, since $\tilde z(x_{1,i+1})=1,$ we have $\varepsilon(x_{1,i+1}) = 0$ and $z_{i+1}(x_{1,i+1}) = z_i(x_{1,i+1}),$ which implies
that $\gamma_{i+1}(x_{1,i+1}) =  \delta_{i+1}(x_{1,i+1}) = 1.$

This completes the induction step. Taking $i$ large enough that $N_i = 0,$ the lemma follows.
\end{proof}

\begin{lemma}\label{(1.5.4)} Let $m \in M(L).$ Then there exists a compact open subgroup $K' \subset N(L)$ such that
$$ K' \subset f_b(N(L) \cap mKm^{-1}) $$
\end{lemma}
\begin{proof} This can be shown using the methods of \cite{GHKR} 5.3.1, 5.3.2. In our present situation, when
$\text{\rm char} L = 0,$
there is a simpler argument which we now sketch.

Let $\mathfrak n = \Lie N$ regarded as an $L$-scheme. The map $f_b$ induces the map
$$ df_b: \mathfrak n \rightarrow \mathfrak n:\quad n \mapsto \ad(b)(\sigma(n)) - n. $$
Since $N(L) \cap mKm^{-1}$ is a bounded open subgroup of $N(L),$ an argument using the
exponential shows that it suffices to show that $df_b$ maps a bounded open subset of $\mathfrak n(L)$
to a bounded open subset of $\mathfrak n(L).$

Now for any $L$-vector space $V$ equipped with a $\sigma$-semi-linear map $\sigma_V,$ the map $\sigma_V - 1$
maps bounded open subset onto bounded open subsets. This may be checked as in \cite{GHKR}, 4.3.1 using the classification
of $\sigma$-isocrystals $(V, \sigma_V).$
\end{proof}

\begin{proof}[Proof of Proposition \ref{(1.5.5)}] Let $g_1 \in X_{\preceq \mu}(b).$
By the Iwasawa decomposition, $g_1$ has a representative in $G(L)$ of the form $nm$ with $n \in N(L)$ and $m \in M(L).$
Let $\chi \in X_*(Z_M)$ be such that $\langle \chi, \alpha \rangle > 0$ for every root $\alpha$ of $T$ in $N.$

Let $\O_L\langle s, s^{-1} \rangle$ and $\O_L\langle s \rangle$ denote the $p$-adic completions of
$\O_L[s,s^{-1}]$ and $\O_L[s]$ respectively. We equip these rings with the Frobenius lifts given by
$s \mapsto s^q,$ and consider them as frames of their mod $p$ reductions.
Define $y = \chi(s)f_b(n)\chi(s)^{-1}\in N(\O_L\langle s,s^{-1}\rangle_L)).$ For any root $\alpha,$ conjugation by
$\chi(s)$ maps $U_{\alpha}(\beta)$ to $U_{\alpha}(s^{\langle \chi,\alpha \rangle}\beta).$
Hence $y \in N(\O_L\langle s \rangle_L).$ Note also that $y(0) = 1,$ while $y(1) = f_b(n).$

Using Lemma \ref{(1.5.4)}, we choose a bounded open subgroup $K' \subset N(L)$ such that $K' \subset f_b(N(L) \cap mKm^{-1}).$ We may also assume that $K' \subset (b\sigma(m))K(b\sigma(m))^{-1}.$
Applying Lemma \ref{(1.5.3)}, we find a finite \'etale covering $\bar k[s]\rightarrow R,$ with canonical frame
$\O_L\langle s \rangle \rightarrow \R,$ an element $z \in N(\R_L),$ and a point
$ x_1 \in (\Spec R)(\bar k)$ over $1,$ such that $f_b(z(x))K' = y(x)K'$ for every $x$ in $(\Spec R)(\bar k),$
and $z(x_1) = n.$ The first condition implies that
$$ f_b(z(x))b\sigma(m)K = y(x)b\sigma(m)K. $$
We may replace $\Spec R$ with the connected component containing $x_1$ and assume that this scheme is connected.

Let $g = zm \in G(\R_L).$ For $x \in (\Spec R)(\bar k)$ such that $s(x) \in \bar k^\times,$
we have
\begin{eqnarray*}
g(x)^{-1}b\sigma(g(x))K &=& m^{-1}f_b(z(x))b\sigma(m)K\\ &=& m^{-1}y(x)b\sigma(m)K  \\
&= &\chi(s(x))^{-1}m^{-1}f_b(n)b\sigma(m)\chi(s(x))K \\
 &=& \chi(s(x))^{-1}g_1^{-1}b\sigma(g_1)\chi(s(x))K\\& \subset&
Kp^{\mu} K,
\end{eqnarray*}
Hence $g \in X_{\preceq \mu}(b)(\R)$ by Lemma \ref{1.1.3}.

Let $x_0 \in (\Spec R)(\bar k)$ be a point mapping to $0$ in  $\Spec \bar k[s].$ Then $f_b(z(x_0)) \in K',$
so there exists $ k \in N(L) \cap mKm^{-1}$ such that
$f_b(z(x_0)) = f_b(k^{-1}).$ This implies that $z(x_0)k \in J_b(F)\cap N(L).$ Hence
$$ g(x_0) = z(x_0)m = [z(x_0)k]\cdot k^{-1}m \in (J_b(F)\cap N(L))M(L)K. $$
Since $g(x_1) = nm = g_1,$ we see that $g_1 \sim jm$ for some $j \in J_b(F)\cap N(L)$ and $m \in M(L).$
\end{proof}


\section{Connecting points}\label{seccp}
\subsection{Main results: Formulation and overview of the proofs}\label{overview}

In this subsection we reduce the  proofs of our main results Theorem \ref{thmzshk} and Theorem \ref{prop1} to four technical propositions whose proof will be the subject of the remainder of this section. At the end of the subsection we also explain how the arguments simplify if one is only interested in the case that $G$ is split.

We let $G \supset B \supset T$ be as above, $\mu \in X_*(T)$ a dominant, minuscule cocharacter, and $b \in B(G,\mu).$

\begin{para}
For every standard Levi subgroup $M$ of $G$, the projection $X_*(T)\rightarrow \pi_1(M)$ induces a bijection
\begin{equation}
\label{minuscule_pi1} \left\{\genfrac{}{}{0pt}{}{M\text{-minuscule}, M\text{-dominant}}{\text{cocharacters in }X_*(T) }\right\}\simto\pi_1(M).
\end{equation}
For $x\in \pi_1(M)$, denote by $\mu_x$ the preimage of $x$ via
\eqref{minuscule_pi1}. For any $b\in M(L)$ and $G$-minuscule $\mu\in X_*(T)$ , let
\begin{eqnarray*}
\bar I_{\mu, b}^{M, G}&=&\{x\in\pi_1(M)\mid (\mu_x)_{G-\dom}=(\mu)_{G-\dom} ,~x=\kappa_M(b) \text{ in }\pi_1(M)_{\Gamma}\}\\
&=&\{x\in\pi_1(M)\mid x=\mu \text{ in }\pi_1(G),~x=\kappa_M(b) \text{ in }\pi_1(M)_{\Gamma},~\mu_x~G\text{-minuscule}\},
\end{eqnarray*}
where $\Gamma=\Gal(\bar k|k)$.

For every $k$-algebra $R$ with frame $\R$ and every $\mu'\in X_*(T)_{M-\dom}$ we have the natural inclusion
$X^M_{\mu'}(b)(\R)\hookrightarrow X_{\mu'_{\dom}}^G(b)(\R).$ Note that if $\mu'_{\dom}=\mu$ then $\mu'$ is $M$-minuscule, hence of the form $\mu'=\mu_{x}$ for some $x\in \pi_1(M)$. Furthermore $\mu'$ has the same image in $\pi_1(G)$ as $\mu$. Finally $X^M_{\mu'}(b)(\R)=\emptyset$ unless $\kappa_M(b)=x$ as elements of $\pi_1(M)_{\Gamma}$. Hence $X^M_{\mu'}(b)(\R)$ is a nonempty subset of $X_{\mu}^G(b)(\R)$ if and only if the image of $\mu'$ via the natural projection $X_*(T)\rightarrow \pi_1(M)$ is in $\bar I_{\mu, b}^{M, G}$.
\end{para}

\begin{para}
Recall that $N$ is the unipotent radical of the standard parabolic subgroup of $G$ corresponding to $M$.
Let $\Phi_{N}$ be the set of roots in $N$, and let $\Phi_{N, \Gamma}$ be the set of Galois orbits of roots in $N$.
\begin{definition}\label{def_alpha_adapted}

\begin{enumerate}
\item For any root $\alpha\in\Phi_N$, we say that $\alpha$ is \emph{adapted} if $\alpha^{\vee}$ is $M$-anti-dominant, and we have $\langle\beta,
\alpha^{\vee}\rangle\in\{-1, 0, 1\}$ for every root $\beta$ in $M$.

\item For any $\Omega\in\Phi_{N,\Gamma}$, we say that $\Omega$ is \emph{adapted} if some $\alpha\in\Omega$ is adapted.
\end{enumerate}
\end{definition}

As $B$ and $M$ are stable under the action of $\Gamma$, if $\Omega$ is adapted, then so is any element in $\Omega$.
\end{para}

\begin{para}\label{para:stdsetup}
 From now on, we assume that $G^{\ad}$ is simple, as in Theorem \ref{thmzshk}, although this assumption will be droppped towards the end of the subsection. We also suppose that $M\subseteq G$ is a standard Levi subgroup defined over $F$ such that $b$ is superbasic in $M$.
Recall that this implies that $M^{\ad}\cong \prod_i \Res_{F_i/F}PGL_{n_i}$ with $F_i/F$ unramified (Lemma \ref{(1.2.6)}).
Using (\ref{minuscule_pi1}) we have an identification of sets $\bar I_{\mu, b}^{M, G}=\bar I_{\mu,b}$, where $\bar I_{\mu, b}$ is defined in Section \ref{secredsb}.
If $G$ is split, this set consists of a single element.

The proofs of the two main theorems are based on the following propositions.
\end{para}

\begin{prop}[Convexity of $\bar I_{\mu, b}^{M, G}$]\label{pm1}
Let $x,x'\in \bar I_{\mu, b}^{M, G}$. Then there are elements $x_i\in \bar I_{\mu, b}^{M, G}$ for $i=1,\dotsc,m$ for some $m,$ such that $x=x_1$, $x'=x_m$ and such that for each $i$, $$x_{i+1}-x_i=\alpha^{\vee}-\alpha'^{\vee} \text{ in } \pi_1(M)$$ for some roots $\alpha,\alpha'\in\Omega$ with $\Omega\in \Phi_{N,\Gamma}$ (depending on $i$).
\end{prop}

\begin{prop}\label{pm2}  Suppose that $x,x'\in \bar I_{\mu, b}^{M, G}$ with $x-x'=\alpha^{\vee}-\alpha'^{\vee}$ for some $\alpha,\alpha'\in\Omega$ with $\Omega\in\Phi_{N,\Gamma}$. Then for any $g\in X_{\mu_x}^M(b)$, there is a $g'\in
X_{\mu_{x'}}^M(b)$ such that the images of $g$ and $g'$ in
$X_{\mu}^{G}(b)$ are in the same connected component.
\end{prop}

\begin{para} \label{lemma_b=b_x}
Let $x \in \pi_1(M)$ and let $P_{x}$ be the parabolic subgroup of $M$
defined by $\mu_x$, $M_x$ its Levi subgroup containing $T$
and $N_x$ its unipotent radical. Let $w_x=w_{0,x}w_{0,M}$
where $w_{0,x}$ is the longest Weyl group element in $M_x$ and
where $w_{0,M}$ is the longest Weyl group element in $M$.

Let $N_M$ be the normalizer of $T$ in $M$. Recall that $W_M=N_M(L)/T(L)$ is the Weyl group of $M_L$.
The natural map $N_M(L)\cap K\rightarrow W_M$ is surjective (see for example \cite{HaiRap} Prop. 13). In particular, $w_x$ has a representative $\dot w_x$ in $K$.
Let $b_x=\mu_x(p)\dot w_x$ with $\dot w_x\in K.$ Note that the representatives of superbasic $\sigma$-conjugacy classes chosen in Section \ref{(1.4)} are also of this form.

The elements $b$ and $b_x$ are in the same
$\sigma$-conjugacy class for the group $M$ (i.e., $[b]=[b_x]$ in
$B(M)$). Indeed, as $\kappa_M(b_x)=x=\kappa_M(b)$, in order to
show that the $\sigma$-conjugacy classes of $b$ and $b_x$ agree,
it suffices to show that $b_x$ is basic in $M$. This is shown in \cite{vw}, proof of Proposition 9.17.

For the next two propositions, we assume that $b=b_{x_0}$
for some fixed $x_0\in \bar I_{\mu, b}^{M, G}.$
\end{para}
\begin{prop}\label{pm4}Suppose $(\mu, b)$ is HN-irreducible. Let
\begin{eqnarray*}C:=\left\{\alpha^{\vee}\in X_*(T)|\alpha\in\Phi_N \text{ is adapted, and }\langle\alpha, \mu_{x_0}\rangle<0 \right\}.\end{eqnarray*} Then the sum of the
coroot lattice of $M$ and the $\mathbb{Z}$-lattice generated by
the Galois orbit of the set $C$ is the coroot lattice of $G$.
\end{prop}

\begin{prop}\label{pm3}\label{pm3weak} Let $\Omega\in\Phi_{N,\Gamma}$ be adapted. Suppose that there exists $\alpha\in \Omega$ such that $\langle\alpha, \mu_{x_0}\rangle<0$. Then there exists an $x\in \bar I_{\mu, b}^{M, G}$ and $g_1, g_2\in X_{\mu_{x}}^M(b)$ such that
\begin{itemize}
\item $g_1$ and $g_2$ are in the same connected component of $X_{\mu}^{G}(b)$;
\item $w_M(g_2)-w_M(g_1)=\sum_{\beta\in \Omega}\beta^{\vee}$ in $\pi_1(M)^{\Gamma}$,
\end{itemize}
where $w_M: M(L)\rightarrow\pi_1(M)$ is the Kottwitz homomorphism.
\end{prop}

\begin{para}
Before proving these propositions let us show how they can be used
to prove the main theorems. We first show the following stronger version of Theorem \ref{prop1} (assuming $G^{\ad}$ is simple) which we then use in the proof of Theorem \ref{thmzshk}.
We continue to assume that $b \in M(L)$ is superbasic, and we let $P = NM$ be the parabolic subgroup corresponding to $M.$
As usual, we write $J_b^M$ for the group defined by $b \in M(L),$ so that $J_b^M(F) = J_b(F)\cap M(L).$
\end{para}

\begin{thm}\label{prop1'}
The image of
$$ \pi_0\left( X_{\mu_x}^M(b)\right) \rightarrow \pi_0(X_{\mu}^G(b)) $$
does not depend on the choice of $x \in \bar I_{\mu, b}^{M, G}.$
In particular, for any such $x,$ the map
\begin{equation}\label{eqn:surjmap}
(J_b(F)\cap N(L))\times \pi_0( X_{\mu_x}^M(b))\rightarrow \pi_0(X_{\mu}^G(b))
\end{equation}
is surjective, and the group $J_b(F)\cap P(L)$ acts transitively on $\pi_0(X_{\mu}^G(b))$.
\end{thm}
\begin{proof}
 Let $x_1,\dotsc,x_n$ be as in Proposition \ref{pm1} for a pair $x,x' \in  \bar I_{\mu, b}^{M, G}.$
To prove the first claim of the Theorem, it is enough to show that for every
$g\in X^M_{\mu_{x_i}}(b)$ there is an element $g'\in X^M_{\mu_{x_{i-1}}}(b)$ such that $g,g'$ are in the same connected component in $X^G_{\mu}(b)$.
This follows by applying Proposition \ref{pm2} to each successive pair $(x_{i-1}, x_i)$.

For the second claim note that, by Proposition \ref{(1.5.5)}, each connected component of $X_{\mu}^G(b)$ contains the image of some element of $(J_b(F)\cap N(L))\times\bigsqcup_{x\in \pi_1(M)} X_{\mu_x}^M(b)$. We thus obtain a surjective map
\begin{equation}\label{eqboundpi}
(J_b(F)\cap N(L))\times \bigsqcup_{x\in \bar I_{\mu, b}^{M, G}} \pi_0\left( X_{\mu_x}^M(b)\right)
\rightarrow \pi_0(X_{\mu}^G(b)).
\end{equation}
Hence the first claim implies that (\ref{eqn:surjmap}) is surjective.
Now the final claim follows as $J_b^M(F)$ acts transitively on
$\pi_0(X^M_{\mu_x}(b))$ by  Proposition \ref{prop:generalsuperbasictrans}.
\end{proof}

\begin{proof}[Proof of Theorem \ref{thmzshk}]
We fix some $x\in \bar I_{\mu, b}^{M, G}$ and $g\in X^M_{\mu_x}(b)$. Then left multiplication by $g^{-1}$ induces a bijection $X^M_{\mu_x}(b)(\R)\cong X^M_{\mu_x}(g^{-1}b\sigma(g))(\R)$ for every $k$-algebra $R$ with frame $\R$ and similarly for $G$. In particular, the sets of connected components of the affine Deligne-Lusztig sets for $b$ and $g^{-1}b\sigma(g)$ coincide. Thus we may assume that $b = b_x.$ In particular, $1\in X^M_{\mu_x}(b)$ and therefore $c_{b,\mu_x}^{(M)}=c_{b, \mu}=1$.

By Proposition \ref{prop:generalsuperbasictrans}  we have $J_b^M(F)$-equivariant morphisms
$$\pi_1(M)^{\Gamma}\cong \pi_0(X^M_{\mu_x}(b))\rightarrow \pi_0(X^G_{\mu}(b))\rightarrow \pi_1(G)^{\Gamma}$$ where the composite of all morphisms is induced by the natural projection $\pi_1(M)\rightarrow \pi_1(G).$ By  Lemma \ref{cartesianIV}, and Proposition \ref{pm4},
the kernel of the composition $\pi_1(M)^{\Gamma}\rightarrow \pi_1(G)^{\Gamma}$ is generated by the elements
$\sum_{\beta\in\Omega}\beta^{\vee}$ where $\Omega\in \Phi_{N,\Gamma},$ satisfies $\Omega\cap C \neq \emptyset$ ($C$ defined as in Proposition \ref{pm4}).


We claim that each of the elements $\sum_{\beta\in\Omega}\beta^{\vee}$ with $\Omega\cap C \neq \emptyset$
 is mapped to $1$ by the composite $\pi_1(M)^{\Gamma}\cong \pi_0(X^M_{\mu_x}(b))\rightarrow \pi_0(X^G_{\mu}(b))$. Then the transitivity of the $J_b^M(F)$-action on $\pi_0(X^M_{\mu_x}(b)),$
implies that this composite factors through $\pi_1(G)^{\Gamma}$. Again, by the transitivity of the $J_b^M(F)$-action on $\pi_0(X^M_{\mu_x}(b)),$ our claim follows if we can show that there are elements $g_1,g_2\in X^M_{\mu_x}(b)$ with $w_M(g_2)-w_M(g_1)=\sum_{\beta\in\Omega}\beta^{\vee}$ and such that $g_1,g_2$ are in the same connected component of $X_{\mu}^G(b)$.

To prove this, we apply Proposition \ref{pm3} to $\alpha\in \Omega\cap C$. Let $x'\in  \bar I_{\mu, b}^{M, G}$ and $g_1', g_2'\in X_{\mu_{x'}}^M(b)$ be the elements produced there. As $J_b^M(F)$ acts transitively on $\pi_0(X^M_{\mu_{x'}}(b))$, we can choose a $j_{\Omega}\in J_b^M(F)$ such that $j_{\Omega}g_1'$ is in the connected component of $g_2'$ in $X^M_{\mu_{x'}}(b)$. Then the image of $j_{\Omega}$ in $\pi_1(M)$ is equal to $\sum_{\beta\in\Omega}\beta^{\vee}$. By Theorem \ref{prop1'}, we see that there is a $g_1\in X^M_{\mu_{x}}(b)$ such that $g_1,g_1'$ are in the same connected component of $X_{\mu}^G(b)$. Hence, also $j_{\Omega}g_1'$ and $j_{\Omega}g_1$ are in the same connected component of $X_{\mu}^G(b)$. Altogether we obtain that in $X_{\mu}^G(b)$ the elements $j_{\Omega}g_1$, $j_{\Omega}g_1'$, $g_1'$, $g_1$ are all in the same connected component. As $j_{\Omega}\in M(L)$ we have $j_{\Omega}g_1, g_1\in X^M_{\mu_{x}}(b),$
and $w_M(j_{\Omega}g_1)-w_M(g_1)=\sum_{\beta\in\Omega}\beta^{\vee}$. This shows our claim.

We have shown the existence of the following diagram:
\[\xymatrix{\pi_1(M)^{\Gamma}\cong \pi_0(X^M_{\mu_x}(b))\ar[r]\ar@{->>}[d] &\pi_0(X^G_{\mu}(b))\ar[r] &\pi_1(G)^{\Gamma} \\
\pi_1(G)^{\Gamma}\ar@{^(->}[ur]\ar[urr]_{=} &&
}.\]

It remains to show that $\pi_0(X^M_{\mu_x}(b))\rightarrow\pi_0(X^G_{\mu}(b))$ (or equivalently $\pi_1(G)^{\Gamma}\rightarrow\pi_0(X^G_{\mu}(b))$) is surjective.
By the second claim in Theorem \ref{prop1'},  it suffices to show that for each $j\in J_b(F)\cap N(L)$ and for each $z\in\pi_0(X^M_{\mu_x}(b))$, the two elements $jz$ and $z$ have the same image in $\pi_0(X^G_{\mu}(b))$. As $J_b^M(F)$ acts on $\pi_0(X^M_{\mu_x}(b))$, it is enough to show the same statement for $mjz$ and $mz$ for some $m\in J_b^M(F)$. We choose $m$ such that $mjm^{-1}$ is contained in the stabilizer in $G(L)/G(\mathcal{O}_L)$ of a chosen representative of $z$ in $G(L)$ and such that the image of $m$ in $\pi_1(G)$ is equal to 1. For example, we can choose $m$ to be a sufficiently dominant element in $Z_M(F),$ in the image of $\tilde G(F),$ where $\tilde G$ denotes the simply connected cover of $G^{\der}.$ Then, by what we saw above, the second property of $m$ implies that $mz$ and $z$ are in the same connected component of $X^G_{\mu}(b)$. Hence the same holds for $mjm^{-1}z$ and $mjm^{-1}mz=mjz$. Finally, the first property of $m$ implies that $mjm^{-1}z$ and $z$ are the same element. Altogether, we see that $mjz$ and $mz$ have the same image in $\pi_0(X^G_{\mu}(b))$.
\end{proof}

\begin{para} We now drop the assumption that $G^{\ad}$ is simple. We have the following corollary and generalization of Theorem \ref{thmzshk}.
\end{para}
\begin{cor}\label{thmzshkIII}
Suppose that $(\mu,b)$ is Hodge-Newton irreducible in $G.$ Then $w_G$
induces a bijection $$\pi_0(X_{\mu}(b))\cong c_{b,\mu}\pi_1(G)^{\Gamma}$$
\end{cor}
\begin{proof} Let $\mu^{\ad} \in X_*(T/Z_G)$ and $b^{\ad} \in G^{ad}(L)$ be the images of $\mu$ and $b,$ and let $M \subset G$ be a Levi subgroup.
Since  $\ker(\pi_1(M)_{\Gamma} \rightarrow \pi_1(G)_{\Gamma})$ is torsion free by Lemma \ref{cor:invariants},
it has trivial intersection with the image of $X_*(Z_G)_{\Gamma}$. Using this one sees that
$(\mu,b)$ is HN-irreducible if and only if $(\mu^{\ad},b^{\ad})$ is.

For $(\mu^{\ad},b^{\ad})$ the corollary follows from Theorem \ref{thmzshk} as the set of connected components of affine Deligne-Lusztig varieties for products of groups is the product of the corresponding sets for the individual factors.  And this implies the result for $(\mu,b)$ by Corollary \ref{cartesianIII}.
\end{proof}

\begin{proof}[Proof of Theorem \ref{prop1}] Note that we have already proved Theorem \ref{prop1} in Theorem \ref{prop1'} above when $G^{\ad}$ is simple. We now deduce the
general case from Theorem \ref{thmzshkIII}.

By Propostion \ref{prophndecomp} we may assume that $(\mu,b)$ is HN-indecomposable in $G.$ Let $(\mu^{\ad},b^{\ad})$ be as in the proof of Theorem \ref{thmzshkIII}.
Consider a decomposition $G^{\ad} = G_1 \times G_2,$ and let $(\mu_1,b_1)$ and $(\mu_2,b_2)$ denote the images of $(\mu^{\ad},b^{\ad})$ in $G_1$ and $G_2$ respectively.
By Theorem \ref{thmzshkII}, we may choose $G_1$ and $G_2$ so that $(\mu_1,b_1)$ is HN-irreducible, and $b_2$ is $\sigma$-conjugate to $p^{\mu_2} \in X_*(Z_{G_2}).$

Now suppose that $M \subset G$ is a Levi subgroup and $b \in M(L) \subset G(L)$ is superbasic. As in the proof of Theorem \ref{prop1'}, it suffices to show that
the image of $\pi_0(X^M_{\mu_x}(b)) \rightarrow \pi_0(X^G_\mu(b))$ is independent of $x \in \bar I^{M,G}_{\mu,b}.$ We may assume that $c_{b,\mu_x} = 1.$
Using Proposition \ref{cartesianIII} one sees that it suffices to show that image of $\pi_0(X^M_{\mu_x}(b)) \rightarrow \pi_0(X^G_{\mu^{\ad}}(b^{\ad}))$ is independent of $x.$

By Theorem \ref{thmzshkIII} and Remark \ref{remspecialcase}, the map $M(L) \rightarrow G^{\ad}(L)$  induces a well defined map
$\pi_1(M)^{\Gamma} \rightarrow \pi_1(G_1)^{\Gamma} \times G_2(F)/G_2(\O_F)$
whose image may be identified with that of  $\pi_0(X^M_{\mu_x}(b)) \rightarrow \pi_0(X^{G^{ad}}_{\mu^{\ad}}(b^{\ad})).$
\end{proof}

\begin{para}
Let us consider the case that $G$ is split. Then $\bar I_{\mu, b}^{M, G}$ consists of a single element, so Propositions \ref{pm1} and \ref{pm2} are no longer needed. In the proof of Proposition \ref{pm3} we have to distinguish essentially between all different Dynkin diagrams equipped with the Galois action, and a fixed Galois orbit of simple roots (subject to some restrictions). This case-by-case study is shortened drastically when assuming that $G$ is split (i.e.~that the Galois action is trivial). The reader only interested in this case is referred to \cite{conncomp}, 2.5 where the completely parallel proof for split groups in the function field case is given in less than five pages.

The remainder of this section will be devoted to the proof of the propositions above.
\end{para}

\subsection{Some maximal rank subgroups of $G$}\label{subsection_G_Omega}
In this subsection, we will introduce some subgroups of maximal rank of $G$.
They will be needed in the proofs of Proposition \ref{pm2} and Proposition \ref{pm3} to distinguish several cases.
From now on we again assume that $G^{\ad}$ is simple, and we denote by $T \subset M \subset G$ a standard Levi subgroup over $F.$

We begin with a, probably well-known, fact on root systems with an endomorphism.
\begin{lemma}\label{pm1le}
Let $\Phi$ be a root system with an action by a
finite cyclic group $\Gamma$ such that there exists a basis
$\Delta$ that is stable under this action. Furthermore we assume
that $\Gamma$ acts transitively on the set of connected components
of the Dynkin diagram. Let $\alpha\in \Phi$ and
$\alpha'\in\Gamma\alpha\setminus \{\alpha\}$. Then
$\langle\alpha,(\alpha')^{\vee}\rangle\in\{0,-1\}$. Moreover,
\begin{itemize}
\item If $\langle\alpha,(\alpha')^{\vee}\rangle=-1$ then the root
system is a disjoint union of finitely many copies of root systems
of type $A_n$ for some even $n$.

\item  $\Gamma\alpha$ has at most 3 elements in each connected
component of the Dynkin diagram. If $\Gamma\alpha$ has 3 elements in each connected
component of the Dynkin diagram, then the root system is a
disjoint union of finitely many copies of root systems of type
$D_4$.
\end{itemize}
\end{lemma}

\begin{proof}
The first assertion can be found for example in \cite{Sp}, Lemma 1. The second and third assertions follow from the classification of Dynkin diagrams.
\end{proof}

\begin{ex}\label{exe}
Let $\tau$ be the non-trivial automorphism of the Dynkin diagram of type $A_{2n}$, and of the corresponding root system. Using the standard notation for this root system, we have $\tau(e_i)=e_{2n+2-i}$ Then a root $\alpha=e_i-e_j$ (for $i<j$) satisfies $\langle\alpha,\tau\alpha^{\vee}\rangle=-1$ if and only if $i$ or $j$ is equal to $n+1$.
\end{ex}

\begin{para}\label{pardefmid} Let $\Phi=\Phi(G,T)$ be the root system of $G,$ and $\Delta \subset \Phi$ a $\Gamma$-stable basis of simple roots for $\Phi$ corresponding
to a Borel subgroup $B \subset G.$
If $\phi=\sum_{\alpha\in\Delta}n_{\alpha}\alpha \in X^*(T)$ is an integral sum of roots ($n_{\alpha} \in \mathbb Z$), we define $|\phi|=\sum_{\alpha\in\Delta}|n_{\alpha}|.$
For $\phi \in X_*(T)$ we define $|\phi|$ analogously, using the basis of coroots $\Delta^{\vee}.$
 We will make repeated use of the following two simple Lemmas.
\end{para}

\begin{lemma}\label{remadd}
Let $\gamma,\gamma' \in \Phi$ with $\gamma\neq -\gamma'.$
\begin{enumerate}
\item If $\langle \gamma, (\gamma')^{\vee}\rangle<0$ then $\gamma+ \gamma'$ is a root.
\item If $\langle \gamma,(\gamma')^{\vee}\rangle > 0$ then $\gamma-\gamma'$ is a root.
\item If $\langle \gamma,(\gamma')^{\vee}\rangle > 0$ and $\gamma, \gamma'$ are positive, then
\[\bigg| |\gamma|-|\gamma'|\bigg|=|\gamma-\gamma'|\neq 0. \]
\end{enumerate}
\end{lemma}
\begin{proof} Indeed, $\gamma\neq -\gamma'$ implies that $\langle \gamma,(\gamma')^{\vee}\rangle=-1$ or $\langle \gamma',(\gamma)^{\vee}\rangle=-1.$ By symmetry,
we may assume that the second is true. Then $s_{\gamma}(\gamma')=\gamma + \gamma'$ is a root.
This proves (1) and (2) follows immediately.  To see (3), write $\gamma - \gamma'= \sum_{\alpha\in\Delta}n_{\alpha}\alpha.$
By (2) all the non-zero $n_{\alpha}$ have the same sign, and (3) follows easily.
\end{proof}

\begin{lemma}\label{lem:addroots}
Let $\alpha \in X^*(T)$ be an integral sum of roots. Then $\alpha$ may be written as a sum of roots
$\alpha = \sum_{i \in I} \gamma_i$ such that $\langle \gamma_i, \gamma_j^{\vee} \rangle \geq 0$ for $i,j \in I.$

Moreover, if $\alpha  = \sum _{j \in J} \alpha_j \in X^*(T)$ with $\alpha_j \in \Phi,$ then
we may take each $\gamma_i$ to be a sum of a subset of $\{\alpha_j\}_{j \in J}.$
In particular, if $\alpha$ is positive, then the $\gamma_i$ may be chosen to be positive.
\end{lemma}
\begin{proof} Write $\alpha = \sum_{i \in I} \gamma_i$ such that each $\gamma_i$ is root and  $|I|$ is as small as possible.
If $i,j \in I$ with $\langle \gamma_i, \gamma_j^{\vee} \rangle < 0,$ then $\gamma_i \neq - \gamma_j$ by the minimality of $I.$
Hence $\gamma_i + \gamma_j$ is a root by Lemma \ref{remadd}, which contradicts the minimality of $I.$

If $\alpha =\sum _{j \in J} \alpha_j$ then write $\alpha = \sum_{i \in I} \gamma_i$ such that each $\gamma_i$ is a root
which is a sum of a subset of $\{\alpha_j\}_{j \in J},$ and $|I|$ is as small as possible. The same argument proves
the second claim. If $\alpha$ is positive,  we may take the $\alpha_j$ to be positive simple roots which proves the final claim.
\end{proof}

\begin{definition}Let $\Phi_1$ be a subset of $\Phi$.
\begin{itemize}
\item $\Phi_1$ is said to be \emph{symmetric}  if $\Phi_1=-\Phi_1$ where $-\Phi_1=\{-\alpha| \alpha\in \Phi_1\}$.
\item $\Phi_1$ is said to be \emph{closed} if $\alpha, \beta\in \Phi_1$ with $\alpha+\beta\in \Phi$ implies $\alpha+\beta\in\Phi_1$.
\end{itemize}
\end{definition}

\begin{remark}If $\Phi_1\subset \Phi$ is a closed symmetric subset, then $\Phi_1$ is a root system in the $\mathbb{R}$-vector space generated by $\Phi_1$ (\cite{Bou} Ch VI, no. 1.8, Prop. 23). In this case we also say that $\Phi_1$ is a root system if there is no confusion.
\end{remark}

\begin{para}\label{para:defnGOmega} Now we will define some subgroups of maximal rank of $G$ which will be used in the proof of the main results. For the general theory of these subgroups, we refer to \cite{Hu} \S 2.1 or \cite{SGA3} Expos\'e XXII.

Let $\Delta_M \subset \Delta$ (resp.~ $\Phi_M \subset \Phi$) denote the roots (resp.~simple roots) contained in $\Lie~M.$
The action of $\Gamma=\Gal(\bar k|k)$ on $\Phi$ factors through some finite cyclic quotient of $\Gamma$. Sometimes we also write $\Gamma$ for that finite cyclic quotient if no confusion can arise. The Frobenius automorphism $\sigma$ is a generator of $\Gamma$. Let $\Phi_{N}$ and $\Phi_{N, \Gamma}$ be as in subsection \ref{overview}. For any $\Omega\in \Phi_{N,\Gamma}$, let $\Phi_{\Omega}$ be the smallest symmetric closed subset of $\Phi$ containing $\Phi_M$ and $\Omega$. As $M$ and $\Omega$ are stable under the Galois action, so is $\Phi_{\Omega}$.  We let $G_{\Omega}$ be the subgroup of $G_L$ generated by $T$ and $U_{\alpha}$ for all $\alpha\in \Phi_{\Omega}$.
\end{para}

\begin{prop}\label{prop_group_G_Omega}For any $\Omega\in \Phi_{N,\Gamma},$ the group $G_{\Omega}$ is defined over $F$.
Moreover, it is a reductive subgroup of $G$ with root system $\Phi_{\Omega}$ with respect to the maximal torus $T$.
\end{prop}
\begin{proof}\cite{BT} Theorem 3.13 (compare also \cite{SGA3} Expos\'e 22, Theorem 5.4.7 and Proposition 5.10.1).
\end{proof}

\begin{remark}\label{remgalpha}
Note that in general $G_{\Omega}$ is not a Levi subgroup of $G$. For example, let $G$ have Dynkin diagram of type $C_2$. Then it may happen that $G_{\Omega}$ is generated by $T$ and the root subgroups for all long roots, hence it is of type $A_1\times A_1$. However, for $\mu\in X_*(T)$, $b\in M(L)$ and for any $G_{\Omega}$-dominant $\mu'\in X_*(T)$ with $(\mu')_{G-\dom}=\mu$, we always have a map $X^{G_{\Omega}}_{\mu'}(b)\rightarrow X^{G}_{\mu}(b)$ given by the natural inclusion and inducing a map between the sets of connected components.
\end{remark}

\begin{prop}\label{lemma_newA} 
Suppose that $\Omega$ is adapted, and that all the roots in $G_{\Omega}$ have the same length.
Then $B\cap G_{\Omega}$ is a Borel subgroup of $G_{\Omega}$ with basis $\Delta_M \cup \Omega$.
\end{prop}
\begin{proof}Let $\Phi_{\Omega}^+$ be the set of roots in $G_{\Omega}$ which are positive as roots in $G$ with respect to $B$. Then $\Phi_{\Omega}=\Phi_{\Omega}^{+}\coprod -\Phi_{\Omega}^{+}$ and $\Phi_{\Omega}^{+}$ is the set of roots in $B\cap G_{\Omega}$. It is clear that $B\cap G_{\Omega}$ is a Borel subgroup of $G_{\Omega}$ (as the set of roots in a Borel subgroup is determined by a regular hyperplane in the corresponding root system). By the definition of $\Phi_{\Omega}$, all elements in $\Phi_{\Omega}$ can be written as linear combinations of roots in $\Delta_{\Omega}:= \Delta_M \cup \Omega$. It suffices to show that all elements in $\Omega$ are indecomposable. Moreover, since $\Phi_{\Omega}^{+}$ is stable under the action of $\Gamma$, we only need to show that some $\alpha\in \Omega$ is indecomposable.

Suppose $\alpha \in \Omega$ is adapted and decomposable. Then there exists a root $\alpha_1\in \Phi_{\Omega}$ such that $\alpha_1, \alpha-\alpha_1\in \Phi_\Omega^{+}$. Write

\begin{eqnarray*}\alpha_1&=&\sum_{\beta\in \Delta_{\Omega}}n_\beta\beta=\sum_{\beta\in \Delta_{\alpha_1}^+}n_\beta \beta+\sum_{\beta\in \Delta_{\alpha_1}^-}n_\beta \beta\\
\alpha - \alpha_1&=&\sum_{\beta\in \Delta_{\Omega}}\tilde{n}_\beta\beta=\sum_{\beta\in \Delta_{\alpha-\alpha_1}^+}\tilde{n}_\beta \beta+\sum_{\beta\in \Delta_{\alpha-\alpha_1}^-}\tilde{n}_\beta \beta
\end{eqnarray*}
 where $\Delta_{\alpha_1}^+=\{\beta\in \Delta_{\Omega}| n_{\beta}>0\}$,
 $\Delta_{\alpha_1}^-=\{\beta\in \Delta_{\Omega}| n_{\beta}<0\}$ and $\Delta^{+}_{\alpha-\alpha_1}$, $\Delta^{-}_{\alpha-\alpha_1}$ are defined in the same way.

By Lemma \ref{lem:addroots} we may write $\sum_{\beta\in \Delta_{\alpha_1}^+}n_\beta \beta=\sum_{i\in I}\gamma_i^+$ and $ \sum_{\beta\in \Delta_{\alpha_1}^-}n_\beta \beta=\sum_{j\in J}\gamma_j^{-}$ as sums of roots such that $\gamma_i^+,\gamma_j^- \in \Phi$ and for $i, i'\in I$ and $j, j'\in J$,
\begin{itemize}
\item $\langle\gamma_{i}^+,\gamma_{i'}^{+\vee}\rangle\geq 0$ and $\langle\gamma_{j}^-,\gamma_{j'}^{-\vee}\rangle\geq 0$;
\item $\gamma_{i}^+$ (resp. $\gamma_j^-$) is a linear combination of roots in $\Delta_{\alpha_1}^+$ (resp. $\Delta_{\alpha_1}^-$) with nonnegative (resp. nonpositive) coefficients.
\end{itemize}


By Lemma \ref{pm1le} and the fact that $\alpha$ is $M$-anti-dominant, for distinct roots $\beta,\beta'\in\Delta_{\Omega}$, we have $\langle\beta,\beta'^{\vee}\rangle\leq 0$. Therefore $\langle\gamma_i^+, \gamma_j^{-\vee}\rangle\geq 0$ for any $i\in I$ and $j\in J$. We show that one of the two sets $I$ and $J$ is empty (or equivalently, that one of the two sets $\Delta^+_{\alpha_1}$ and $\Delta^-_{\alpha_1}$ is empty). Suppose that $I$ is non-empty, the other case being analogous. For $i_0\in I$, the inequality
$$\langle\alpha_1, \gamma_{i_0}^{+\vee}\rangle=\langle \sum_{i\in I}\gamma_i^+ +\sum_{j\in J}\gamma_j^-, \gamma_{i_0}^{+\vee}\rangle\geq 2$$ implies that $\alpha_1=\gamma_{i_0}^+$. Hence $J$ is empty and $\alpha-\alpha_1=\alpha-\gamma_{i_0}^+$. Moreover the sets $\Delta^+_{\alpha-\alpha_1}=\{\alpha\}$ and $\Delta^-_{\alpha-\alpha_1}=\Delta^+_{\alpha_1}$ are both non-empty which is impossible according to the same discussion as above, but applied to $\alpha-\alpha_1$.
\end{proof}

\begin{remark} If not all roots in $G_{\Omega}$ have the same length, then in general Proposition \ref{lemma_newA} does not hold. In fact, in this case, the root system generated by the root system of $M$ and the roots in $\Omega$ is not necessarily the root system of $G_{\Omega}$. Here is an example. Consider the split group $G=\GSp_4$. The Dynkin diagram is of type $C_2$ with simple roots $\beta_1=(1, -1)$ and $\beta_2=(0,2)$. Let $M$ be the standard Levi subgroup corresponding to $\beta_1$. And let $\alpha=\beta_1+\beta_2=(1, 1)$. Then the sub root system generated by $\beta_1$ and $\alpha$ is of type $A_1\times A_1$ while $G_{\Omega}=G$ as the commutator $[U_{\alpha}(x), U_{\beta_1}(y)]$ is a non-trivial element of the root subgroup $U_{\alpha+\beta_1}$.
\end{remark}

\begin{prop} Let $\Omega\in \Phi_{N,\Gamma}$ be adapted. Then $M$ is a standard Levi subgroup of $G_{\Omega}$.
\end{prop}
\begin{proof}By the proof of Proposition \ref{lemma_newA}, the basis of $G_{\Omega}$ corresponding to the Borel subgroup $B\cap G_{\Omega}$ is the set of indecomposable elements of $\Phi_{\Omega}^+$. Therefore $M$ is a standard Levi subgroup of $G_{\Omega}$ as any $\beta \in \Delta_M$ is indecomposable in $\Phi_{\Omega}^+$.
\end{proof}

\subsection{Proof of Proposition \ref{pm1}}
\label{subsection_reduction to the distance one case}

From now on let $\Gamma$ be the image of the absolute Galois group of $F$ in the group of automorphisms of the Dynkin diagram of $G$. It is thus a finite and cyclic group, generated by Frobenius. As $G^{\ad}$ is assumed to be simple, $\Gamma$ acts transitively on the set of connected components of the Dynkin diagram.
All assertions involving the Galois action on $X_*(T)$ can then be studied using the induced $\Gamma$-action.

The proof of Proposition \ref{pm1} is divided into two steps: We first reduce the general statement to the special case where $M=T$ is a maximal torus of $G$. More precisely we want to show
\begin{prop}\label{pm11}
Let $x,x'\in \bar I_{\mu, b}^{M, G}$. Then there exists a $w\in W_M$ (the Weyl group of $M$) such that $\mu_{x}=w\mu_{x'}$ in $X_*(T)_{\Gamma}$.
\end{prop}
In particular $\mu_x, w\mu_{x'}\in X_*(T)=\pi_1(T)$ then satisfy that $w\mu_{x'}\in \bar I_{\mu,\mu_x(p)}^{T,G}.$
Furthermore, under the canonical projection $X_*(T)\rightarrow \pi_1(M)$, the set $\bar I_{\mu, \mu_x(p)}^{T,G}$ is mapped to a subset of $\bar I_{\mu, b}^{M, G},$
and $\mu_{x'}$ and $w\mu_{x'}$ have the same image.  Proposition \ref{pm1} is then implied by the following proposition
\begin{prop}\label{pm12}
Let $x,x'\in \bar I^{T,G}_{\mu, b}$ for some $\mu\in X_*(T)$ and $b\in T(L)$. Then there are elements $x_i\in \bar I^{T,G}_{\mu, b}\subset X_*(T)$ for $i=0,\dotsc,m$ for some $m$ such that $x=x_0$, $x'=x_m$ and such that for each $i$, $$x_{i+1}-x_i=\alpha^{\vee}-\alpha'^{\vee}$$ for some roots $\alpha, \alpha'\in\Omega$ with $\Omega\in \Phi_{N,\Gamma}$ (depending on $i$).
\end{prop}
It remains to show these two propositions.

\begin{definition}\label{def_distance}
\begin{enumerate}
\item Let $\phi=\sum_{\alpha\in\Delta}n_{\alpha}\alpha^{\vee}\in X_*(T)$ be an integral sum of coroots.
We write $|\phi|_{\Gamma}=\sum_{\Gamma\alpha\in\Gamma\backslash\Delta}|\sum_{\beta\in \Delta\cap\Gamma\alpha}n_{\beta}|$.
\item For all $\mu_1, \mu_2\in X_*(T)$ having the same image in $\pi_1(G)$ we define $$d(\mu_1,\mu_2)=|\mu_1-\mu_2|,$$ $$d_{\Gamma}(\mu_1,\mu_2)=|\mu_1-\mu_2|_{\Gamma}.$$
\item For $x,x'\in \bar I_{\mu, b}^{M, G}$ let $d(x,x')=d(\mu_x,\mu_{x'})$ and similarly for $d_{\Gamma}$.
\end{enumerate}
\end{definition}
Note that $|x|_{\Gamma}\leq |x|$ (where the latter expression is as in \ref{pardefmid}) with equality if and only if for each Galois orbit $\Gamma\alpha$ all $n_{\beta}$ for $\beta\in\Gamma\alpha$ have the same sign.

As a preparation for the proofs of the propositions we provide several smaller lemmas. For these we consider a root datum $(V,\Phi,V^{\vee},\Phi^{\vee})$ equipped with
an action of $\Gamma,$ together with a $\Gamma$-stable basis of simple roots $\Delta.$ We assume that $\Gamma$ acts transitively on the set of connected
components of the Dynkin diagram of $(V,\Phi,V^{\vee},\Phi^{\vee}).$

\begin{lemma}\label{pm1la}
\begin{enumerate}
\item\label{pm1la1} Let $\sum_{i\in I}\gamma_i^{\vee}=\sum_{j\in J}\lambda_j^{\vee}\neq 0$ be two equal sums of coroots. Then there are an $i\in I$ and $j\in J$ with $\langle \gamma_i,\lambda_j^{\vee}\rangle>0$.
\item\label{pm1la2} Let $\gamma_i^{\vee},\lambda_j^{\vee}$ (for $i\in I, j\in J$) be coroots with $\sum_{i\in I}\gamma_i^{\vee}=\sum_{j\in J}\lambda_j^{\vee}\neq 0$ as elements of $V^{\vee}_{\Gamma}$. Then there are $i\in I$, $j\in J$ and $\tau\in \Gamma$ with $\langle \gamma_i,\tau\lambda_j^{\vee}\rangle>0$.
\end{enumerate}
\end{lemma}

\begin{proof}

By Lemma \ref{lem:addroots}, applied to $\alpha = \sum_{i\in I}\gamma_i,$ we me assume that
$\langle \gamma_{i_1},\gamma_{i_2}^{\vee}\rangle\geq 0$ for all $i_1,i_2\in I$. Then for all $i_0\in I$ we have
$$0<\langle \gamma_{i_0},\sum_{i\in I}\gamma_{i}^{\vee}\rangle=\langle \gamma_{i_0},\sum_{j\in J}\lambda_{j}^{\vee}\rangle.$$ Hence there is a $j\in J$ with $\langle \gamma_{i_0},\lambda_j^{\vee}\rangle>0$.

Let now $\gamma_i$, $\lambda_j$ be as in the second assertion. Then the first assertion holds for
\[\sum_{i\in I}\sum_{\tau\in \Gamma}\tau\gamma_i^{\vee}=\sum_{j\in J}\sum_{\tau\in \Gamma}\tau\lambda_j^{\vee}.\]
Indeed $V^{\vee}$ is a sum of induced $\Gamma$-modules (cf.~ the proof of Lemma \ref{cartesianIV}), so $V^{\vee}_{\Gamma}$ is a free abelian group
and thus these sums are non-zero in $V^{\vee}_{\Gamma}$. This implies the second assertion.
\end{proof}

\begin{lemma}\label{pm1lb}
Let $\sum_{i\in I}\gamma_i^{\vee}\in V^{\vee}$ be a sum of
coroots which maps to $0$ in $V^{\vee}_{\Gamma}$. Then there exist
$\tau_i\in \Gamma$ for all $i\in I$ such that $\sum_{i\in
I}\tau_i(\gamma_i^{\vee})=0\in V^{\vee}$  and such that all
$\tau_i(\gamma_i^{\vee})$ are in the same connected component of
the Dynkin diagram.
\end{lemma}

\begin{proof}
We use induction on $|I|$. Let $I^+$ be the set of $i\in I$ such that $\gamma_i^{\vee}$ is positive and $I^-=I\setminus I^+$. Then
$$\sum_{i\in I}\gamma_i^{\vee}=\sum_{i\in I^+}\gamma_i^{+\vee}-\sum_{i\in I^-}\gamma_i^{-\vee}$$
where $\gamma_i^+=\gamma_i$ and $\gamma_i^-=-\gamma_i$ are all positive. Assume that one of the sums on the right hand side is zero. Then the left hand side lies in the positive resp. the negative cone. As $\Gamma$ fixes the set of simple roots and as $\sum_{i\in I}\gamma_i^{\vee}=0$ in $V^{\vee}_{\Gamma}$, this implies that the other sum is also equal to $0$ (first in $V^{\vee}_{\Gamma}$ but then also in $V^{\vee}$). Furthermore, this only occurs if none of the sums contains any non-zero summand. Thus in this case the assertion of the lemma is trivial. From now on we exclude this case.

Then by Lemma \ref{pm1la} (\ref{pm1la2}) there is a $j_+\in I^+$, a $j_-\in I^-$ and a $\tau\in \Gamma$ such that $\langle \gamma^+_{j_+},\tau \gamma^{-\vee}_{j_-}\rangle >0$. If $\gamma^+_{j_+}=\tau \gamma^{-}_{j_-}$ we have that $\sum_{i\in I}\gamma_i^{\vee}=\sum_{i\in I\setminus\{j_+,j_-\}}\gamma_i^{\vee}=0$ in $V^{\vee}_{\Gamma}$. Then the statement follows by induction. Thus we may assume that $\gamma^+_{j_+}\neq \tau \gamma^{-}_{j_-}$. Then by Lemma \ref{remadd} (applied to $-\gamma^+_{j_+}, \tau \gamma^{-}_{j_-}$) we obtain that $\alpha^{\vee}=\tau \gamma^{-\vee}_{j_-}-\gamma^{+\vee}_{j_+}$ is a coroot. Then
$$\sum_{i\in I}\gamma_i^{\vee}=\sum_{i\in I\setminus\{j_+,j_-\}}\gamma_i^{\vee}-\alpha^{\vee}=0$$
as elements of $V^{\vee}_{\Gamma}$. The assertion follows again by induction.
\end{proof}

\begin{lemma}\label{pm1lc}
Let $v=\sum_{\beta\in\Delta}n_{\beta}\beta^{\vee}\in
V^{\vee}\setminus \{0\}$ with $|v|=|v|_{\Gamma}$. Then there is a
coroot $\alpha^{\vee}$ such that
$|v|=|\alpha^{\vee}|+|v-\alpha^{\vee}|$ and $\langle
\sum_{\tau\in\Gamma}\tau\alpha,v\rangle>0$.
\end{lemma}

\begin{proof}
We first consider the case that the $\Gamma$-action on the Dynkin diagram is trivial. Using Lemma \ref{lem:addroots}, we may write $v$ as a sum of coroots $v=\sum_{i\in I}\gamma_i^{\vee}$ in such a way that $|v|=\sum_i |\gamma_i^{\vee}|$ and $\langle \gamma_i, \gamma_j^{\vee}\rangle \geq 0$ for all $i,j\in I$. Then for all $i$ we have $\langle \gamma_i,v\rangle>0$. Thus each $\alpha=\gamma_i$ is as claimed.

We now assume that $\Gamma$ acts non-trivially on the (connected) Dynkin diagram. This implies that the Dynkin diagram is of type $A$, $D$ or $E_6$, and in particular all roots have equal length. Let $\beta_1,\dotsc,\beta_n$ be representatives of the $\Gamma$-orbits on $\Delta$. Note that $|v|=|v|_{\Gamma}$ implies that $n_{\beta_i},n_{\tau\beta_i}$ have the same sign for all $\tau\in \Gamma$. For $1\leq i\leq n$ let $m_i=|\sum_{\tau\in\Gamma}n_{\tau\beta_i}|=\sum_{\tau\in\Gamma}|n_{\tau\beta_i}|$. By possibly changing the representatives $\beta_i$ we may assume that $n_{\beta_i}\neq 0$ whenever $m_i\neq 0$. We have
$$\left\langle \frac{n_{\beta_i}}{|n_{\beta_i}|}\beta_i,v \right\rangle= 2|n_{\beta_i}|-\sum_{\alpha\in\Delta, \langle \beta_i,\alpha^{\vee}\rangle=-1}\frac{n_{\beta_i}}{|n_{\beta_i}|}n_{\alpha}.$$
For $\alpha\in \Gamma\beta_j$, let $m_{\alpha}=m_j$. Then we obtain
\begin{eqnarray*}
\left\langle \frac{n_{\beta_i}}{|n_{\beta_i}|}\sum_{\tau\in\Gamma}\tau\beta_i,v \right\rangle&=
& 2m_i +\sum_{\alpha\in\Delta, \langle \beta_i,\alpha^{\vee}\rangle=-1}\frac{n_{\beta_i}}{|n_{\beta_i}|}\frac{n_{\alpha}}{|n_{\alpha}|}m_{\alpha}\\
&\geq&2m_i- \sum_{\alpha\in\Delta, \langle \beta_i,\alpha^{\vee}\rangle=-1}m_{\alpha}.
\end{eqnarray*}
If $2m_i- \sum_{\alpha\in\Delta, \langle \beta_i,\alpha^{\vee}\rangle=-1}m_{\alpha}>0$ for some $i\leq n$ the claim is shown. Thus it suffices to show that $$\{m_i\in \mathbb{N}^n\mid 2m_i- \sum_{\alpha\in\Delta, \langle \beta_i,\alpha^{\vee}\rangle=-1}m_{\alpha}\leq 0\}=\{(0,\dotsc,0)\}.$$
This can be done by an easy case-by-case computation considering the different possible types of Dynkin diagrams.
\end{proof}

\begin{lemma}\label{pm1ld}
Let $\mu',\mu''\in X_*(T)$ be minuscule and such that $(\mu')_{G-\dom}=(\mu'')_{G-\dom}$.
Then we have a decomposition
$\mu'-\mu''=\sum_{i\in I}\gamma_i^{\vee}$ as a sum of coroots such that
\begin{itemize}
\item $\langle\gamma_i,\gamma_j^{\vee}\rangle=0$ for $i\neq j.$
\item $d(\mu',\mu'')=\sum_{i\in I}|\gamma_i^{\vee}|.$
\item $\langle \gamma_i,\mu'\rangle=1$, $\langle \gamma_i,\mu''\rangle=-1$ for all $i \in I.$
\item $\mu''=\left(\prod_{i\in I}s_{\gamma_i}\right)\mu'$ where the product does not depend on the order.
\end{itemize}
\end{lemma}

\begin{proof} Applying Lemma \ref{lem:addroots} to $\mu'-\mu''$ written as an integral sum of simple coroots, we see that
 $\mu'-\mu''=\sum_{i\in I}\gamma_i^{\vee}$ where the $\gamma_i$ are roots such that $\gamma_i\neq -\gamma_j$ and $\langle\gamma_i,\gamma_j^{\vee}\rangle \geq 0$
for all $i,j \in I,$ and $d(\mu',\mu'')=\sum_{i\in I}|\gamma_i^{\vee}|.$ Then for all $i_0\in I$,
$$2\leq \langle \gamma_{i_0},\sum_{i\in I}\gamma_i^{\vee}\rangle=\langle \gamma_{i_0},\mu'-\mu''\rangle\leq 2$$ where the last inequality follows from $\mu',\mu''$ minuscule. Thus both inequalities are equalities. We obtain $\langle \gamma_{i_0},\mu'\rangle=1$, $\langle \gamma_{i_0},\mu''\rangle=-1$ and $\langle \gamma_{i_0},\gamma_j^{\vee}\rangle=0$ for all $j\neq i_0$.
\end{proof}

\begin{proof}[Proof of Proposition \ref{pm11}]
Let $x_1,x_2\in \bar I_{\mu, b}^{M, G}$. If $\mu_{x_1}=\mu_{x_2}$ in $X_*(T)_{\Gamma}$, then we are done. So we may assume $\mu_{x_1}\neq \mu_{x_2}$ in $X_*(T)_{\Gamma}$. We use
induction on $d_{\Gamma}(\mu_{x_1},\mu_{x_2})$.  Write
$\mu_{x_2}-\mu_{x_1}=\sum_{i=1}^r \gamma_i^{\vee}$ as in Lemma
\ref{pm1ld}.

Recall that $\Gamma$ acts transitively on the set of connected components of the Dynkin diagram of $G$ as $G^{\ad}$ is simple.  As $\mu_{x_2}=\mu_{x_1}$ in $\pi_1(M)_{\Gamma}$, there exist roots
$(\beta_j)_{j}$ in $M$ such that $\sum_i\gamma_i^{\vee}=\sum_j
\beta_j^{\vee}\neq 0$ as elements of $X_*(T)_{\Gamma}$ and $|\sum_i\gamma_i^{\vee}|_{\Gamma}=\sum_j
|\beta_j^{\vee}|$. Then  $d_{\Gamma}(\mu_{x_1}, \mu_{x_2})=\sum_j
|\beta^{\vee}_j|.$  By Lemma \ref{pm1lb} (applied to $\sum_i \gamma_i^{\vee}-\sum_j
\beta_j^{\vee}$), and after replacing $\beta_j$ by some representative in $\Gamma\beta_j$, there exist $(\tau_i)_{1\leq i\leq r}\in \Gamma^{r}$ such that $\sum_i
\tau_i\gamma_i^{\vee}=\sum_j \beta_j^{\vee}$. As $|\sum_j \beta_j^{\vee}|_{\Gamma}=\sum_j |\beta_j^\vee|$, we have $|\sum_j\beta_j^{\vee}|=|\sum_j \beta_j^{\vee}|_{\Gamma}$. By applying Lemma \ref{pm1lc} to $\sum_j\beta_j^{\vee}$ in the root datum of $M$, there is a coroot
$\alpha^{\vee}$ in $M$  such that $|\sum_{j}\beta_j^{\vee}|=|\alpha^{\vee}|+|\sum_j
\beta_j^{\vee}-\alpha^{\vee}|$ and $\langle\sum_{\tau\in\Gamma}\tau\alpha,\sum_j
\beta_j^{\vee}\rangle>0$.
Thus
\begin{eqnarray*}
\langle \sum_{\tau\in \Gamma}\tau\alpha,\mu_{x_2}-\mu_{x_1}\rangle&=&\langle \sum_{\tau\in \Gamma}\tau\alpha,\sum_i \gamma_i^{\vee}\rangle\\
&=&\langle \sum_{\tau\in \Gamma}\tau\alpha,\sum_i \tau_i\gamma_i^{\vee}\rangle\\
&=&\langle \sum_{\tau\in \Gamma}\tau\alpha,\sum_j \beta_j^{\vee}\rangle\\
&>&0
\end{eqnarray*}
Thus there is a $\tau_0\in\Gamma$ with $\langle \tau_0\alpha,\mu_{x_2}-\mu_{x_1}\rangle>0$. Hence $\langle \tau_0\alpha,\mu_{x_2}\rangle=1$ or $\langle \tau_0\alpha,\mu_{x_1}\rangle=-1$.
In the first case,
\begin{eqnarray*}
d_{\Gamma}(s_{\tau_0\alpha}\mu_{x_2},\mu_{x_1})&=&|\sum_i\beta_i^{\vee}-\tau_0(\alpha)|_{\Gamma} \\
&<&|\sum_i\beta_i^{\vee}|=d_{\Gamma}(\mu_{x_2},\mu_{x_1}),
\end{eqnarray*}
and the statement is shown by induction. In the second case we proceed analogously using $$d_{\Gamma}(\mu_{x_2},s_{\tau_0\alpha}\mu_{x_1})<d_{\Gamma}(\mu_{x_2},\mu_{x_1}).$$
\end{proof}

\begin{proof}[Proof of Proposition \ref{pm12}]
By assumption $\Gamma$ permutes the connected components of the Dynkin diagram of $G$ transitively and each element $\tau\neq 1$ acts non-trivially.

Let $\mu' \mu''\in \bar I^{T,G}_{\mu,b}$. We prove the proposition by induction on $d(\mu',\mu'')$. We assume that $\mu'\neq\mu''$. We write $\mu'-\mu''=\sum_{i}\gamma_i^{\vee}$ as in Lemma \ref{pm1ld}. Gathering the positive resp.~the negative $\gamma_i$ we obtain
$$\mu'-\mu''=\sum_{i\in I}\gamma_i^{+\vee}-\sum_{j\in J}\gamma_j^{-\vee}$$ where now all $\gamma^+_i,\gamma^-_j$ are positive. By Lemma \ref{pm1la} there is a $\gamma_{i_0}^+$, a $\gamma_{j_0}^-$ and a $\tau\in\Gamma$ such that $\langle \tau\gamma_{i_0}^+,\gamma_{j_0}^{-\vee}\rangle >0$. By orthogonality of the $\gamma_i$ we have $\tau\neq 1$. Let $\gamma^+=\gamma_{i_0}^+$ and $\gamma^-=\gamma_{j_0}^-$. Note that $s_{\gamma^-}s_{\gamma^+}\mu'=\mu'-\gamma^{+\vee}+\gamma^{-\vee}.$ If $\gamma^+=\tau\gamma^-$ then $d(s_{\gamma^-}s_{\tau\gamma^-}\mu',\mu'')<d(\mu',\mu'')$ and the induction hypothesis applies. So we may assume that $\gamma^+\neq\tau\gamma^-$. Then $\langle \tau\gamma^+,\gamma^{-\vee}\rangle=1$ or $\langle \gamma^-,\tau\gamma^{+\vee}\rangle=1,$ and by symmetry we may assume that the second equation holds. Let $$\alpha^{\vee}=s_{\gamma^-}(\tau\gamma^{+\vee})=\tau\gamma^{+\vee}-\gamma^{-\vee}.$$ We need to distinguish several cases.

\noindent{\bf Case 1: $\langle \tau\gamma^+,\gamma^{-\vee}\rangle>1$.}

In this case, the root system has roots of different lengths, in particular the connected components do not have non-trivial automorphisms,
and $\langle \tau\gamma^+,\gamma^+ \rangle = 0,$ as $\tau \neq 1.$

We have $\alpha=s_{\gamma^-}(\tau\gamma^+)=\tau\gamma^+-\langle\tau\gamma^+,\gamma^{-\vee}\rangle\gamma^-$. Thus
$$-1\leq \langle\tau\gamma^+,\mu'\rangle=\langle\alpha+\langle\tau\gamma^+,\gamma^{-\vee}\rangle\gamma^-,\mu'\rangle\leq -1.$$
Here the first inequality follows from the fact that $\mu'$ minuscule. For the second we use $\mu'$ minuscule, $\langle\tau\gamma^+,\gamma^{-\vee}\rangle\geq 2$ and $\langle\gamma^-,\mu'\rangle=-1$ (the last equation following from our choice of the $\gamma_i$). Let $\tilde\mu=s_{\tau\gamma^+}s_{\gamma^+}\mu'$. Then $\tilde\mu\in \bar I^{T, G}_{\mu, b}.$
Since $\langle \tau\gamma^+,\gamma^- \rangle > 0,$ we have $|\tau\gamma^{+\vee}-\gamma^{-\vee}| <|\gamma^{+\vee}| + |\gamma^{-\vee}|$ by Lemma \ref{remadd},
which implies that $ d(\tilde\mu,\mu'') <d(\mu',\mu''),$ so the induction hypothesis applies.

\noindent{\bf Case 2: $\langle \tau\gamma^+,\gamma^{-\vee}\rangle=1$.}

By Lemma \ref{pm1le} we have $\langle\gamma^+,\tau\gamma^{+\vee}\rangle, \langle\gamma^-,\tau^{-1}\gamma^{-\vee}\rangle \in \{0,-1\}$.
Since  $\langle \tau\gamma^+,\gamma^- \rangle > 0,$ $\tau\gamma^+$ and $\gamma^-$ are in the same connected component of the Dynkin diagram. Using Lemma \ref{pm1le} again we see that if one of the products above is equal to $-1$, then the Dynkin diagram is of type $A_{n}$ with $n$ even. The explicit description of Example \ref{exe} then shows that $\langle \gamma^+,\gamma^{-\vee}\rangle=0$ implies that at most one of the two products can in fact be equal to $-1$. Hence we have $\langle\gamma^+,\tau\gamma^{+\vee}\rangle = 0$ or $\langle\gamma^-,\tau^{-1}\gamma^{-\vee}\rangle = 0$.

\noindent{\it Case 2.1:  Assume that one of the following conditions is satisfied.

\begin{itemize}
\item $\langle\gamma^-,\tau^{-1}\gamma^{-\vee}\rangle=0$ and  $\langle\tau^{-1}\gamma^{-},\mu'\rangle>0$
\item $\langle\gamma^+,\tau\gamma^{+\vee}\rangle=0$ and  $\langle\tau\gamma^{+},\mu'\rangle<0$
\item $\langle\gamma^-,\tau^{-1}\gamma^{-\vee}\rangle=0$ and  $\langle\tau^{-1}\gamma^{-},\mu''\rangle<0$
\item $\langle\gamma^+,\tau\gamma^{+\vee}\rangle=0$ and  $\langle\tau\gamma^{+},\mu''\rangle>0$
\end{itemize} }

If the first assumption holds let $\tilde\mu=s_{\tau^{-1}\gamma^-}s_{\gamma^-}\mu'$.
Then $\langle \gamma^+,\tau^{-1}\gamma^- \rangle > 0$ implies $d(\tilde\mu,\mu'')<d(\mu',\mu''),$ as above, and the induction hypothesis applies.
The arguments for the other three assumptions are analogous.

\noindent{\it Case 2.2: Assume none of the four possible conditions of case 2.1 are satisfied, and that there is a $\tilde\tau\in \Gamma$ such that $\tilde\tau\alpha$ is not in the same connected component as $\gamma^+$ or $\gamma^-$ and that one of the following conditions holds.
\begin{itemize}
\item $\langle \tilde\tau\alpha,\mu'\rangle=-1$
\item $\langle \tilde\tau\alpha,\mu''\rangle= 1$
\end{itemize}}
Note that by the last assertion of Lemma \ref{remadd} and the assumption of case 2, $|\gamma^{-\vee}|\neq|\gamma^{+\vee}|$. We show that statement for the first of the two alternative assumptions, the other one being analogous, exchanging $\mu'$ and $\mu''$ (and suitable signs). Furthermore we assume that $\langle\gamma^-,\tau^{-1}\gamma^{-\vee}\rangle=0$,
which implies $\langle \tau^{-1}\gamma^-,\mu'\rangle\leq 0,$ as we are excluding Case 2.1.  The alternative case for $\langle\gamma^+,\tau\gamma^{+\vee}\rangle=0$ can be shown by the same argument exchanging $\gamma^-,\gamma^+$ (and suitable signs).

As $\langle \tau^{-1}\gamma^-,\mu'\rangle\leq 0$ we obtain
\begin{equation}\label{eq2123}
\langle\tau^{-1}\alpha,\mu'\rangle=\langle \gamma^+-\langle\tau\gamma^+,\gamma^{-\vee}\rangle\tau^{-1}\gamma^-,\mu'\rangle\geq 1-1\cdot 0=1.
\end{equation}

Let $\tilde\mu=s_{\gamma^-}s_{\gamma^+}s_{\tilde\tau\alpha}\mu'$ and $\tilde{\tilde\mu}=s_{\tilde\tau\alpha}s_{\tau^{-1}\alpha}\mu'$
As $\alpha^{\vee}=\tau\gamma^{+\vee}-\gamma^{-\vee}$, these two coweights (in particular the first) are still in $\bar I_{\mu, b}^{T,G}$.
Notice that
\begin{eqnarray*}\tilde{\mu}-\tilde{\tilde{\mu}}&=&(\mu'+\gamma^{-\vee}-\gamma^{+\vee}+\tilde{\tau}\alpha^{\vee})-
(\mu'+\tilde{\tau}\alpha^{\vee}-\tau^{-1}\alpha^{\vee})\\ &=& \gamma^{-\vee}-\tau^{-1}\gamma^{-\vee}.\end{eqnarray*}
Here we have used that $\tau^{-1}\alpha$ is in the same component as $\gamma^+,$ so that $\langle \tilde\tau\alpha, \tau^{-1}\alpha \rangle = 0.$
Therefore in order to use induction it is enough to show that $d(\tilde\mu,\mu'')<d(\mu',\mu'')$.
We have
$$|\tilde\tau\alpha^{\vee}|=|\alpha^{\vee}|=|\tau\gamma^{+\vee}-\gamma^{-\vee}|=\bigg| |\tau\gamma^{+\vee}|-|\gamma^{-\vee}|\bigg | = \bigg| |\gamma^{+\vee}|-|\gamma^{-\vee}|\bigg |.$$
 Here the second equality follows from Lemma \ref{remadd} as $\tau\gamma^+$ and $\gamma^-$ are both positive roots. Thus $$d(\tilde\mu,\mu'')\leq d(\mu',\mu'')-|\gamma^{-\vee}|-|\gamma^{+\vee}|+|\alpha^{\vee}|<d(\mu',\mu'').$$This implies the assertion for this case.

\noindent{\it Case 2.3: $\langle\gamma^+,\tau\gamma^{+\vee}\rangle =-1$ or $\langle\gamma^-,\tau^{-1}\gamma^{-\vee}\rangle =-1$, but none of the cases considered in 2.1 and 2.2  applies.}

We will show that this case is impossible. We have seen above that then the Dynkin diagram is a union of Dynkin diagrams of type $A_n$ for even $n$. We assume that $\langle\gamma^+,\tau\gamma^{+\vee}\rangle =-1$, the other case being similar. Then $\langle\gamma^-,\tau^{-1}\gamma^{-\vee}\rangle =0$. The roots $\gamma^+,\tau\gamma^+,\gamma^-,\tau\gamma^-$ all lie within one connected component of the Dynkin diagram.

The inequality (\ref{eq2123}) still holds, and $$\langle\alpha,\mu'\rangle= \langle\tau\gamma^+-\gamma^-,\mu'\rangle\geq 0.$$ Furthermore excluding case 2.2 implies that for all
$\tilde\tau \neq \tau, 1$ in $\Gamma,$ we have $\langle \tilde\tau\alpha,\mu'\rangle\geq 0$. A similar argument applies to $\mu',$ and yields
$\langle \tilde\tau\alpha,\mu''\rangle\leq 0.$ Recall that $\mu'=\mu''$ in $X_*(T)_{\Gamma}$. Altogether we obtain
$$0<\langle \sum_{\tilde\tau\in \Gamma}\tilde\tau \alpha,\mu'\rangle=\langle \sum_{\tilde\tau\in \Gamma}\tilde\tau \alpha,\mu''\rangle<0,$$
a contradiction.

\noindent{\it Case 2.4: $\langle\gamma^+,\tau\gamma^{+\vee}\rangle =0=\langle\gamma^-,\tau^{-1}\gamma^{-\vee}\rangle$, but none of the cases in 2.1 and  2.2  apply.}

As before
we have that $\langle \tau\gamma^+,\mu'\rangle\geq 0$
which implies $\langle\alpha,\mu'\rangle=1$ and that $\langle
\tau^{-1}\gamma^-,\mu'\rangle\leq 0,$ which implies
$\langle\tau^{-1}\alpha,\mu'\rangle=1.$ Similarly we obtain
$\langle\alpha,\mu''\rangle=-1 $ and
$\langle\tau^{-1}\alpha,\mu''\rangle=-1.$
 Notice again that
\begin{eqnarray}\label{eqn_sum_Gamma}\langle\sum_{\tilde\tau\in\Gamma}\tilde\tau\alpha,\mu'\rangle=
\langle\sum_{\tilde\tau\in\Gamma}\tilde\tau\alpha,\mu''\rangle.\end{eqnarray}

This equality implies that $\Gamma\alpha$ has at least two elements in each connected
component of the Dynkin diagram. Indeed otherwise we would have $\langle \tilde\tau \alpha, \mu' \rangle \geq 0 \geq \langle \tilde\tau \alpha, \mu'' \rangle$
for $\tilde \tau \neq 1,$ as we are excluding Case 2.2, and $\langle \alpha, \mu' \rangle \geq 1 > -1  \geq \langle \alpha, \mu'' \rangle,$
as we are excluding Case 2.1.  In particular all roots have
equal length. Therefore $\alpha\neq \tau^{-1}\alpha$ since $|\tau^{-1}\gamma^{-}|\neq |\tau\gamma^{+}|,$ as we saw above, and
\[\langle \tau^{-1}\alpha, \alpha^{\vee}\rangle= \langle \gamma^+-\tau^{-1}\gamma^{-}, \tau\gamma^{+\vee}-\gamma^{-\vee}\rangle=-\langle \tau^{-1}\gamma^{-}, \tau\gamma^{+\vee}\rangle \neq 2.\]

As we excluded case 2.2, using again (\ref{eqn_sum_Gamma}), we obtain
 $\tau_1,\tau_2\in\Gamma$ such that $\tau_1\alpha\neq \tau_2\alpha$, and $\tau_1\alpha,\tau_2\alpha$ are each in the connected component of
$\alpha$ or $\tau^{-1}\alpha$ with one of the following two conditions satisfied
 \begin{itemize}
 \item $\langle \tau_1\alpha,\mu'\rangle=-1$ and $\langle \tau_2\alpha,\mu'\rangle=-1$.

 \item $\langle \tau_1\alpha,\mu''\rangle=1$ and $\langle \tau_2\alpha,\mu''\rangle=1$.
\end{itemize}
Assume that the first of the above two alternative conditions holds, the other one being analogous.
From our calculation of the products with $\mu',\mu''$ above we
see that $\tau_i\alpha\neq \alpha,\tau^{-1}\alpha$ for $i=1,2$.
Moreover $\alpha$ and $\tau^{-1}\alpha$ cannot be
in the same connected component of the Dynkin diagram, otherwise
the four roots $\alpha$, $\tau^{-1}\alpha$,
$\alpha_1:=\tau_1\alpha$ and $\alpha_2:=\tau_2\alpha$ are in the
same connected component which is impossible according to Lemma \ref{pm1le}.

\noindent{\it Case 2.4.1: $\Gamma\alpha$ has 2 elements in each
connected component.}

We assume that $\alpha_1$ is in the same connected component as $\alpha$ (and thus as $\gamma^-$), the other case being analogous. Then $\alpha_2=\tau^{-1}\alpha_1$. We want to show that $\langle\alpha_1,\gamma^{-\vee}\rangle\geq 0$. As $\langle\alpha_1, \mu'\rangle=-1$ and $\langle \gamma^{-}, \mu'\rangle = -1$, we have $\alpha_1\neq -\gamma^{-}$.
Hence if $\langle\alpha_1, \gamma^{-\vee}\rangle < 0$ then  $\gamma^-+\alpha_1$ is a root by Lemma \ref{remadd}.
Since $\langle\gamma^-+\alpha_1,\mu'\rangle=-2,$ this contracts the condition that $\mu'$ is minuscule.

In the same way one shows that $\langle\alpha_2,\gamma^{+\vee}\rangle\leq 0$. On the other hand,
$$0\geq \langle\alpha_2,\gamma^{+\vee}\rangle =\langle\alpha_1,\tau\gamma^{+\vee}\rangle=\langle\alpha_1,\alpha^\vee\rangle+\langle\alpha_1,\gamma^{-\vee}\rangle$$ and by Lemma \ref{pm1le} the first of the summands on the right hand side is $0$ or $-1$. Thus $\langle\alpha_2,\gamma^{+\vee}\rangle=0$ or $\langle\alpha_1,\gamma^{-\vee}\rangle=0$. We consider the second case, the other being analogous.
Let $\tilde\mu=s_{\alpha_1}s_{\gamma^-}s_{\gamma^+}\mu'$ and
$\tilde{\tilde\mu}=s_{\alpha_1}s_{\tau^{-1}\alpha}\mu'$. Then $\tilde\mu, \tilde{\tilde\mu}\in \bar I^{T, G}_{\mu, b}$. Moreover, since
we are excluding Case 2.1, $\langle \tau^{-1}\alpha, \mu' \rangle = 1,$ so
$$\tilde{\mu}-\tilde{\tilde\mu}=-\gamma^{+\vee}+\gamma^{-\vee}+\tau^{-1}\alpha^{\vee}=\gamma^{-\vee}-\tau^{-1}\gamma^{-\vee},$$
and $d(\tilde\mu,\mu'')<d(\mu',\mu''),$ as in Case 2.2. Thus the assertion follows
by induction.

\noindent{\it Case 2.4.2: $\Gamma\alpha$ has 3 elements in each
connected component.}

In this case, the Dynkin diagram is of type $D_4$ by Lemma
\ref{pm1le}. Suppose $\alpha_1:=\tau_1\alpha$ is in the same
connected component as $\alpha$. Then $\langle\tau_1\rangle\subset
\Gamma$ is the stabilizer of each connected component of the
Dynkin diagram. Let $\{\beta_i\}_{0\leq i\leq 3}$ be the basis of
the connected component of the root system containing $\alpha$
such that $\tau_1\beta_0=\beta_0$ and $\tau_1$ acts transitively
on $\{\beta_i\}_{1\leq i\leq 3}$. We may suppose that $\alpha$ is
positive. Then $\alpha$ is of the form $\beta_i$ or
$\beta_{i}+\beta_0$ or $\beta_i+\beta_0+\beta_j$ with $1\leq i\neq
j\leq 3$, and therefore $\alpha_1-\alpha=\beta_{i_0}-\beta_{j_0}$
for some $1\leq i\neq j\leq 3$. As $\langle\alpha_1-\alpha,
\mu'\rangle=-2$, we have $\langle\beta_{i_0},\mu'\rangle=-1$,
$\langle\beta_{j_0}, \mu'\rangle=1$, and for $0\leq k\leq 3$,
$k\neq i_0, j_0$, $\langle\beta_{k}, \mu'\rangle=0$ since $\mu'$
is minuscule. Thus $\langle\gamma^{-},\mu'\rangle=-1$ implies that
$\beta_{i_0}^{\vee}\preceq \gamma^{-\vee}$.

On the other hand, notice that
$$\langle \tau_1^2(\alpha)-\tau_1(\alpha), \mu'\rangle=\langle
\tau_1(\beta_{i_0})-\tau_1(\beta_{j_0}), \mu'\rangle\in \{\pm
1\}.$$ This implies that $\langle \tau_1^2(\alpha), \mu'\rangle=0$
and $\tau_2(\alpha)$ is not in the same connected component as
$\alpha$, so it is in the same connected component as
$\tau^{-1}\alpha$.  By applying the same method as above to the
connected component of the Dynkin diagram of $\tau^{-1}\alpha$, we
can find $1\leq j'_0\leq 3$ such that
$\langle\tau^{-1}\beta_{j'_0}, \mu'\rangle=1$ and
$\tau^{-1}\beta_{j'_0}^{\vee}\preceq \gamma^{+\vee}$. Let
$\tilde{\mu}:=s_{\tau^{-1}\beta_{j'_0}}s_{\beta_{i_0}}\mu'$, then
$d(\tilde{\mu}, \mu'')< d(\mu', \mu'')$ and the induction
hypothesis applies.
\end{proof}

\subsection{Immediate distance case}\label{sec63}

\label{subsection_affine_lines_immdiate_distance} In Proposition \ref{pm1}, for any two elements $x, x'\in \bar I^{M, G}_{\mu, b}$, we have found a series of elements
$x_1,\cdots, x_r\in \bar I^{M, G}_{\mu, b}$ with $x=x_1$, $x'=x_r$ such that the difference of each two successive elements in the series is of the form $\alpha^\vee-\sigma^{m}(\alpha^\vee)$ in $\pi_1(M),$ where $\alpha$ is a root in $N.$ In this subsection, we want to add some elements in that series such the each pair of successive elements in the enlarged series has ``minimal distance" in a sense that we will define below. Such pairs will be called in immediate distance (cf. Definition \ref{def_immediate_distance}).
\bigskip

We now return to the assumptions of \ref{para:stdsetup}, so that $G^{\ad}$ is simple, $M \subset G$ is a standard Levi, and $b \in M(L)$ is superbasic.
For any $\Omega\in\Phi_{N, \Gamma}$, we recall that the subgroup $G_{\Omega}$ of $G$ is defined in \ref{para:defnGOmega}.
We first provide several useful lemmas that will be used in the sequel.

\begin{lemma} \label{lemma_modify_alpha} Let $\alpha\in\Phi_N$ be a (positive) root, and
let $\Omega=\Gamma\alpha$. There exists an adapted root $\alpha'$ in $G_{\Omega}$ such that $\alpha^{\vee}=\alpha'^{\vee}$ in $\pi_1(M)$.
\end{lemma}
\begin{proof} Let $\alpha_1$ be the $M$-anti-dominant representative in
$W_M\alpha$. If $\alpha_1$ is adapted, then let $\alpha'=\alpha_1$ and we are done. If $\alpha_1$ is not adapted, then there is a root $\beta$ in $M$ with $\langle\beta,
\alpha_1^{\vee}\rangle< -1$. This means that the irreducible sub-root system (corresponding to a connected component of the Dynkin diagram)  of $G_{\Omega}$ which contains $\alpha_1$ and $\beta$ has roots of different length, and $\beta$ is a long root while $\alpha_1$ is a short one. Let $\alpha'$ be the $M$-anti-dominant
representative in $W_M(\alpha_1^{\vee}+\beta^{\vee})^{\vee}$. By definition, $\alpha'$ is a long root and thus it is adapted.
\end{proof}


\begin{definition}\label{defn:type} Let $\Omega\in \Phi_{N,\Gamma}.$ Then $\Omega$ is \emph{of type I (resp II, resp
III)} if any irreducible sub-root system (corresponding to a connected component of the Dynkin diagram) of $G_{\Omega}$ which contains some element of $\Omega$ has 1 (resp.~2, resp.~3) root(s) in $\Omega$.
\end{definition}

\begin{remark}\label{rem_after_newA} Suppose that $\Omega$ is adapted and that $\Gamma$ acts transitively on the connected components of the Dynkin diagram of $G_{\Omega}$. If $\Omega$ is of type II or III, then all roots in $G_{\Omega}$ have the same length and Proposition \ref{lemma_newA} applies. In particular all the roots in $\Omega$ are simple roots in $G_{\Omega}$ for the Borel subgroup $B\cap G_{\Omega}$.  Moreover, the fact that the stabilizer in $\Gamma$ of each connected component of the Dynkin diagram of $M$ acts trivially on that component (compare Lemma \ref{(1.2.6)}) implies the following additional conditions on $\Omega$. If $\Omega$ is of
type III, then the Dynkin diagram of $G_{\Omega}$ is of type $D_4$. If $\Omega$ is of type II, then only the following cases may occur. For type $A_n$ with $n$ even, $\Omega$ consists of the two middle simple roots in each connected component of the Dynkin diagram. For type $A_n$ with $n$ is odd, it consists of the two neighbors of the middle simple root in each connected component. For type $D_n$ the intersection of $\Omega$ with any connected component consists of two of the roots with only one neighbor, which are exchanged by some element of $\Gamma$. For type $E_6$, $\Omega$ consists  of the two simple roots having two neighbors in each connected component.
\end{remark}

\begin{para}
Recall that for $x\in \pi_1(M),$ $\mu_x$ denotes the unique
$M$-dominant, $M$-minuscule cocharacter with image $x.$
As in \ref{lemma_b=b_x}, we write $M_x \subset M$ for the centralizer of $\mu_x,$
and we set $w_x=w_{0,x}w_{0,M}$
where $w_{0,x}$ is the longest Weyl group element in $M_x$ and
where $w_{0,M}$ is the longest Weyl group element in $M$.
\end{para}

\begin{lemma}\label{lemma_C}Suppose that $\Omega\in \Phi_{N,\Gamma}$ is adapted, and $x \in \pi_1(M).$ Then \begin{enumerate}
\item $w_x^{-1}(\mu_x)=w_{0,M}(\mu_x)$.

\item $(\mu_x+\gamma^{\vee})_{M-\dom}=\mu_{x+\gamma^{\vee}}$ for
$\gamma\in \Omega$.

\item $\mu_x-w_x\gamma^{\vee}=\mu_{x-\gamma^{\vee}}$ for
$\gamma\in\Omega$.

\end{enumerate}
\end{lemma}
\begin{proof}(1) This follows as $\mu_x$ is by definition invariant under conjugation by $w_{0,x}$ and as $w_{0,M}=w_{0,M}^{-1}$.

(2) It suffices to show that $\mu_x+\gamma^{\vee}$ is $M$-minuscule. For positive roots $\beta$ in $M$ we have $\langle\beta, \mu_x\rangle\in\{0, 1\}$. As $\gamma$ is adapted, we have $\langle\beta, \gamma^{\vee}\rangle\in\{0, -1\}$. Therefore $\langle\beta, \mu_x+\gamma^{\vee}\rangle \in\{-1,0,1\}$, thus $\mu_x+\gamma^{\vee}$ is $M$-minuscule.

(3) It is enough to show that the element on the left hand side is $M$-dominant and $M$-minuscule. To compute the pairing with all simple roots $\beta$ of $M$, recall that by definition $$\langle\beta,\mu_x\rangle=\begin{cases}
1&\text{if }\beta\text{ is a simple root in }N_x\\
0&\text{otherwise.}
\end{cases}$$
On the other hand, $\langle\beta,w_x\gamma^{\vee}\rangle\in\{-1,0,1\}$ as $\gamma$ is adapted. Notice that
$w_x^{-1}\beta=w_{0, M}w_{0,x}\beta$. If $\beta$ is a simple root in $M_x$, then $w_x^{-1}\beta$ is a simple root of $w_x^{-1}M_xw_x$ with respect to the Borel $B\cap w_x^{-1}M_xw_x$. In particular it is a simple root of $M$. If $\beta$ is a simple root in $N_x$, then $-w_x^{-1}\beta$ is a highest root in $M$.
Therefore $\langle\beta,w_x\gamma^{\vee}\rangle=\langle w_x^{-1}\beta,\gamma^{\vee}\rangle=1$ if and only if $\beta$ is a simple root in $N_x$ and $\gamma^\vee$ is not central on the connected component of the Dynkin diagram of $M$ containing $\beta$.  Moreover $\langle\beta,w_x\gamma^{\vee}\rangle=-1$ occurs for at most one $\beta$ in each connected component of the Dynkin diagram of $M$. This follows from the fact that $\langle\beta,\gamma^{\vee}\rangle=-1$ for at most one simple root in each connected component of the Dynkin diagram of $M$.
\end{proof}

\begin{lemma}\label{lemma_newB} Suppose $x, x'\in \bar I_{\mu, b}^{M, G}$ such that
$x'-x=\alpha^{\vee}-\tau(\alpha)^{\vee}$ with $\alpha$ an adapted root in $N,$ such that $\alpha \neq \tau(\alpha).$
Then we have $(\mu_{x+\alpha^{\vee}})_{G-\dom}=\mu$ and  $(\mu_{x-\tau(\alpha^{\vee})})_{G-\dom} = \mu.$ Moreover,
$(\mu_x+\alpha^{\vee})_{G-\dom}=\mu$,
$(\mu_x-w_x\tau(\alpha)^{\vee})_{G-\dom}=\mu$ and $(\mu_x+\alpha^{\vee}-w_x\tau(\alpha)^{\vee})_{G-\dom}=\mu$.
\end{lemma}

\begin{proof} Write
$\mu_{x'}-\mu_x=\alpha^{\vee}-\tau(\alpha)^{\vee}+\sum_{\beta} n_{\beta}\beta^{\vee}$ where $\beta$ runs
over simple coroots of $M,$ and $n_{\beta} \in \mathbb Z.$ Let $\Delta^+$ (resp.~$\Delta^-$)
denote the set of $\beta$ with $n_{\beta} > 0$ (resp. $n_{\beta} < 0).$ Note that $\langle \alpha, \tau(\alpha)^{\vee} \rangle \leq 0$
by Lemma \ref{pm1le}, and $\langle \beta, \alpha^{\vee} \rangle, \langle \beta, \tau(\alpha)^{\vee} \rangle \leq 0$ for any
$\beta$ since $\alpha$ is adapted. Hence, if $\gamma_1^{\vee}, \gamma_2^{\vee}$ are coroots of the
form $\gamma_1^{\vee} = \alpha^\vee+\sum_{\beta \in \Delta^+} m_{\beta}\beta^\vee,$ $\gamma_2^{\vee} = \tau(\alpha)^\vee+\sum_{\beta \in \Delta^-} m_{\beta}\beta^\vee$
with $m_{\beta}$ positive integers, then $\langle \gamma_1, \gamma_2^{\vee} \rangle \leq 0.$
It follows by the proof of Lemma \ref{pm1ld}, that we can write
\begin{itemize}
\item $\mu_{x'}-\mu_{x}=\sum_{i\in I}\gamma_i^{\vee}$ as in
Lemma \ref{pm1ld};

\item there exists $i_1, i_2\in I$ with
$\gamma_{i_1}^{\vee}=\alpha^{\vee}$ in $\pi_1(M)$,
$\gamma_{i_2}^{\vee}=-\tau(\alpha)^{\vee}$ in $\pi_1(M)$;

\item for $\forall i\in I\backslash \{i_1, i_2\}$,
$\gamma_i^{\vee}=0$ in $\pi_1(M)$.
\end{itemize}
Thus $\mu_x$, $s_{\gamma_{i_1}}(\mu_x) = \mu_x+\gamma_{i_1}^{\vee}$ and $s_{\gamma_{i_2}}(\mu_x) \mu_x+\gamma_{i_2}^{\vee}$
are in the same Weyl group orbit. In particular,
$\mu_x+\gamma_{i_1}^{\vee}$ and $\mu_x+\gamma_{i_2}^{\vee}$ are
$M$-minuscule. So
$(\mu_x+\gamma_{i_1}^{\vee})_{M-\dom}=\mu_{x+\alpha^{\vee}}$ and
$(\mu_x+\gamma_{i_2}^{\vee})_{M-\dom}=\mu_{x-\tau(\alpha)^{\vee}}$. It
follows that $(\mu_{x+\alpha^{\vee}})_{G-\dom}=\mu$ and
$(\mu_{x-\tau(\alpha^{\vee})})_{G-\dom}=\mu$. The equalities
$(\mu_x+\alpha^{\vee})_{G-\dom}=\mu=(\mu_x-w_x\tau(\alpha)^{\vee})_{G-\dom}$ follow directly from
 Lemma \ref{lemma_C}, which also implies the last equality as
$$(\mu_x+\alpha^{\vee}-w_x\tau(\alpha)^{\vee})_{M-\dom} = (\mu_{x-\tau(\alpha)^{\vee}}+\alpha^{\vee})_{M-\dom}
= (\mu_{x+\alpha^{\vee}-\tau(\alpha)^{\vee}})_{M-\dom}.$$
\end{proof}

\begin{lemma}\label{lemma_w_x_alpha minimal} Suppose that $\alpha$ is
  an adapted root in $N$. Then for all $w \in W_M$ we have $\langle w\alpha, \mu_x\rangle
\leq \langle w_x\alpha, \mu_x\rangle$ and the root $w_x\alpha$
is the unique minimal element in the set
\begin{eqnarray*}\bigg\{w\alpha\mid w\in W_M, \langle w\alpha, \mu_x\rangle
= \langle w_x\alpha, \mu_x\rangle\bigg\}\end{eqnarray*} for the order $\preceq \, = \, \preceq_M$.
\end{lemma}

\begin{proof} Since  $w_{0,M}\mu_x$ is $M$-anti-dominant, for $w \in W_M,$ $w_{0,M}\mu_x \preceq w\mu,$
and hence $\langle \alpha, w_{0,M}\mu_x \rangle  \geq \langle \alpha, w\mu_x \rangle,$ as $\alpha$ is adapted.
By Lemma \ref{lemma_C} (1), this implies $\langle \alpha, w_x^{-1}\mu_x \rangle  \geq \langle \alpha, w\mu_x \rangle.$
Hence
\begin{eqnarray*}
I_{w_x\alpha}&:=&\bigg\{w\alpha\mid w\in W_M, \langle w\alpha, \mu_x\rangle
\geq \langle w_x\alpha, \mu_x\rangle\bigg\} \\
&=&\bigg\{w\alpha\mid w\in W_M, \langle w\alpha, \mu_x\rangle
= \langle w_x\alpha, \mu_x\rangle\bigg\}\\
&=&\bigg\{w\alpha\mid w\in W_M, \langle w\alpha, \mu_x\rangle
=\mathrm{sup}_{w'\in W_M} \langle w'\alpha, \mu_x\rangle\bigg\}.
\end{eqnarray*}

We first prove that $w_x\alpha$ is a minimal element in the set $I_{w_x\alpha}$ by reduction to absurdity.
Suppose that $w_x\alpha$ is not a minimal element. Then there
exists $w'\in W_M$ such that $w'\alpha\preceq w_x\alpha$ with
$w'\alpha\neq w\alpha$ and $\langle w_x\alpha, \mu_x\rangle= \langle w'\alpha, \mu_x\rangle$.
As $w'\alpha$ and $w_{x}\alpha$ are in the same Weyl group orbit, they have the same length, so
$\langle w_{x}\alpha - w'\alpha, w_{x}\alpha^{\vee} \rangle = 2 - \langle w'\alpha, w_{x}\alpha^{\vee} \rangle \geq 1.$
Hence, there exists a positive simple root $\beta$ in $M$ such that $\langle
w_x\alpha, \beta^\vee\rangle >0$ and $w'\alpha+\beta\preceq w_x\alpha$. Moreover
\[ \langle w'\alpha, \mu_x\rangle=\langle w_x\alpha, \mu_x\rangle
\geq \langle w'\alpha+\beta, \mu_x\rangle \geq \langle w'\alpha, \mu_x\rangle\] implies that
$\langle \beta,
\mu_x\rangle=0$. Then $\beta$ is a root in $M_x$. As the groups
$M_x$ and $M$ are both of type A, the root $-w_{0, x}(\beta)$ is
simple in $M_x$ and $w_x^{-1}(\beta)=w_{0,
M}w_{0,x}(\beta)$ is a simple root in $M$. Therefore
$\langle \alpha, w_x^{-1}(\beta^\vee)\rangle\leq 0$ as $\alpha$ is
$M$-anti-dominant. This is a contradiction to $\langle w_x\alpha, \beta^{\vee}\rangle>0$.

Now we show that $w_x\alpha$ is the unique minimal element. By Lemma \ref{(1.2.6)}, the Dynkin diagram of $M$ is of type $A$. As we can work separately with each connected component of the Dynkin diagram of $M$, we may suppose without loss of generality that the Dynkin diagram of $M$ is connected with simple roots $\beta_1,\cdots, \beta_m$ with $\langle \beta_{i}, \beta_{i+1}^{\vee}\rangle=-1$ for $1\leq i\leq m-1$. If for all $1\leq i\leq m$, $\langle\beta_i, \alpha^{\vee}\rangle =0$, then the set $I_{w_x\alpha}$ contains a single element $\alpha$ and we are done. Otherwise $\langle\sum_{i=1}^m\beta_i, \alpha^{\vee}\rangle = -1$, and hence there exists a unique $1\leq i_0\leq m$ with $\langle \beta_{i_0}, \alpha^{\vee}\rangle=-1$. If $\langle \beta_i, \mu_x\rangle =0$ for all $1\leq i\leq m$, then $I_{w_x\alpha}=W_M \alpha$ and $w_x\alpha = \alpha$ is the unique minimal element as it is $M$-anti-dominant. It remains the case when there exists $1\leq j_0\leq m$ such that $\langle \beta_{j_0}, \mu_x\rangle=1$. We may assume that $j_0\leq i_0$, the other case being analogous. Then
\[I_{w_x\alpha}=\{s_{\beta_k}s_{\beta_{k+1}}\cdots s_{\beta_{i_0}}\alpha \mid 1\leq k\leq j_0\}.\]
This is a totally ordered set and therefore has a unique minimal element.
\end{proof}

\begin{definition}\label{def_immediate_distance}Let $x_1, x_2\in \bar I_{\mu, b}^{M, G}$
 such that $x_2-x_1=\alpha^{\vee}-\sigma^m(\alpha^{\vee})$ in $\pi_1(M)$ with
$\alpha$ a positive root in $N$ and $m\in\mathbb{N}$. By Lemma \ref{lemma_modify_alpha}, we may assume that $\alpha$ is adapted. Let $\Omega:=\Gamma\alpha$ and
$\alpha^i:=\sigma^i(\alpha)$ for $i\in\mathbb{N}$. The distance from $x_1$ to $x_2$ is called
\emph{immediate} if the following two conditions are satisfied.
 \begin{enumerate}
\item if $\Omega$ is of type I (resp. II, resp. III), we require
that $0<m <|\Omega|$ (resp. $0< m\leq \frac{|\Omega|}{2}$, resp.
$0< m< \frac{2|\Omega|}{3}$).

\item $x_1+\alpha^{i\vee}-\alpha^{m\vee}\notin \bar I_{\mu, b}^{M, G}$
and $x_1+\alpha^{\vee}-\alpha^{i\vee}\notin \bar I_{\mu, b}^{M, G}$ for
all $0< i < m$.
\end{enumerate}

We write $x_1\rightarrow x_2$ when the distance from $x_1$ to $x_2$ is immediate.
\end{definition}

\begin{remark}\label{rem_D4_2b_explicite_comp}Using the same notations as in the above definition we assume that $\Omega$ is of type III
 and $d <m <2d$ with $d=\frac{|\Omega|}{3}$. Suppose that $\Gamma$ acts transitively on the connected components of the Dynkin diagram of $G_{\Omega}$. By Proposition \ref{lemma_newA}, let  $\{(\beta^i)_{0\leq i\leq d-1},
(\alpha^i)_{0\leq i\leq  3d-1}\}$ be the basis of $G_{\Omega}$ with
$\beta^i$  the common neighbor of $\alpha^i,
 \alpha^{i+d}$ and $\alpha^{i+2d}$. As $x_2=x_1+\alpha^{\vee}-\alpha^{m\vee}\in \bar I_{\mu, b}^{M, G}$, $\langle \alpha, \mu_{x_1}\rangle=-1$ and $\langle\alpha^m+\beta^{m-d}, \mu_{x_1}\rangle=1$.
Similarly $x_1+\alpha^{\vee}-\alpha^{i\vee}\notin \bar I_{\mu, b}^{M, G}$ and $x_1+\alpha^{i\vee}-\alpha^{m\vee}\notin \bar I_{\mu, b}^{M, G}$ for $i=m-d, d$ imply that $\langle\alpha^d,
\mu_{x_1}\rangle=0$ and $\langle \alpha^{m-d}+\beta^{m-d},
\mu_{x_1}\rangle=0$. Therefore the vector $(\langle\beta^0,
\mu_{x_1}\rangle, \langle\alpha^0, \mu_{x_1}\rangle,
 \langle\alpha^d, \mu_{x_1}\rangle)$ is equal either to $(0, -1, 0)$ or to $(1, -1,
0)$, and the vector $(\langle\beta^{m-d}, \mu_{x_1}\rangle,
\langle\alpha^{m-d},\mu_{x_1}\rangle, \langle\alpha^m,
\mu_{x_1}\rangle)$ is equal either to $(1, -1, 0)$ or to
$(0, 0, 1)$.
\end{remark}

\begin{prop}\label{prop_reduction_immediate}For $x, x'\in \bar I_{\mu, b}^{M, G}$, there exists $n\in\mathbb{N}$ and a
series of elements $x_1,\cdots, x_{n+1}\in \bar I_{\mu, b}^{M, G}$ such
that $x_1=x$, $x_{n+1}=x'$ and for $i=1,\cdots, n$, either $x_i\rightarrow
x_{i+1}$ or $x_{i+1}\rightarrow x_i$.
\end{prop}

\begin{proof}By Proposition \ref{pm1}, and Lemma \ref{lemma_modify_alpha}, we may assume that
$x'-x=\alpha^{\vee}-\sigma^m(\alpha^{\vee})$ with $\alpha$ an adapted, positive root in $N$. Then
$x-x'=\alpha'^{\vee}-\sigma^{m'}(\alpha'^{\vee})$ with
$m'=|\Omega|-m$ and $\alpha'=\sigma^m(\alpha)$. We may assume that $m\leq \frac{|\Omega|}{2}$ as otherwise, we can exchange $x$ and $x'$. Then the first condition of
Definition \ref{def_immediate_distance} is already satisfied.

We use induction on $m$ to prove that we can achieve that the second condition of Definition \ref{def_immediate_distance} holds . Suppose the condition
is not satisfied for the pair $(x, x')$. Then there exists some $1\leq i<m$, such that
$x+\alpha^{i\vee}-\alpha^{m\vee}\in \bar I_{\mu, b}^{M, G}$ or
$x+\alpha^{\vee}-\alpha^{i\vee}\in \bar I_{\mu, b}^{M, G}$. We may assume that
$x+\alpha^{i\vee}-\alpha^{m\vee}\in \bar I_{\mu, b}^{M, G}$, the other case being
analogous. Then we can apply the induction hypothesis to
the couple $(x, x+\alpha^{i\vee}-\alpha^{m\vee})$ and the pair
$(x+\alpha^{i\vee}-\alpha^{m\vee}, x')$.
\end{proof}

\subsection{Proof of Proposition \ref{pm2}}\label{sec65}
In this subsection, we will construct affine lines in the immediate distance case to prove
Proposition \ref{pm2}. For any $x\in \bar I^{M, G}_{\mu, b}$, let $\mu_x$, $w_x$ be as above. In the following, two roots in $G$ which are in the same irreducible sub-root system corresponding to a connected component of Dynkin diagram of $G$ will also be said to be in the same connected component of the Dynkin diagram of $G.$
We use the analogous expression for the roots in other groups.

We need one more lemma.

\begin{lemma}\label{lemma_D}Let $x\in \bar I_{\mu, b}^{M, G}$ and let $\alpha$ be a positive root in $N$.
Suppose
$$\text{$(\mu_{x+ \alpha^{\vee}})_{G-\dom}\neq \mu$ and $(\mu_{x- \alpha^{\vee}})_{G-\dom}\neq \mu$}.$$
Then $\langle\alpha, \mu_x\rangle=0$.
Furthermore, $\mu_x$ is central on each connected component of the Dynkin diagram of $M$ satisfying that there is a simple root $\beta$ in that component with $\langle\beta,\alpha^{\vee}\rangle\neq 0$. In particular, $w_x(\alpha)=\alpha$.
\end{lemma}

\begin{proof} Suppose $\langle\alpha, \mu_x\rangle\neq 0$. Then depending on the sign of $\langle\alpha, \mu_x\rangle,$
one of $\mu_x + \alpha^{\vee}$ and $\mu_x - \alpha^{\vee}$ is conjugate to $\mu_x$ in $G,$
and in particular is $G$-minuscule. Hence $(\mu_{x+\alpha^{\vee}})_{G-\dom}=\mu$ or
$(\mu_{x-\alpha^{\vee}})_{G-\dom}=\mu$. This implies the first assertion.

The same argument also shows that our assumption implies $\langle\alpha, w\mu_x\rangle= 0$ for all $w\in W_M$. Fix a connected component of the Dynkin diagram of $M$ and assume that there is a simple root $\beta$ in that component such that $\langle\beta,\alpha^{\vee}\rangle\neq 0$.  As $\langle\alpha, \mu_x\rangle=\langle\alpha, s_{\beta}\mu_x\rangle=0$, we have $\langle\beta, \mu_x\rangle=0$. Similarly, for every neighbor $\beta'$ of $\beta$ in the Dynkin diagram of $M$ we have $\langle\alpha, \mu_x\rangle=\langle\alpha, s_{\beta}s_{\beta'}\mu_x\rangle=0$. Thus $\langle\beta', \mu_x\rangle=0$. By induction,  we obtain $\langle\gamma, \mu_x\rangle=0$ for every simple root $\gamma$ in that connected component of the Dynkin diagram of $M$. Hence $\mu_x$ is central in that connected component. The last assertion follows.
\end{proof}

\begin{remark}\label{rem452}Let $x, x'\in \bar I_{\mu, b}^{M, G}$ and $x\rightarrow x'$.
Suppose $x'-x=\alpha^\vee-\alpha^{m\vee}$ with $\alpha$ adapted, and $m$ satisfying the conditions
in Definition \ref{def_immediate_distance}. By Lemma \ref{lemma_newB}, $\mu_{x+\alpha^{\vee}}$ and
$\mu_{x-\alpha^{m\vee}}$ are $G$-minuscule. Hence, for any $\alpha^i$ not
in the same connected component of the Dynkin diagram of $G$ as $\alpha$ or $\alpha^m$ with $0<
i< m$, the conditions $x+\alpha^{\vee}-\alpha^{i\vee}\notin \bar I_{\mu, b}^{M, G}$ and $x+\alpha^{i\vee}-\alpha^{m\vee}\notin \bar I_{\mu, b}^{M, G}$ imply that $(\mu_{x+\alpha^{i\vee}})_{G-\dom}\neq
\mu$ and $(\mu_{x-\alpha^{i\vee}})_{G-\dom}\neq \mu$. Hence by
Lemma \ref{lemma_D}, we have $\langle \alpha^i, \mu_x\rangle=0$
and $w_x(\alpha^i)=\alpha^i$.
\end{remark}

\begin{para}
Let $x\in \bar I_{\mu, b}^{M, G}$. By Remark \ref{lemma_b=b_x}, there is a $g_x\in M(L)$ with
$g_x^{-1}b\sigma(g_x)=b_x$. Then $g_xM(\mathcal{O}_L)\in
X_{\mu_x}^M(b)$.

The main ingredient of the proof of Proposition \ref{pm2} is the following proposition.
\end{para}

\begin{prop}\label{cor_line_immediate}Let $x, x'\in \bar I_{\mu, b}^{M, G}$ and $x\rightarrow x'$. Suppose $x-x'=\alpha^{\vee}-\alpha^{m\vee}$ as in Definition \ref{def_immediate_distance} with $\alpha$ adapted. Let $g_x M(\mathcal{O}_L)\in X_{\mu_x}^M(b)$ as before. Then there exists $g' M(\mathcal{O}_L) \in X^M_{\mu_{x'}}(b)$ such that  $g_x$ and $g'$ have the same image in $\pi_0(X^{G}_{\mu}(b))$. Moreover,
\begin{eqnarray}\label{eqn_line_immediate}w_M(g_x)-w_M(g')= \sum_{i=0}^{m-1} \alpha^{i\vee} \text { in } \pi_1(M).\end{eqnarray}
\end{prop}

Before giving the proof of Proposition \ref{cor_line_immediate}, we first show how to use it to prove Proposition \ref{pm2}.

\begin{proof}[Proof of Proposition \ref{pm2}] By Proposition
\ref{prop_reduction_immediate}, we may assume that the distance from $x$ to $x'$ is immediate. As $J_b^M(F)$ acts transitively on $\pi_0(X^M_{\mu_x}(b))$ by Proposition \ref{prop:generalsuperbasictrans}, for any $g\in X^{M}_{\mu_x}(b)$ there exists $j\in J_b^M(F)$ such that $g$ and $jg_x$ have the same image in $\pi_0(X^M_{\mu_x}(b))$. In particular, they have the same image in $\pi_0(X^G_{\mu}(b))$. By Proposition \ref{cor_line_immediate}, there exists $g_1M(\mathcal{O}_L) \in X^M_{\mu_{x'}}(b)$ such that $g_x$ and $g_1$ have the same image in $\pi_0(X^{G}_{\mu}(b))$. Therefore $g$ and $jg_1$ have the same image in $\pi_0(X^{G}_{\mu}(b))$. So $g'=jg_1$ is the desired element.
\end{proof}

\begin{para}\label{para:notnsetup} Now it remains to prove Proposition \ref{cor_line_immediate}. The strategy of the proof is as follows. First we construct some ``affine lines'' $g_{x,x'}$ and view them as part of ``projective lines''. By an explicit computation, we will see that $g$ and $g'$ are both on the ``projective lines" corresponding to the points at 0 and $\infty$ respectively. The proposition then follows:

Keep the notation of Proposition \ref{cor_line_immediate}, and let $\Omega=\Gamma\alpha$.
Recall the element $b_x = \mu_x(p)\dot w_x$ in the $\sigma$-conjugacy class of $b,$ defined in \ref{lemma_b=b_x}.
For $i\geq 0$ we set $b_{x}^{(i)}=b_x\sigma(b_x)\dotsm\sigma^i(b_x)$. It will be convenient to set $b_x^{(-1)} = 1.$
The  root subgroup $U_{\alpha} \subset G$ is naturally defined over $\O_L.$ In the following we fix isomorphisms
$\theta_{\gamma}: U_{\alpha} \iso \mathbb G_a$ over $\O_L,$ satisfying $\sigma^*(\theta_{\gamma}) = \theta_{\sigma(\gamma)}.$
Then $\dot w_x U_{\alpha}(y)\dot w_{x}^{-1}=U_{w_x\alpha}(c^x_{\alpha}y)$ for some $c^x_{\alpha}\in \mathcal{O}^{\times}_L$ depending on $\dot w_x$ and on $\alpha$.

Let $R = \bar k[y]$ and $\R = \O_L\langle y \rangle$ equipped with the Frobenius $\sigma(y) = y^q.$
We define
$g_{x,x'}(y)\in G(\R_L)/G(\R)$ as follows:
\[g_{x, x'}(y):=g_x
 (b_{x}^{(m-2)}\sigma^{m-1}
U_{\alpha}(p^{-1}y)(b_{x}^{(m-2)})^{-1})\cdots  (b_{x}\sigma
U_{\alpha}(p^{-1}y)b_{x}^{-1})U_{\alpha}(p^{-1}y),\]

{\em except} if $\Omega$ is of type III, $d<m<2d$, and $\langle\beta^{m-d}, \mu_x\rangle=1$, in which case we let
\[g_{x, x'}(y):=g_{x'} (b_{x'}^{(m-2)}\sigma^{m-1}
U_{-\alpha}(p^{-1}y)(b_{x'}^{(m-2)})^{-1})\cdots  (b_{x'}\sigma
U_{-\alpha}(p^{-1}y)b_{x'}^{-1})U_{-\alpha}(p^{-1}y)
\]
\end{para}

\begin{prop}\label{prop_affineline_x_x'} With the notations above, we have
$$S_{\preceq\mu}(g_{x,x'}(y)^{-1}b\sigma g_{x, x'}(y))=\Spec R.$$
\end{prop}

\begin{proof} We first deal with the case when $\Omega$ is of type
I or II. By Lemma \ref{lemma_D} and Remark \ref{rem452}, we have $b_{x}^{(i-1)}
U_{\alpha^{i}}(p^{-1}\sigma^{i}y)(b_{x}^{(i-1)})^{-1}=U_{\alpha^i}(c_ip^{-1}\sigma^i(y))$ for $i=1,\dotsc, m-1$ with $c_i\in\mathcal{O}_L^{\times}$ arising from the action of the representative ${\dot w}_x$ on the root subgroups. By Lemma \ref{lemma_newB},
$\mu_x-w_x\alpha^{m\vee}$ and $\mu_x+\alpha^\vee$ are
$G$-minuscule, so $\langle\alpha, \mu_x\rangle=-1$ and
$\langle w_{x}\alpha^m, \mu_x\rangle=1$. As $U_{\alpha},\dotsc,
U_{\alpha^{m-1}}$ are in different connected components they obviously commute. Using this, together
with Remark \ref{rem452}, and keeping in mind that $g_x^{-1}b\sigma (g_x)=b_x$, many of the factors in the definition of $g(y)^{-1}b\sigma(g(y))$ cancel and we obtain
\begin{eqnarray*}
A:=g_{x,x'}(y)^{-1}b\sigma g_{x, x'}(y)&=&U_{\alpha}(-p^{-1}y)
 (b_{x}^{(m-1)} U_{\alpha^{m}}(p^{-1}\sigma^{m}(y))(b_{x}^{(m-1)})^{-1})b_x\\
 &=&U_{\alpha}(-p^{-1}y) U_{w_x\alpha^{m}}(c\sigma^{m}(y))p^{\mu_x}\dot w_x
\end{eqnarray*}
for some $c\in\mathcal{O}_L^{\times},$  Here in the second equality we have used $\langle w_{x}\alpha^m, \mu_x\rangle=1.$

We want to show that $A\in G(\R) p^{\mu_x}G(\R)$. This assertion only depends on the element $U_{\alpha}(-p^{-1}y) U_{w_x\alpha^{m}}(c\sigma^{m}(y))p^{\mu_x}\in G_{\Omega}(L)$. This element (and also every factor in the product) is contained in the standard Levi subgroup of $H_{\Omega} \subset G_{\Omega}$ corresponding to the Galois orbit of the connected component of the Dynkin diagram of $G_{\Omega}$ which contains $\Omega$. Note that $\Gamma$ acts transitively on the connected components of the Dynkin diagram of $H_{\Omega}.$

If $U_{\alpha}$ and $U_{w_x\alpha^m}$ commute, then using $\langle\alpha, \mu_x\rangle=-1,$ we obtain
$$A\in G(\R) U_{\alpha}(-p^{-1}y)p^{\mu_x}G(\R)= G(\R) p^{\mu_x} U_{\alpha}(-y)G(\R)= G(\R) p^{\mu_x}G(\R).$$

If $U_{\alpha}$ and $U_{w_x(\alpha^m)}$ do not commute, then
$\Omega$ is of type II and all the roots in $H_{\Omega}$ are of
the same length. In this case, $\langle w_x\alpha^m,
\alpha^{\vee}\rangle=-1$ and $\alpha+ w_x\alpha^{m}$ is the only
positive linear combination of $\alpha$ and $w_x\alpha^{m}$ which can be a root. By Lemma \ref{lemma_newB},
$\mu_x+\alpha^{\vee}-w_x\alpha^{m\vee}$ is $G$-minuscule. On the
other hand,
\[\langle w_x\alpha^{m}, \mu_x+\alpha^{\vee}-w_x\alpha^{m\vee}\rangle= -2,\]
so we get a contradiction.

Now we deal with the case when $\Omega$ is of type III. Recall
that $|\Omega|=3d$. Suppose either $m\leq d$ or $d< m< 2d$ with
$\langle\beta^0,\mu_x\rangle=\langle\beta^{m-d}, \mu_x\rangle=0.$
Then by Lemma \ref{lemma_D} and Remark \ref{rem_D4_2b_explicite_comp},
$\langle \alpha^i, \mu_x \rangle = 0,$ and $w_x(\alpha^i) = \alpha^i$
$i=1,\cdots m-1,$ and hence
$b_{x}^{(i-1)}U_{\alpha^{i}}(p^{-1}\sigma^{i}(y))(b_{x}^{(i-1)})^{-1}=U_{\alpha^i}(p^{-1}c_i\sigma^i(y))$
for some $c_i\in \mathcal{O}_L^{\times}$. Keeping in mind that in this case $U_{\alpha^i}$ and $U_{\alpha^{i+d}}$
commute, and that $U_{\alpha}$ and $U_{w_x\alpha^m}$ commute,
the same calculation for $A$ as in the case above applies.

Now suppose $d<m<2d$. We may assume that $\langle\beta^{m-d},
\mu_x\rangle=0$. Otherwise, $x'-x= (-\alpha)^{\vee}-(-\alpha)^{m\vee},$ and one checks that
that $x'\rightarrow x$ if we use negative roots instead of
positive ones. Now $\langle\beta^{m-d}, \mu_x\rangle=1$ implies that $\langle (-\beta)^{m-d},
\mu_{x'}\rangle=0$. Therefore we may reduce to the above case by exchanging
$x$ and $x'$, and using the opposite Borel group and negative
roots.

It remains to consider the case when $d< m< 2d$, $\langle\beta^0,
\mu_x\rangle=1$ and $\langle\beta^{m-d}, \mu_x\rangle=0$. By Remark
\ref{rem_D4_2b_explicite_comp}, we have $\langle\alpha,
\mu_x\rangle=-1$, $\langle\alpha^m, \mu_x\rangle=1$ and
$\langle\alpha^i, \mu_x\rangle=0$ for $i=d, m-d$.

For $i=1,\cdots, m-1$, $i\neq m-d$, $\alpha^i$ is not in the same
connected component as $\alpha^m$, so
\begin{eqnarray*}
A&=&g_{x,x'}(y)^{-1}b\sigma g_{x, x'}(y)\\
&=&U_{\alpha}(-p^{-1}y)(b_{x}^{(m-d-1)}U_{\alpha^{m-d}}(-p^{-1}\sigma^{m-d}(y))(b_{x}^{(m-d-1)})^{-1}) \\
&&\cdot~(b_{x}^{(m-1)} U_{\alpha^{m}}(p^{-1}\sigma^{m}(y))(b_{x}^{(m-1)})^{-1}) \\
&&\cdot~(b_{x}^{(m-d-1)}U_{\alpha^{m-d}}(p^{-1}\sigma^{m-d}(y))(b_{x}^{(m-d-1)})^{-1})b_x
\\&=& U_{\alpha}(-p^{-1}y) U_{\alpha^{m-d}}(-p^{-1}c_{1}\sigma^{m-d}(y)) U_{\alpha^m+\beta^{m-d}}(pc_2\sigma^m(y))U_{\alpha^{m-d}}(p^{-1}c_{1}\sigma^{m-d}(y))b_x
\end{eqnarray*}
where the last equality follows by Lemma \ref{lemma_D} and where $c_1,c_2\in \mathcal{O}_L^{\times}$ are constants arising from the action of the representative $\dot w_x$ on the root subgroups.

Note that $\alpha$ is also not in the same connected component as $\alpha^{m-d}$
and $\alpha^m$. Thus in order to show $A\in G(\R)p^{\mu}G(\R)$, it suffices to
show the following elements are in $G(\R)p^{\mu}G(\R).$
\begin{eqnarray*}A_1&:=&U_{\alpha^{m-d}}(-p^{-1}c_{1}\sigma^{m-d}(y)) U_{\alpha^m+\beta^{m-d}}(pc_2\sigma^m(y))U_{\alpha^{m-d}}(p^{-1}c_{1}\sigma^{m-d}(y))p^{\mu_x}
 \\
A_2&:=&U_{\alpha}(-p^{-1}y)p^{\mu_x}.\end{eqnarray*}
But  $A_2 = p^{\mu_x}U_{\alpha}(-y)\in G(\R)p^{\mu}G(\R)$ and
\[A_1=U_{\alpha^m+\beta^{m-d}}(pc_2\sigma^m(y)) U_{\alpha^m+\beta^{m-d}+\alpha^{m-d}}(c_3\sigma^{m-d}(y)\sigma^m(y))p^{\mu_x} \in G(\R)p^{\mu}G(\R).\]
where $c_3\in \mathcal{O}_L$ such that
\[  [U_{\alpha^m+\beta^{m-d}}(-pc_2\sigma^m(y)), U_{\alpha^{m-d}}(-p^{-1}c_{1}\sigma^{m-d}(y))]=U_{\alpha^m+\beta^{m-d}+\alpha^{m-d}}(c_3\sigma^{m-d}(y)\sigma^m(y)).\]
\end{proof}

\begin{proof}[Proof of Proposition \ref{cor_line_immediate}]
Let $\R' = \O_L \langle y,y^{-1}\rangle,$ equipped with the Frobenius given by $\sigma(y) = y^q.$ So the natural map $\R \rightarrow \R'$
is a morphism of frames.

Recall that for any root $\gamma$ in $G,$ we have chosen an isomorphism of $\O_L$-groups
$\theta_{\gamma}: U_{\gamma} \iso \mathbb G_a,$ with $\sigma^*(\theta_{\gamma}) = \theta_{\sigma(\gamma)}.$
An $\mathrm{SL}_2$-calculation shows that given $\theta_{\gamma},$
$\theta_{-\gamma}$ may be chosen so that we have
\begin{eqnarray}\label{eqn_SL2_computation} U_\gamma(p^{-1}y)=U_{-\gamma}(py^{-1})p^{-\gamma^{\vee}}h\end{eqnarray}
for some $h \in G(\O_L[y,y^{-1}]) \subset G(\R').$
Moreover, $(\theta_{\sigma\gamma}, \theta_{-\sigma\gamma}) = (\sigma^*\theta_{\gamma}, \sigma^*\theta_{-\gamma})$
then also satisfy the same property with respect to the root $\sigma(\gamma).$
In the following we fix such a choice for the Galois orbits of all roots $\gamma.$

If $\Omega$ is of type I or II, then $\alpha,\dotsc,\alpha^{m-1}$
are in different connected components of the Dynkin diagram of
$G_{\Omega}.$ We have
\begin{eqnarray}
\nonumber g_{x, x'}(y)&=&g_x U_{\alpha^{m-1}}(p^{-1}c_{m-1}\sigma^{m-1}(y))\cdots U_{\alpha^{1}}(p^{-1}c_1\sigma(y))U_{\alpha}(p^{-1}y)\\
\label{eqfg1}&\in& g_x U_{-\alpha^{m-1}}(pc_{m-1}^{-1}\sigma^{m-1}(y^{-1}))\cdots U_{-\alpha^{1}}(pc_{1}^{-1}
\sigma(y^{-1}))\\
&&\nonumber \quad\cdot\; U_{-\alpha}(py^{-1})p^{-\sum_{i=0}^{m-1}\alpha^{i\vee}}G(\R')
\end{eqnarray}
for suitable constants $c_i\in \mathcal{O}_L^{\times}$.

We define a second element $f_{x,x'}(y)\in G(\R_L)$ by setting
$$f_{x,x'}(y)=g_x U_{-\alpha}(py)U_{-\alpha^{1}}(pc_{1}^{-1}\sigma(y))\cdots
U_{-\alpha^{m-1}}(pc_{m-1}^{-1}\sigma^{m-1}(y))p^{-\sum_{i=0}^{m-1}\alpha^{i\vee}}.$$
Then  $f_{x,x'}(y)\in g_{x,x'}(y^{-1})$ in $G(\R'_L)/G(\R').$
In particular, by Proposition \ref{prop_affineline_x_x'}
$$S_{\preceq\mu}(f_{x,x'}^{-1}b\sigma(f_{x,x'}))\supseteq \Spec(k[y])\setminus \{0\}.$$
By Lemma \ref{1.1.3}, this set is Zariski closed. Hence $f_{x,x'}$ defines an element of $X_{\preceq\mu}(b)(\R)$.
In particular $g_{x,x'}(0)=g_x\in X^M_{\mu_x}(b)$ and $g':=f_{x,x'}(0)$ have the same image in $\pi_0(X_{\preceq\mu}(b))$.

By the definition of $f_{x,x'}$ we have
\[g'=f_{x,x'}(0)
=g_xp^{-\sum_{i=0}^{m-1}\alpha^{i\vee}}
\] in $M(\R_L)/M(\R)$. Therefore $g'\in X^M_{\mu_{\tilde x}}(b)$ for some $\tilde{x}\in \bar I^{M, G}_{\mu, b}$. As $\tilde{x}=w_M(g'^{-1}b\sigma g')=x'$ in $\pi_1(M)$, we have $g' \in X^M_{\mu_{x'}}(b)$ and (\ref{eqn_line_immediate}) holds.

If $\Omega$ is of type III, we apply the same construction. As in
the proof of Proposition \ref{prop_affineline_x_x'} we may assume that $\Gamma$ acts transitively on the connected components of the Dynkin diagram of $G_{\Omega}$ and that $\langle\beta^{m-d},
\mu_x\rangle=0$ (otherwise, we exchange $x$ and $x'$ and use
negative roots instead of positive ones). Moreover if $m\leq d$ or $d<m
<2d$ with $\langle\beta^{0},\mu_x\rangle=0$, then the definition of
$f_{x,x'}$ and the computation of $g':=f_{x,x'}(0)$ are the same
as above. It remains to consider the case when  $d<m<2d$,
$\langle\beta^{0},\mu_x\rangle=1$ and $\langle\beta^{m-d},
\mu_x\rangle=0$. By Remark
\ref{rem_D4_2b_explicite_comp},
\[\begin{split}g_{x, x'}(y)=&g_{x} U_{\alpha^{m-1}+\beta^{m-d-1}}(c_{m-1}\sigma^{m-1}(y)) \cdots U_{\alpha^{d}+\beta^{0}}(c_{d}\sigma^{d}(y))\\
&\cdot~ U_{\alpha^{d-1}}(p^{-1}c_{d-1}\sigma^{d-1}(y))\cdots U_{\alpha^{1}}(p^{-1}c_{1}\sigma^{1}(y))U_{\alpha}(p^{-1}y)
\end{split}\]
where as usual the $c_i$ are constants in $\mathcal{O}_L^{\times}$ arising from the conjugation by the representative $\dot w_x$ on the root subgroups. We can decompose $g_{x, x'}(y)=g_x h_0(y)\cdots
h_{d-1}(y)$ into the terms corresponding to the different connect
components of the Dynkin diagram of $G_{\Omega}$. Here
\[h_i(y)=\begin{cases}U_{\beta^i+\alpha^{i+d}}(c_{i+d}\sigma^{i+d}(y)) U_{\alpha^i}(p^{-1}c_i\sigma^i(y))
 &i=0, \dotsc, m-d-1 \\
 U_{\alpha^i}(p^{-1}c_i\sigma^i(y)) &i=m-d,
 \dotsc, d-1. \end{cases}\]
When $0\leq i\leq m-d-1,$ we have the following equalities in $G(\R'_L)/G(\R'):$
\begin{eqnarray*}
h_i(y)&\in &  U_{\beta^i+\alpha^{i+d}}(c_{i+d}\sigma^{i+d}(y)) U_{-\alpha^i}(p c_i^{-1}\sigma^i(y^{-1}))p^{-\alpha^{i\vee}}\\
&=&   U_{-\alpha^i}(p c_i^{-1}\sigma^i(y^{-1}))p^{-\alpha^{i\vee}}U_{\beta^i+\alpha^{i+d}}(p^{-1}c_{i+d}\sigma^{i+d}(y))\\
&=& U_{-\alpha^i}(p c_i^{-1}\sigma^i(y^{-1}))p^{-\alpha^{i\vee}}U_{-\beta^i-\alpha^{i+d}}(pc_{i+d}^{-1}\sigma^{i+d}(y^{-1})) p^{-\beta^{i\vee}-\alpha^{i+d\vee}}\\
&=& U_{-\alpha^i}(p c_i^{-1}\sigma^i(y^{-1}))U_{-\beta^i-\alpha^{i+d}}(c_{i+d}^{-1}
\sigma^{i+d}(y^{-1})) p^{-\alpha^{i\vee}-\beta^{i\vee}-\alpha^{i+d\vee}}.
\end{eqnarray*}
Write the last of the expressions above as $f_{x,x'}^i(y^{-1}),$ where $f_{x,x'}^i(y)\in G(\R_L).$ Then $f_{x,x'}^i(y) = h_i(y^{-1})$ in $G(\R'_L)/G(\R').$
Moroever $f^i_{x,x'}(0) = p^{-\alpha^{i\vee}-\beta^{i\vee}-\alpha^{i+d\vee}}$.

When $i\geq m-d$ and $y\neq 0$,
$$h_i(y)\in U_{-\alpha^i}(pc_i^{-1}\sigma^i(y^{-1}))p^{-\alpha^{i\vee}}G(\R')$$
Defining $f_{x,x'}^i(y)=U_{-\alpha^i}(pc_i^{-1}\sigma^i(y))p^{-\alpha^{i\vee}}$ we obtain again $f_{x,x'}^i(y) = h_i(y^{-1})$ in $G(\R'_L)/G(\R'),$ and
 $f^i_{x,x'}(0) = p^{-\alpha^{i\vee}}.$ Let $f_{x,x'}=g_xf_{x,x'}^0\dotsm f_{x,x'}^{d-1}$. Then
$$g':=f_{x,x'}(0)=g_xf_{x,x'}^0(0)\dotsm f_{x,x'}^{d-1}(0) = g_x
p^{-\sum_{i=0}^{m-1}\alpha^{i\vee}-\sum_{j=0}^{m-d-1}\beta^{j\vee}}$$ and (\ref{eqn_line_immediate}) holds.
The same verification as in the type I and II cases shows that $g'=f_{x,x'}(0)\in X^M_{\mu_{x'}}(b)$.
\end{proof}

\subsection{Proof of Proposition \ref{pm4}}

In order to prove Proposition \ref{pm4}, we need the following
lemma.

\begin{lemma}\label{lemma_Weyl_orbit} Let $H \subset G$ be a standard Levi subgroup,
and $\alpha$ a positive root of $G,$ which is $H$-anti-dominant.
If $\gamma\in W_H\alpha,$ then there exists a
finite set of  positive roots $(\beta_i)_{i\in J}$ in $H$ such that

\begin{itemize}
\item  $\langle\beta_i, \beta_j^{\vee}\rangle=0$ for all $i,
j\in J$ with $i\neq j$.

\item $\gamma=(\prod _{i\in J}s_{\beta_i})(\alpha)$ where the
product does not depend on the order of $s_{\beta_i}$.
\item $\langle \gamma, \beta_i^{\vee} \rangle > 0 > \langle \alpha, \beta_i^{\vee} \rangle$ for $i \in J.$
\item $|\gamma|=|\alpha|+\sum_{i\in J}|\langle\alpha,
\beta_i^{\vee}\rangle|\cdot|\beta_i|$.
\end{itemize}
\end{lemma}
\begin{proof}\noindent{\it Case 1: $\alpha$ is not longer than any root in $G$}

As $\gamma\in W_{H}\alpha$, $\gamma$ has the same
length as $\alpha$. Then for any root $\beta$ in $G$ other than
$\pm\alpha, \pm\gamma$,

\[|\langle\alpha, \beta^{\vee}\rangle|,|\langle \gamma, \beta^{\vee}\rangle|\in\{0,1\}.\]

Since $\alpha$ is $H$-anti-dominant, we may write $\gamma-\alpha=\sum_{i\in J}\beta_i$ with $\beta_i$ positive
roots in $H$. By Lemma \ref{lem:addroots}, after regrouping $\beta_i$ , we may assume that $\langle\beta_i,
\beta_j^{\vee}\rangle\geq 0$ for all $i,j\in J$. As the $\beta_j $ are roots in $H$ we have $\beta_j\neq \pm\alpha, \pm\gamma$ for every $j\in J$. Therefore

\[2\geq \langle\gamma, \beta_j^{\vee}\rangle-\langle\alpha, \beta_j^{\vee}\rangle
=\langle\sum_{i\in J}\beta_i, \beta_j^{\vee}\rangle\geq 2.\] This
implies that $\langle\gamma, \beta_j^{\vee}\rangle=1$,
$\langle\alpha, \beta_j^{\vee}\rangle=-1$ and $\langle\beta_i,
\beta_j^{\vee}\rangle=0$ for all $i, j\in J$ with $i\neq j$. So the
$(\beta_i)_{i\in J}$ have all the desired properties.

\noindent{\it Case 2: $\alpha$ is a long root in $G$}

Then $\alpha^{\vee}$ is not longer than any coroots in $G.$ Applying the above construction
using coroots instead of roots, we find a finite set of positive roots $(\beta_i)_{i\in J}$ in
$H$ such that $\gamma^{\vee} = (\prod_{i \in J} s_{\beta_i})(\alpha^{\vee}),$
$\langle \beta_i,\gamma^{\vee}\rangle=1$, $\langle\beta_i,
\alpha^{\vee}\rangle=-1$, and $\langle\beta_i,
\beta_j^{\vee}\rangle=0$ for all $i, j\in J$ with $i\neq j$. Then $\gamma = (\prod_{i \in J} s_{\beta_i})(\alpha),$
$\langle\gamma, \beta_{i}^{\vee}\rangle>0$ and $\langle\alpha,
\beta_i^{\vee}\rangle<0$. Therefore $(\beta_i)_{i\in J}$ is still
the set of desired roots.
\end{proof}

\begin{proof}[Proof of Proposition \ref{pm4}]
 Recall that we are assuming $b=p^{\mu_{x_0}}\dot w_{x_0}$ with $x_0\in \bar I_{\mu, b}^{M, G}$.

By assumption $G^{\ad}$ is simple, so $\Gamma$ acts transitively on the set of connected
components of the Dynkin diagram of $G$. Let
\begin{eqnarray*}C_1&:=&\big\{\alpha^{\vee}\in X_*(T)\mid \alpha \text{ is a positive root in
}N, \text{ such that } \langle \alpha,
\mu_{x_0}\rangle <0 \big\},  \\
 C_2 &:=&\big\{\alpha^{\vee}\in X_*(T)\mid\begin{array}{l}
\alpha \text{ is an }M \text{-anti-dominant and positive root in }N,
\\ \text{ such that } \langle \alpha, \mu_{x_0}\rangle <0
\end{array}\big\}.\end{eqnarray*} Then $C\subset C_2\subset C_1$.

Let $L_C$ (resp. $L_{C_i}$) be the $\mathbb{Z}$-lattice generated
by the elements of the Galois orbit of $C$ (resp. $C_i$ for $i=1,
2$) and the coroots of $M$.

Let $\alpha$ be a simple root in $N$, and $\Omega=\Gamma\alpha.$
Let  $\tilde{G}_{\Omega}$ be the standard Levi subgroup of $G$ corresponding to the set of simple roots not in $\Omega.$
We set

\begin{eqnarray*}R_{\alpha}&:=& W_{\tilde{G}_{\Omega}}\alpha
\\
\tilde{R}_{\alpha}&:=&\{\gamma\in R_{\alpha}|\langle\gamma,
\mu_{x_0}\rangle <0 \}\\
\tilde{R}_{\Omega}&:=&\bigcup_{\alpha'\in \Omega}
\tilde{R}_{\alpha'}\subset R_{\Omega}:=\bigcup_{\alpha'\in
\Omega} R_{\alpha'}.\end{eqnarray*}

\noindent {\it Claim 1: $\tilde{R}_{\Omega}\neq \emptyset$.}

Once Claim 1 is proved for the Galois orbit $\Omega$, we
define $\gamma(\Omega)$ to be a minimal element in
$\tilde{R}_{\Omega}$ for the order $\preceq$.

We now prove this claim.  Take $w\in W_{\tilde{G}_{\Omega}}$ with
$w\mu_{x_0}$ $\tilde{G}_{\Omega}$-dominant. Then $w\mu_{x_0}$ is
not $G$-dominant, otherwise $w\mu_{x_0}=\mu$ and $\mu_{x_0}=\mu$
in $\pi_1(\tilde{G}_{\Omega})$ which contradicts that $(\mu,b)$ is Hodge-Newton irreducible. So there exists $\tilde{\alpha}\in\Omega$
with $\langle\tilde{\alpha}, w\mu_{x_0}\rangle<0$ and therefore
$w^{-1}\tilde{\alpha}\in\tilde{R}_{\Omega}$. This shows
Claim 1.

\noindent {\it Claim 2: $L_{C_1}$ is the coroot lattice of $G$.}

In order to show Claim 2, it suffices to show that for any
simple root $\alpha$ in $N$, there exists
$\tau\in\Gamma$ such that $(\tau\alpha)^{\vee}\in L_{C_1}$.
We may assume that $\gamma(\Omega)\in R_{\alpha},$ and we
show that this implies $\alpha^{\vee}\in L_{C_1}$.

By the definition of $\gamma(\Omega)$, we have
$\langle\gamma(\Omega), \mu_{x_0}\rangle <0$. Then
$\gamma(\Omega)^{\vee}\in C_1$. By Lemma
\ref{lemma_Weyl_orbit}, there exists a finite set of positive roots
$(\beta_i)_{i\in J}$  such that $\gamma(\Omega)=(\prod_{i\in
 J}s_{\beta_i})\alpha$ satisfying the conditions in Lemma \ref{lemma_Weyl_orbit}.
Therefore in order to show $\alpha^{\vee}\in L_{C_1}$, it suffices to show that for all $i\in
J$, $\beta_i^{\vee}\in L_{C_1}$.

For $i\in J$, if $\beta_i$ is a root in $M$, then $\beta^{\vee}\in
L_{C_1}$ by the definition of $L_{C_1}$.  It remains the case when
$\beta_i$ is a root in $N$. Since $\langle \gamma, \beta_i^{\vee} \rangle > 0,$
$s_{\beta_i}(\gamma(\Omega)) \preceq \gamma(\Omega).$ Hence by the minimality
of $\gamma(\Omega)$, we have
$\langle s_{\beta_i}\gamma(\Omega), \mu_{x_0}\rangle \geq 0>
\langle\gamma(\Omega), \mu_{x_0}\rangle$, and hence
$\langle\beta_i, \mu_{x_0}\rangle<0$. Therefore $\beta_i^\vee\in
C_1$. This show Claim 2.

For any $\gamma^{\vee}\in C_1$, let $\tilde{\gamma}$ be the
$M$-anti-dominant representative of $\gamma$ in $W_M\gamma$. Then
$\tilde{\gamma}\in C_2$ and  $L_{C_2}$ is the coroot lattice
of $G$ by Claim 2. Hence, in order to show this proposition, it suffices to
show that $C_2\subset L_C$.

Suppose $\gamma^{\vee}\in C_2\backslash C$. Then exists a positive
root $\beta$ in $M$ such that $\langle \beta, \gamma^{\vee}\rangle
<-1$. This implies that there is a simple root $\beta'$ of $M$
with $\beta'\preceq\beta$ and  $\langle \beta', \gamma^{\vee}\rangle < 0.$ Since $M$ is of type A,
$\beta$ and $\beta'$ have the same length, so
$\langle \beta', \gamma^{\vee}\rangle = \langle \beta, \gamma^{\vee}\rangle < - 1.$
Replacing $\beta$ by $\beta',$ we may assume $\beta$ is simple.
Let $\gamma_1=s_{\gamma}(\beta).$ Then $\gamma_1$ is  longer than $\gamma$ and $\gamma_1 \in C_1$ since
\[\langle \gamma_1, \mu_{x_0}\rangle=
\langle \beta-\langle \beta, \gamma^{\vee}\rangle \gamma,
\mu_{x_0}\rangle =\langle \beta, \mu_{x_0}\rangle+\langle \beta,
\gamma^\vee\rangle<0. \]
Furthermore, as
$\gamma_1^{\vee}=s_{\gamma}(\beta^{\vee})=\gamma^{\vee}+\beta^{\vee},$
we have
$$\langle \beta, \gamma_1^{\vee}\rangle=\langle \beta, \gamma^{\vee}+\beta^{\vee}\rangle \leq 0,$$
so $\gamma_1^{\vee}$ is $M$-anti-dominant as $\gamma$ is $M$-anti-dominant.
Therefore $\gamma_1^{\vee}\in L_C$ and then $\gamma^{\vee}\in L_C$.
\end{proof}

\subsection{Proof of Proposition \ref{pm3}}
We continue to use the notation introduced above. Thus for $x\in \bar I_{\mu, b}^{M, G}$, we have the element $b_x=p^{\mu_x}\dot w_x\in M(L)$
defined in Subsection \ref{overview}, so that $b_x$ is basic in $M$ and there is a $g_x\in G(L)$ with $g_x^{-1}b\sigma(g_x)=b_x.$
Then $g_x\in X_{\mu_x}^M(b).$ Moreover, we continue to use the normalization of the root subgroups of $G$ fixed in
the proof of Proposition \ref{cor_line_immediate}, and, as above, for any root $\alpha$ of $G,$ we write $\alpha^i = \sigma^i(\alpha).$

Let $\Omega\in\Phi_{N,\Gamma}$ be adapted and $\alpha\in \Omega$. Let $d>0$ be the minimal positive integer such that $\alpha$ and $\alpha^d$ are in the same connected component of the Dynkin diagram of $G_{\Omega}$. Then $n:=|\Omega|$ is equal to $d$, $2d$, or $3d$ if $\Omega$ is of type I, II, or III, respectively. If $\Omega$ is of type II or III, by Proposition \ref{lemma_newA}, all the roots in $\Omega$ are simple in $G_{\Omega}$. If $\Omega$ is of type II and $\alpha,\alpha^d$ are not neighbors, then by Lemma \ref{(1.2.6)} applied to $M$ the two simple roots $\alpha,\alpha^d$ have a common neighbor $\beta$ in the Dynkin diagram of $G_{\Omega}$. If $\Omega$ is of type III, let $\beta$ be the common neighbor of $\alpha$, $\alpha^d$ and $\alpha^{2d}$. In all other cases let $\beta=0$. Let
$$\tilde\alpha=\begin{cases}
\alpha&\text{if } \Omega\text{ is of type I}\\
\alpha+\beta+\alpha^d&\text{if }\Omega\text{ is of type II}\\
\alpha+\alpha^d+\alpha^{2d}+\beta&\text{if } \Omega\text{ is of type III.}
\end{cases}$$
Note that in all cases $\tilde\alpha$ is a positive root.

\begin{lemma}\label{lemma_bar_I}Let $\Omega\in\Phi_{N,\Gamma}$. For any $x\in \bar I^{M, G}_{\mu, b}$,
 we have $x\in \bar I^{M, G_{\Omega}}_{\mu_{x}, b}\subseteq \bar I^{M, G}_{\mu, b}$. Moreover, for any $x_1\in \bar I^{M, G_{\Omega}}_{\mu_{x}, b}$, if $x_2=x_1+\alpha^{\vee}-\alpha'^{\vee}\in\bar I^{M, G}_{\mu, b}$ with $\alpha,\alpha'\in \Omega$, then $x_2\in \bar I^{M, G_{\Omega}}_{\mu_{x}, b}$.

\end{lemma}
\begin{proof}Recall that $$\bar I^{M,G_{\Omega}}_{\mu_{x}, b}=\{y\in \pi_1(M)\mid (\mu_y)_{G_{\Omega}-\dom}=(\mu_{x})_{G_{\Omega}-\dom}, y=\kappa_M(b)\text{ in }\pi_1(M)_{\Gamma}\}.$$ It is obvious that  $x\in \bar I_{\mu_{x}, b}^{M, G_{\Omega}}$. For the second assertion, let $x_1, x_2$ be as in the lemma. As $(\mu_{x_1})_{G-\dom}=\mu=(\mu_{x_2})_{G-\dom}$ and $\mu_{x_2}-\mu_{x_1}$ is a linear combination of coroots of $G_{\Omega}$, we have $(\mu_{x_2})_{G_{\Omega}-\dom}=(\mu_{x_1})_{G_{\Omega}-\dom}=(\mu_{x})_{G_{\Omega}-\dom}$. Thus $x_2\in \bar I^{M, G_{\Omega}}_{\mu_{x},b }$.
\end{proof}

\begin{lemma}\label{lempm3A}
Let $\Omega\in\Phi_{N,\Gamma}$ be adapted. Let $x,x'\in \bar I_{\mu, b}^{M, G}$ with $x'=x+\alpha^{\vee}- \alpha^{l\vee}$ for some $\alpha\in \Omega$ and $0<l<n$. We assume in addition that either $\Omega$ is of type I or that $\langle\tilde\alpha^i,\mu_y\rangle \geq 0$ for all $i\in\mathbb N$ and all $y\in \bar I^{M, G_{\Omega}}_{\mu_{x},b}$. Then for all $g\in X_{\mu_{x}}^M(b)$ there is a $g'\in X_{\mu_{x'}}^M(b)$ such that $g\sim g'$ and $w_M(g')=w_M(g)-\sum_{i=0}^{l-1}\alpha^{i\vee}$.
\end{lemma}

\begin{proof} We remind the reader that $g \sim g'$ means that $g,g'$ are in the same connected component of $X_{\mu}(b).$

We use induction on $l$. Suppose that $x'':=x+\alpha^{l_0\vee}- \alpha^{l\vee}\in \bar I_{\mu, b}^{M, G}$ for some $0<l_0<l.$ Then $x'', x'\in \bar I^{M, G_{\Omega}}_{\mu_x, b}$ by Lemma \ref{lemma_bar_I}. Applying the induction hypothesis to $(x,x'')$ and $(x'',x')$ we obtain a $g''\in X_{\mu_{x''}}^M(b)$ such that $g\sim g''$ and $w_M(g'')=w_M(g)-\sum_{i=0}^{l_0-1}\alpha^{i\vee}$, and a $g'\in X_{\mu_{x'}}^M(b)$ such that $g''\sim g'$ and $w_M(g'')=w_M(g')-\sum_{i=l_0}^{l-1}\alpha^{i\vee}$. Then $g'$ is the desired element.
Thus we may assume that for  all $0<i<l,$ we have $x+\alpha^{\vee}- \alpha^{i\vee}\notin \bar I_{\mu, b}^{M, G}.$ A similar argument shows that
we may also assume  $x+\alpha^{i\vee}- \alpha^{l\vee}\notin \bar I_{\mu, b}^{M, G}$ for $0<i<l.$ We assume from now on that these two conditions
hold.

As $J_b^M(F)$ acts transitively on the set of connected components of each $X_{\mu_x}^M(b)$ by Proposition \ref{prop:generalsuperbasictrans}, and $w_M$ is constant on connected
components by Lemma \ref{1.1.3}, it is enough to prove the lemma for the particular element $g=g_x$. If $x\rightarrow x'$ is immediate, then the desired element $g'$ is the one constructed in Proposition \ref{cor_line_immediate}. Thus it remains to consider the case where $x\rightarrow x'$ does not hold. In particular, by Definition \ref{def_immediate_distance}, we only need to consider the following two cases: either $\Omega$ is of type II and $d< l<2d$ or $\Omega$ is of type III and $2d\leq l <3d$. For $i\in \mathbb{N}$, let
$$\tilde U_{\alpha}^i(y)=b_{x}^{(i-1)}\sigma^i (U_{\alpha}(y))(b_{x}^{(i-1)})^{-1}.$$ For $i=0$ this coincides with $U_{\alpha}(y)$. Let $R = \bar k[y]$
and $\R$ be the $R$-frame chosen in \ref{para:notnsetup}. We define
$g(y)\in G(\R_L)/G(\R)$ by
$$g(y)=g_{x}\tilde U_{\alpha}^{l-1}(p^{-1}y)\dotsm \tilde U_{\alpha}^{0}(p^{-1}y).$$ Using the same strategy as in Section \ref{sec65} we want to show that $S_{\preceq\mu}(g(y)^{-1}b\sigma g(y))=\Spec R.$ Then we will extend this family to a ``projective line'' and use that the point $g(0)$ and the point $g'$ ``at infinity'' are in the same connected component of $X_{\preceq \mu}^G(b)$. In order to compute $\tilde U_{\alpha}^i$, to verify the above statement and to compute $g'$ we consider the different types of $\Omega$ separately. We distinguish two cases according to the type of $\Omega$.

\begin{lemma}\label{lem:casetypeII} Keep the above notations and assumptions, and suppose that $\Omega$ is of type II and $d<l<2d$.
Then $\beta \neq 0$ if and only if $\langle \beta, \mu_x \rangle = 1.$
Moreover, we have
\begin{itemize}
\item $w_{x}\alpha^d = \alpha^d + \beta$
and $\langle w_x\alpha^d,\mu_x\rangle=1$
\item $w_x\alpha^{l-d}=\alpha^{l-d}$
and $\langle w_x\alpha^{l-d},\mu_x\rangle=0$
\item  For $0< i< l$ with $i \neq l-d,d,$ $w_x\alpha^i=\alpha^i$,  $\langle\beta^i, \mu_x\rangle=0$ and $\langle\alpha^i, \mu_{x}\rangle=0.$
\end{itemize}
\end{lemma}

\begin{proof}
As $x+\alpha^{\vee}-\alpha^{l\vee}\in \bar I_{\mu, b}^{M, G}$ we have $\langle\alpha,\mu_x\rangle=-1$ and $\langle w_x\alpha^l,\mu_x\rangle=1,$ by Lemma
\ref {lemma_newB}.
Our assumption $\langle \alpha+\beta+\alpha^d,\mu_x\rangle\geq 0$ and the fact that $\mu_x$ is minuscule then imply that $\langle \beta+\alpha^d,\mu_x\rangle=1$.
If $\langle \beta, \mu_x\rangle=1$, then
we have
$$1=\langle\beta+\alpha^d, \mu_x\rangle= \langle s_{\beta}(\alpha^d), \mu_x \rangle,$$
and if $\langle \beta, \mu_x\rangle=0$, we have $1=\langle\alpha^d, \mu_x\rangle=\langle w_{x}\alpha^d, \mu_x\rangle$.
Therefore, by Lemma \ref{lemma_w_x_alpha minimal}, we obtain
$w_x\alpha^d=\alpha^d+\beta\langle \beta, \mu_x\rangle$ and $\langle w_x\alpha^d,\mu_x \rangle = 1.$

Moreover, if $\langle \alpha^d, \mu_x \rangle = 1,$
and $\langle \alpha, \alpha^{d\vee} \rangle = 0,$ then $x+\alpha^{\vee}-\alpha^{d\vee} = s_{\alpha^d}(\mu_x+\alpha^{\vee}),$ which contradicts
$x+\alpha^{\vee}-\alpha^{d\vee}\notin \bar I_{\mu, b}^{M, G}.$ Hence $\langle \beta, \mu_x\rangle = 0$ implies that $\alpha, \alpha^d$ are neighbors,
and $\beta = 0.$ In particular $w_x\alpha^d=\alpha^d + \beta.$

If $\langle w_x\alpha^{l-d}, \mu_x \rangle = 1,$ then $s_{w_x\alpha^{l-d}}(\mu_x) = \mu_x - w_x\alpha^{l-d\vee} = \mu_{x-\alpha^{l-d}}$ by Lemma \ref{lemma_C},
which contradicts $x+\alpha^{\vee}-\alpha^{l-d\vee}\notin \bar I_{\mu, b}^{M, G}.$
 Hence we obtain that  \begin{eqnarray}\label{eqn_lemma3A}\langle \alpha^{l-d},\mu_x\rangle\leq \langle w_x\alpha^{l-d},\mu_x\rangle\leq 0.\end{eqnarray} We use an indirect proof to show $\langle \alpha^{l-d},\mu_x\rangle= 0$, so assume $\langle \alpha^{l-d},\mu_x\rangle=-1$. By assumption $\langle \tilde{\alpha}^{l-d}, \mu_x\rangle=\langle\alpha^{l-d}+\beta^{l-d}+\alpha^l,\mu_x\rangle\geq 0$, hence  $\langle\beta^{l-d}+\alpha^l,\mu_x\rangle=1$ and $\langle \tilde\alpha^{l-d}, \mu_x\rangle=0$. As above this implies that $w_x\alpha^l=\alpha^l$ or $w_x\alpha^l=\alpha^l+\beta^{l-d}$ by Lemma \ref{lemma_w_x_alpha minimal}. Thus
\begin{eqnarray*}
\langle\alpha^{l-d}+\beta^{l-d}+\alpha^l,\mu_{x'}\rangle&=&\langle\alpha^{l-d}+\beta^{l-d}+\alpha^l,\mu_{x}-w_x\alpha^{l\vee}\rangle\\
&=&0-1\quad <0
\end{eqnarray*}
which contraditcs an assumption of the Lemma. So $\langle \alpha^{l-d},\mu_x\rangle= 0$.  Then by (\ref{eqn_lemma3A}), $0= \langle \alpha^{l-d},\mu_x\rangle\leq \langle w_x\alpha^{l-d},\mu_x\rangle\leq 0$. This implies $\alpha^{l-d}=w_x\alpha^{l-d}$ by Lemma \ref{lemma_w_x_alpha minimal}.

Finally, for $0< i < l$ with $i \neq l-d, d$, the conditions
$x+\alpha^{\vee}- \alpha^{i\vee}\notin \bar I_{\mu, b}^{M, G}$ and $x+\alpha^{i\vee}- \alpha^{l\vee}\notin \bar I_{\mu, b}^{M, G}$
 imply that $(\mu_{x+\alpha^{i\vee}})_{G-\dom}\neq \mu$ and $(\mu_{x-\alpha^{i\vee}})_{G-\dom}\neq \mu$. Then by Lemma \ref{lemma_D}, $w_x\alpha^i=\alpha^i$,  $\langle\beta^i, \mu_x\rangle=0$ and $\langle\alpha^i, \mu_{x}\rangle=0.$
\end{proof}

\begin{para}{\em Proof of Lemma \ref{lempm3A} continued:}
Assume that $\Omega$ is of type II and $d<l<2d.$
As $w_x\alpha^{l-d}=\alpha^{l-d}$, we have $w_x\beta^{l-d}=\beta^{l-d}$. Then
$$w_{x}\sigma^{l-d}w_x\alpha^d = w_x(\alpha^l+\beta^{l-d})=w_x\alpha^l+\beta^{l-d}.$$
Using the $M$-dominance of $\mu_x$ we have $\langle w_{x}\sigma^{l-d}w_x\alpha^d,\mu_x\rangle\geq\langle w_{x}\alpha^l,\mu_x\rangle=1$.

Altogether, using Lemma \ref{lem:casetypeII} we obtain
$$\tilde U_{\alpha}^i(p^{-1}y)=\begin{cases}
U_{\alpha^i}(p^{-1}c_i\sigma^i(y))&\text{if }0\leq i<d\\
U_{\sigma^{i-d}w_x\alpha^d}(c_i\sigma^i(y))&\text{if }d\leq i<l\\
U_{w_x\sigma^{l-d}w_x\alpha^d}(pc_i\sigma^l(y))&\text{if }i=l
\end{cases}$$
with $c_i\in\mathcal{O}_L^{\times}$ as usual depending on $\dot w_x$ and $\alpha^i$, but not on $y$, and with $c_0=1$. Obviously root subgroups corresponding to roots in different connected components of the Dynkin diagram of $G_{\Omega}$ commute. By definition, we have $$p^{\mu_x}\dot w_x \sigma(\tilde U_{\alpha}^i(y))(p^{\mu_x}\dot w_x)^{-1}=\tilde U_{\alpha}^{i+1}(y).$$ Using these two facts many of the factors in the definition of $g(y)^{-1}b\sigma(g(y))$ cancel and we obtain
\begin{eqnarray*}
&\ &g(y)^{-1}b\sigma(g(y))\\ &=&\tilde U_{\alpha}^{0}(-p^{-1}y)\tilde U_{\alpha}^{l-d}(-p^{-1}y)\tilde U_{\alpha}^{l}(p^{-1}y)\tilde U_{\alpha}^{l-d}(p^{-1}y)p^{\mu_x}\dot w_x\\
&=&U_{\alpha}(-p^{-1}y)U_{\alpha^{l-d}}(-p^{-1}c_{l-d}\sigma^{l-d}(y))U_{w_x\sigma^{l-d}w_x\alpha^{d}}(pc_l\sigma^{l}(y))\\
&&\cdot U_{\alpha^{l-d}}(p^{-1}c_{l-d}\sigma^{l-d}(y))p^{\mu_x}\dot w_x.
\end{eqnarray*}
If $\langle w_x\sigma^{l-d}w_x\alpha^d,\alpha^{l-d\vee}\rangle =0$, then $U_{\alpha^{l-d}}$ and $U_{w_x\sigma^{l-d}w_x\alpha^{d}}$ commute.
Using in addition $\langle \alpha,\mu_x\rangle=-1$ we obtain
$$g(y)^{-1}b\sigma(g(y))=U_{\alpha}(-p^{-1}y)U_{w_x\sigma^{l-d}w_x\alpha^{d}}(pc_l\sigma^{l}(y))p^{\mu_x}\dot w_x\in G(\R)p^{\mu_x}G(\R).$$
If $\langle w_x\sigma^{l-d}w_x\alpha^d,\alpha^{l-d\vee}\rangle =-1$, then $U_{\alpha^{l-d}}$ and $U_{w_x\sigma^{l-d}w_x\alpha^{d}}$ do not commute. We obtain
\begin{eqnarray*}&&g(y)^{-1}b\sigma(g(y))\\&=&U_{\alpha}(-p^{-1}y)U_{\alpha^{l-d}+w_x\sigma^{l-d}w_x\alpha^{d}}(c\sigma^{l}(y)\sigma^{l-d}(y))U_{w_x\sigma^{l-d}w_x\alpha^{d}}(pc_l\sigma^{l}(y))p^{\mu_x}\dot w_x\end{eqnarray*} where $c\in\mathcal{O}_L$ is the product of $c_l,c_{l-d}$ and the structure constant obtained from the commutator of the two root subgroups. Thus $g(y)^{-1}b\sigma(g(y))$ is again in $G(\R)p^{\mu_x}G(\R)$, hence $S_{\preceq\mu}(g(y)^{-1}b\sigma g(y))=\Spec R.$

Now we compute the point $g'$ ``at infinity'' of the affine line $g(y)$.
Let $\R'$ be as in the proof of Proposition \ref{cor_line_immediate}.
Then for $0\leq i\leq l-d-1$ we have (using (\ref{eqn_SL2_computation}))
the following equalities in $G(\R'_L)/G(\R').$
\begin{eqnarray*}
&&\tilde U_{\alpha}^{i+d}(p^{-1}y)\tilde U_{\alpha}^{i}(p^{-1}y)\\
&=&U_{\alpha^{i+d}+\beta^i}(c_{i+d}\sigma^{i+d}(y))U_{\alpha^{i}}(p^{-1}c_i\sigma^i(y))\\
&=&U_{\alpha^{i+d}+\beta^i}(c_{i+d}\sigma^{i+d}(y))U_{-\alpha^{i}}(pc_i^{-1}\sigma^i(y^{-1}))p^{-\alpha^{i\vee}}\\
&=&U_{-\alpha^{i}}(pc_i^{-1}\sigma^i(y^{-1}))p^{-\alpha^{i\vee}}U_{\alpha^{i+d}+\beta^i}(p^{-1}c_{i+d}\sigma^{i+d}(y))\\
&=&U_{-\alpha^{i}}(pc_i^{-1}\sigma^i(y^{-1}))p^{-\alpha^{i\vee}}U_{-(\alpha^{i+d}+\beta^i)}(pc_{i+d}^{-1}\sigma^{i+d}(y^{-1}))p^{-(\alpha^{i+d}+\beta^i)^{\vee}}\\
&=&U_{-\alpha^{i}}(pc_i^{-1}\sigma^i(y^{-1}))U_{-(\alpha^{i+d}+\beta^i)}(c_{i+d}^{-1}\sigma^{i+d}(y^{-1}))p^{-(\alpha+\beta^i+\alpha^{i+d})^{\vee}}\\
\end{eqnarray*}

We define a second element $f(y)\in G(\R_L)$ by setting
\begin{multline*}f(y)=g_x\prod_{i=0}^{l-d-1}\left( U_{-\alpha^{i}}(pc_i^{-1}\sigma^i(y))U_{-(\alpha^{i+d}+\beta^i)}(c_{i+d}^{-1}\sigma^{i+d}(y))p^{-(\alpha^{i}+\beta^i+\alpha^{i+d})^{\vee}}\right)\\
\cdot~\prod_{i=l-d}^{d-1} \left( U_{-\alpha^i}(pc_i^{-1}\sigma^i(y))p^{-\alpha^{i\vee}}\right)\end{multline*}
where the $d$ factors of the two products correspond to different connected components of the Dynkin diagram of $G_{\Omega}$ and can thus be multiplied in any order. The above computation shows that for all $y\neq 0$ we have $f(y) = g(y^{-1})$ in $G(\R'_L)/G(\R').$ In particular, $S_{\preceq\mu}(f^{-1}b\sigma(f))\supseteq \Spec(k[t])\setminus \{0\}$. By Lemma \ref{1.1.3}, this set is Zariski closed. Hence $f(y)$ defines an element of $X_{\preceq\mu}(b)(\R)$. In particular $g(0)=g_x$ and $g'=f(0)$ have the same image in $\pi_0(X_{\preceq\mu}(b))$. Furthermore, $g'=f(0)\in g_xp^{-\sum_{i=0}^{l-1}\alpha^{i\vee}-\sum_{i=0}^{l-d-1}\beta^{i\vee}}$, which proves the lemma in this case.

Next we consider the case where $\Omega$ is of type III
\end{para}

\begin{lemma}\label{lem:casetypeIII} With the above assumptions and notation, suppose that $\Omega$ is of type III and $2d \leq l < 3d.$
\begin{itemize}
\item  If $d \nmid i,(l-i),$ then $w_x\alpha^i=\alpha^i$,  $\langle\beta^i, \mu_x\rangle=0$ and $\langle\alpha^i, \mu_{x}\rangle=0.$
\item If $l= 2d,$ then $\langle\beta,\mu_x\rangle=0,$ $\langle\alpha,\mu_x\rangle=-1,$ $\langle\alpha^d,\mu_x\rangle=0$ and $\langle\alpha^{2d},\mu_x\rangle=1.$
\item If $l > 2d$, then $\langle\beta,\mu_x\rangle=1$, $ \langle\beta^{l-2d},\mu_x\rangle=0,$ $\langle\alpha,\mu_x\rangle=-1$, $\langle\alpha^i,\mu_x\rangle= 0$
for $i=d,2d,l-d,l-2d$, and $\langle \alpha^l,\mu_x\rangle=1$.
\end{itemize}
\end{lemma}
\begin{proof}
The equalities when $d \nmid i,(l-i),$ follow as in the proof of Lemma \ref{lem:casetypeII}, using  Lemma \ref{lemma_D}.

If $l=2d$, then $x+\alpha^{\vee}-\alpha^{l\vee}\in \bar I_{\mu, b}^{M, G}$ implies that $\langle\alpha,\mu_x\rangle=-1$ and $\langle \alpha^{2d},\mu_x\rangle=1.$
Hence $\langle\beta,\mu_x\rangle=0.$ The minimality assumption on $l,$ and the condition $\langle \tilde \alpha, \mu_x \rangle \geq 0,$ then imply
$\langle\alpha^d,\mu_x\rangle=0.$

Suppose $l > 2d.$ As before we have  $\langle\alpha,\mu_x\rangle=-1$ and $\langle w_x\alpha^l,\mu_x\rangle=1$ by Lemma \ref{lemma_newB}.
Then the minimality assumption on $l$ implies $\langle\alpha^i,\mu_x\rangle\leq 0$ for $i=d,2d,$ and also for $i = l-d,l-2d$, using
Lemma \ref{lemma_C}, as above. As
$$\langle \tilde\alpha, \mu_x\rangle=\langle \alpha+\beta+\alpha^d+\alpha^{2d}, \mu_x\rangle\geq 0,$$
 we have $\langle \beta, \mu_x\rangle =1$ and $\langle \alpha^i, \mu_x\rangle=0$ for $i=d, 2d$.

Next we show $\langle \beta^{l-2d}, \mu_x\rangle=0$. Suppose \ $\langle \beta^{l-2d}, \mu_x\rangle=1$, then one checks that $x+\alpha^{\vee}-\alpha^{i\vee}\notin \bar I_{\mu, b}^{M, G}$
for $i=l-d,l-2d$ implies $\langle \alpha^{i}, \mu_x\rangle =-1,$ and hence $\langle \tilde\alpha^{l-2d}, \mu_x\rangle=-1 <0$ which contradicts our standing assumptions.
Therefore $\langle \beta^{l-2d}, \mu_x\rangle=0$, $\langle\alpha^l, \mu_x\rangle=1$ and $\langle \alpha^i, \mu_x\rangle=0$ for $i=l-d, l-2d$.
\end{proof}

\begin{para}{\em Proof of Lemma \ref{lempm3A} continued:}
Suppose $\Omega$ is of type III, and $l = 2d.$ Then using Lemma \ref{lem:casetypeIII} we have
$$\tilde U_{\alpha}^i(p^{-1}y)=\begin{cases}
U_{\alpha^i}(p^{-1}c_i\sigma^i(y))&\text{if }0\leq i<2d\\
U_{\alpha^{2d}}(c_{2d}\sigma^{2d}(y))&\text{if }i=2d\\
\end{cases}$$
with $c_i\in\mathcal{O}_L^{\times}$.
In particular, all these elements commute, and one easily verifies that we have $g(y)^{-1}b\sigma(g(y))\in G(\R)p^{\mu}G(\R)$.

Now suppose that  $\Omega$ is of type III, and $l>2d.$
Using Lemma  \ref{lem:casetypeIII} we obtain
$$\tilde U_{\alpha}^i(p^{-1}y)=\begin{cases}
U_{\alpha^i}(p^{-1}c_i\sigma^i(y))&\text{if }0\leq i<d\\
U_{\alpha^i+\beta^{i-d}}(c_i\sigma^i(y))&\text{if }d\leq i<2d\\
U_{\alpha^i}(c_i\sigma^i(y))&\text{if }2d\leq i<l\\
U_{\alpha^{l}}(pc_l\sigma^{l}(y))&\text{if }i=l\\
\end{cases}$$
with $c_i\in\mathcal{O}_L^{\times}$. When computing $g(y)^{-1}b\sigma(g(y))$ many factors commute and cancel. We obtain
\begin{eqnarray*}
&\ &g(y)^{-1}b\sigma(g(y))\\
&=&U_{\alpha}(-p^{-1}y)\tilde U_{\alpha}^{l-2d}(-p^{-1}y)\tilde U_{\alpha}^{l-d}(-p^{-1}y)\tilde U_{\alpha}^{l}(p^{-1}y)  \tilde U_{\alpha}^{l-d}(p^{-1}y)\tilde U_{\alpha}^{l-2d}(p^{-1}y)p^{\mu_x}\dot w_x\\
&=&U_{\alpha}(-p^{-1}y) U_{\alpha^{l-2d}}(-p^{-1}c_{l-2d}\sigma^{l-2d}(y))\left(U_{\alpha^{l-d}+\beta^{l-2d}}(-c_{l-d}\sigma^{l-d}(y))U_{\alpha^{l}}(pc_l\sigma^l(y))\right.\\
&&\left.\quad\cdot~ U_{\alpha^{l-d}+\beta^{l-2d}}(c_{l-d}\sigma^{l-d}(y))\right)U_{\alpha^{l-2d}}(p^{-1}c_{l-2d}\sigma^{l-2d}(y))p^{\mu_x}\dot w_x\\
&=&U_{\alpha}(-p^{-1}y) U_{\alpha^{l-2d}}(-p^{-1}c_{l-2d}\sigma^{l-2d}(y))U_{\alpha^{l-d}+\beta^{l-2d}+\alpha^l}(pc'\sigma^l(y)\sigma^{l-d}(y))U_{\alpha^{l}}(pc_l\sigma^l(y))\\
&&\quad\cdot~ U_{\alpha^{l-2d}}(p^{-1}c_{l-2d}\sigma^{l-2d}(y))p^{\mu_x}\dot w_x\\
&=&U_{\alpha}(-p^{-1}y)U_{\alpha^{l-2d}+\alpha^{l-d}+\alpha^l+\beta^{l-2d}}(c''\sigma^l(y)\sigma^{l-d}(y)\sigma^{l-2d}(y))\\
&&\quad\cdot~  U_{\alpha^{l-d}+\beta^{l-2d}+\alpha^l}(pc'\sigma^l(y)\sigma^{l-d}(y))U_{\alpha^{l}}(pc_l\sigma^l(y)) p^{\mu_x}\dot w_x
\end{eqnarray*}
with $c',c''\in\mathcal{O}_L.$ The final expression is in $G(\R)p^{\mu}G(\R)$ as $\langle \alpha,\mu_x \rangle = -1.$

The construction and computation of the ``point at infinity'' $g'$ is as in case of type II.
\end{para}
\end{proof}

Before we prove Proposition \ref{pm3weak}, we need one more lemma.

\begin{lemma}\label{lemma_conexity_G_Omega}Let $\Omega\in\Phi_{N,\Gamma}$ be adapted. Then for all $x,x'\in \bar I^{M, G_{\Omega}}_{\mu_{x_0}, b}$, there exists a series of elements $x_1,\cdots, x_r$ in $\bar I^{M, G_{\Omega}}_{\mu_{x_0}, b}$ such that $x_1=x$, $x_r=x'$ and $x_{i+1}-x_i=\alpha^\vee-\alpha'^\vee$ in $\pi_1(M)$ for some $\alpha,\alpha'\in\Omega$ (depending on $i$) for all $1\leq i\leq r-1$.
\end{lemma}
\begin{proof} As the problem only concerns the elements in $\pi_1(M)$,  and $x = x'$ in $\pi_1(G_{\Omega}),$ after replacing $G_{\Omega}$ by the standard Levi subgroup corresponding to the Galois orbit of any connected component of the Dynkin diagram of $G_{\Omega}$ which contains some element in $\Omega$, we may assume that $\Gamma$ acts transitively on the set of connected components of the Dynkin diagram of $G_{\Omega}$. If $v\in X_*(T)$ is a linear combination of coroots of $G_{\Omega}$, let $|v|_{M}=\sum_{\alpha\in\Omega}|n_\alpha|$ where $v=\sum_{\alpha\in\Omega}n_{\alpha}\alpha^{\vee}$ in $\pi_1(M)$.  For $x,x'\in \bar I^{M, G_{\Omega}}_{\mu_{x_0}, b}$, let $d_{M}(x, x'):=|\mu_{x'}-\mu_x|_M$. We will prove the lemma by induction on $d_M(x, x')$.

Suppose $x'\neq x$ in $\pi_1(M)$. Write $x'-x=\sum_{\alpha\in\Omega}n_{\alpha}\alpha^{\vee}$ in $\pi_1(M)$ with $n_{\alpha}\in\Z.$
Then $\sum_{\alpha \in \Omega} n_{\alpha} = 0.$
Write $\Omega^+=\{\alpha\in\Omega|n_\alpha>0\}$ and $\Omega^-=\{\alpha\in\Omega|n_{\alpha}<0\}$.  Let $\mu_{x'}-\mu_x=\sum_{i\in I'}\gamma_i^\vee$ be as in Lemma \ref{pm1ld}. Write $I=\{i\in I'| \gamma_i^\vee\neq 0 \text{ in }\pi_1(M)\},$ $I^+=\{i\in I| \gamma_{i} \text{ is positive}\}$ and $I^-=\{i\in I| \gamma_i \text{ is negative}\}$. Then for any $i\in I^+$ (resp. $i\in I^-$), the image of $\gamma_i^\vee$ in $\pi_1(M)$ is a linear combination of $(\alpha^\vee)_{\alpha\in\Omega^+}$ (resp. $(\alpha^\vee)_{\alpha\in\Omega^-}$).

If all the $(\gamma_i)_{i\in I}$ are in the same connected component of Dynkin diagram of $G_{\Omega}$, then we may replace
$G_{\Omega}$ by the standard Levi subgroup corresponding to that component, and assume that $G_{\Omega}$ has connected Dynkin diagram.
If $\Omega$ is of type I, this implies $x = x'$ contrary to our assumption.
If $\Omega$ is of type II, then $|\Omega|=2$, and $|\Omega^+|=|\Omega^-|=1$. Therefore $|\gamma_i^\vee|_M=1$ for all $i\in I$. Take any $i^+\in I^+$ and $i^-\in I^-$, and define $x_1=x+\gamma_{i^+}^\vee+\gamma_{i^-}^\vee\in\pi_1(M)$. Then $x_1-x=\alpha^\vee-\alpha'^\vee$ for some $\alpha, \alpha'\in\Omega$ and $(\mu_{x_1})_{G_{\Omega}-\dom}=(\mu_x)_{G_{\Omega}-\dom}$, hence $x_1\in \bar I^{M, G_{\Omega}}_{\mu_{x_0}, b}$. Moreover,
\[d_M(x_1, x')=\sum_{i\in I}|\gamma_i^\vee|_M-|\gamma_{i^+}^\vee|_M-|\gamma_{i^-}^\vee|_M=\sum_{i\in I}|\gamma_i^\vee|_M-2<\sum_{i\in I}|\gamma_i^\vee|_M=d_M(x, x').\]By induction hypothesis, we are done.

If $\Omega$ is of type III, then $|\Omega|=3$ and the Dynkin diagram of $G_{\Omega}$ is of type $D_4$. As $|\Omega^+|+|\Omega^-|\leq |\Omega|=3$, we have $|\Omega^+|=1$ or $|\Omega^-|=1$. We may assume that $\Omega^+=\{\alpha\}$ has only one element, the other case being analogous. Then as before, $|\gamma_i^\vee|_M = 1$ for all $i\in I^+$. If there exists $i^-\in I^-$ such that $|\gamma_{i^-}^\vee|_M=1$, then the choice of $i^+$ and $i^-$ as before applies and we are done. Otherwise there exists $i\in I^{-}$ such that $|\gamma_i^\vee|_M\geq 2$. As $x'=x$ in $\pi_1(M)_{\Gamma}$, $\sum_{\alpha\in\Omega^+}|n_{\alpha}|=\sum_{\alpha\in\Omega^-}|n_{\alpha}|\geq 2$. Thus there exist two different elements $i_1, i_2\in I^+$ such that $\gamma_{i_1}^\vee=\gamma_{i_2}^\vee=\alpha^\vee$ in $\pi_1(M)$. This is impossible since $\langle\gamma_{i_1}, \gamma_{i_2}^\vee\rangle=0$.

It remains to consider the case when not all the $\gamma_i$ for $i \in I$ are in the same connected component of Dynkin diagram of $G_{\Omega}$. Choose $i^+\in I^+$ and $i^-\in I^-$ such that $\gamma_{i^+}$ and $\gamma_{i^-}$ are not in the same connected component of $G_{\Omega}$. As $\langle\gamma_{i^+}, \mu_x\rangle=-1$, there exists an $\alpha\in \Omega$ such that $\alpha^\vee\preceq \gamma_{i^+}^{\vee}$ and $\langle\alpha, \mu_x\rangle=-1$. On the other hand suppose $\gamma_{i^-}^\vee=-\alpha_1^\vee-\cdots-\alpha_s^\vee$ in $\pi_1(M)$ for $\alpha_1,\cdots, \alpha_s\in\Omega$. Then by Lemma \ref{lemma_C}(1), $$1=\langle-\gamma_{i^-},\mu_x \rangle\leq \sum_{1\leq j\leq s}\langle w_x\alpha_i, \mu_x\rangle.$$ Therefore there exists $\alpha':=\alpha_i\in\Omega$, such that $\langle w_x\alpha', \mu_x\rangle=1$. Let $x_1 = x+\alpha^\vee-\alpha'^\vee.$
As $\alpha$ and $\alpha'$ are not in the same connected component of $G_{\Omega}$, we have
$\mu_{x_1} = s_{\alpha}s_{w_x\alpha'}(\mu_x),$ by Lemma \ref{lemma_C}, so $x_1\in \bar I^{M, G}_{\mu_{x_0}, b}.$ Hence $x_1\in \bar I^{M, G_{\Omega}}_{\mu_{x_0}, b}$
by Lemma \ref{lemma_bar_I}.  As $d_{M}(x_1, x')<d_{M}(x, x'),$ we are done by induction.
\end{proof}

In the following we will prove Proposition \ref{pm3weak} by subdividing it into several particular cases which we prove in the form of Lemmas \ref{lempm3B},  \ref{lempm3C} and \ref{lempm3D}.

\begin{lemma}\label{lempm3B}
Proposition \ref{pm3weak} holds under the following additional hypotheses.
\begin{itemize}
\item The set $\bar I^{M, G_{\Omega}}_{\mu_{x_0}, b}$ has at least two elements.
\item $\Omega$ is of type I or $\langle\tilde\alpha^i,\mu_x\rangle \geq 0$ for all $i\in\mathbb N$ and all $x\in \bar I^{M, G_{\Omega}}_{\mu_{x_0}, b}$.
\end{itemize}
\end{lemma}
\begin{proof}
As the set $\bar I^{M, G_{\Omega}}_{\mu_{x_0}, b}$ has at least two elements, by Lemma \ref{lemma_bar_I} and Lemma \ref{lemma_conexity_G_Omega}, there exists $x_1=x_0+\alpha^{\vee}-\alpha^{l\vee}\in \bar I^{M, G_{\Omega}}_{\mu_{x_0}, b}$ with $\alpha\in\Omega$ and $0<l<n=|\Omega|$. For any $g_1\in X^M_{\mu_{x_0}}(b)(W(\bar k))$,  by applying Lemma \ref{lempm3A} to the pair $(x_0,x_1)$, we obtain a $g'\in X_{\mu_{x_1}}^M(b)$ such that $g_1\sim g'$ and $w_M(g')=w_M(g_1)-\sum_{i=0}^{l-1}\alpha^{i\vee}$. As $\bar I^{M, G_{\Omega}}_{\mu_{x_0}, b}=\bar I^{M, G_{\Omega}}_{\mu_{x_1}, b}$, we apply again Lemma \ref{lempm3A} to the pair $(x_1,x_0)$. We obtain a $g_2\in X_{\mu_{x_0}}^M(b)(W(\bar k))$ such that $g'\sim g_2$ and $w_M(g_2)=w_M(g')-\sum_{i=l}^{n-1}\alpha^{i\vee}$. Then $g_2$ is the desired element of Proposition \ref{pm3weak} for $x=x_0$.
\end{proof}

\begin{lemma}\label{lempm3C}
Proposition \ref{pm3weak} holds under the following additional hypotheses.
\begin{itemize}
\item The set $\bar I^{M, G_{\Omega}}_{\mu_{x_0}, b}=\{x_0\}$ contains only one element.
\item $\Omega$ is of type I or $\langle\tilde\alpha,\mu_{x_0}\rangle \geq 0$ for all $\alpha\in\Omega$.
\end{itemize}
\end{lemma}

\begin{ex}\label{ex_for_lempm3C}Here is an example where all the hypotheses of Lemma \ref{lempm3C} are satisfied. Let $G$ be a unitary similitude group such that $G_{L}\simeq \mathrm{GL}_5\times \mathbb{G}_{m, L}$ with standard simple roots $\beta_i=e_i-e_{i+1}$ for $i=1, 2, 3, 4$. The group $\Gamma=\{\mathrm{Id}, \sigma\}$ acts on $G_{L}$  with $\sigma\beta_i= \beta_{5-i}$ for $i=1, \cdots, 4$. The Levi subgroup $M$ is defined by the roots $\beta_1$ and $\beta_4$. The cocharacter $\mu_{x_0}$ is defined as follows:
\begin{eqnarray*}
\mu_{x_0}: \mathbb{G}_{m, L} &\rightarrow& G_{L} \simeq \mathrm{GL}_5\times \mathbb{G}_{m, L} \\ y&\mapsto& (\mathrm{diag}(y, y, 1, y, 1), y)
\end{eqnarray*}
Then $\mu_{x_0}$ determines $x_0$ and $\mu$. Therefore it determines $w_{x_0}$ and $b=p^{\mu_{x_0}}\dot w_{x_0}$. Let $\alpha= \beta_3$. One can check that the datum $(M, G, \Gamma\alpha, b, \mu)$ satisfies all the conditions of Lemma \ref{lempm3C}.
\end{ex}

\begin{proof}[Proof of Lemma \ref{lempm3C}]
For simplicity, we write $x$ for $x_0$, and let $\alpha\in\Omega$ such that $\langle\alpha,\mu_x\rangle <0$. Let $g=g_x\in \bar I_{\mu_x, b}^{M, G_{\Omega}}$.

Suppose that $0<i<n$ and that $\alpha$ and $\alpha^i$ are in different components of the Dynkin diagram of $G_{\Omega}.$
Since $x+\alpha^{\vee}-\alpha^{i\vee}\notin \bar I_{\mu_x, b}^{M, G_{\Omega}}$ we have $\langle w_x\alpha^i,\mu_x\rangle \leq 0$. By assumption $\nu_b$ is $G$-dominant, and so $\langle\alpha,\nu_b\rangle\geq 0.$
Since $b$ is basic in $M,$ $\nu_b$ is the $W_M$-average of the Galois-average of $\mu_x.$
Using $\langle w_x\alpha^i,\mu_x\rangle \leq 0$ and Lemma \ref{lemma_w_x_alpha minimal}, there exists $\alpha_1\in\Gamma\alpha$ which is in the same connected component of the Dynkin diagram of $G_{\Omega}$ as $\alpha$ such that $\langle w_x\alpha_1,\mu_x\rangle > 0.$ Since $x-\alpha_1^\vee+\alpha^{i\vee}\notin \bar I_{\mu_x, b}^{M, G_{\Omega}}$, we obtain that $\langle \alpha^i,\mu_x\rangle \geq 0.$ Hence by Lemma \ref{lemma_w_x_alpha minimal},  $\alpha^i=w_x\alpha^i,$ and for every positive
root $\beta$ in $M,$ we have $-\langle\alpha^i,\beta^{\vee}\rangle \langle \beta, \mu_x \rangle  \leq \langle s_{\beta} \alpha^i,\mu_x \rangle \leq 0.$
In particular, if $\beta$ is a maximal root in $M,$ such that $\alpha^i$ and $\beta$ are contained in the same component of the Dynkin diagram
of $G_{\Omega},$ then $\langle\alpha^i,\beta^{\vee}\rangle< 0$ so $\langle \beta, \mu_x \rangle = 0,$ as $\mu_x$ is $M$-dominant. This implies
$\langle \beta, \mu_x \rangle = 0$ for every positive root $\beta$ in the same component as $\alpha^i.$
Thus $\mu_x$ and $w_x$ are central in the connected component of $G_{\Omega}$ containing $\alpha^i.$
\smallskip

\noindent{\it Case 1: $\Omega$ is of type I.}
By the above, we have $\mu_x$ and $w_x$ are central in the connected component of $G_{\Omega}$ containing $\alpha^i$
for $0<i<n$, and $\alpha_1=\alpha$. In particular $\langle w_x\alpha^i,\mu_x\rangle = 0$, and $\langle w_x\alpha, \mu_x\rangle>0$.

{\it Claim: $U_{\alpha}$ and $U_{w_x(\alpha)}$ commute. }

 By Lemma \ref{lemma_Weyl_orbit}, there exist positive roots $(\beta_i)_{i\in J}$ in $M$ such that
 \begin{itemize}
 \item $\langle \beta_i, \beta_j^{\vee}\rangle=0$ for all $i\neq j\in J$.

 \item $w_x\alpha=(\prod_{j\in J}s_{\beta_i})(\alpha)$ and $\langle\alpha, \beta_i^{\vee}\rangle <0$ for all $i\in J$.

 \item $|w_x\alpha|=|\alpha|+\sum_{i\in J} |\langle \alpha, \beta^\vee_i\rangle|\cdot |\beta_i|$.
 \end{itemize}
By the hypothesis of Proposition \ref{pm3weak}, $\langle \beta_i, \alpha^\vee\rangle =-1$ for all $i\in J$. And by Lemma \ref{lemma_w_x_alpha minimal}, $\langle\beta_i, \mu_x\rangle=1$ for all $i\in J$ (Indeed, If $\langle\beta_i, \mu_x\rangle=0$ for some $i$, then $\langle s_{\beta_i}w_x\alpha, \mu_x\rangle=\langle w_x\alpha, \mu_x\rangle$. But $s_{\beta_i}w_x\alpha \preceq w_x\alpha$, so this contradicts the minimality of $w_{x}\alpha$ in Lemma \ref{lemma_w_x_alpha minimal}). Therefore

\[2=\langle w_x\alpha, \mu_x\rangle-\langle\alpha, \mu_x\rangle=-\sum_{i\in J}\langle\alpha, \beta_i^{\vee}\rangle\cdot\langle\beta_i, \mu_x\rangle=-\sum_{i\in J}\langle\alpha, \beta_i^{\vee}\rangle.\]In particular the cardinality of the set $J$ is at most 2. Furthermore we have
\[\langle w_x\alpha, \alpha^{\vee}\rangle =\langle \alpha-\sum_{i\in J}\langle \alpha, \beta_i^{\vee}\rangle\beta_i, \alpha^{\vee}\rangle=2+\sum_{i\in J}\langle\alpha, \beta_i^{\vee}\rangle =0.\]

Thus, if $\alpha+w_x\alpha$ is a root, then it is longer than $\alpha$ and hence longer than $\beta_i$ for all $i\in J$. And so is the root $s_{\alpha}(\alpha+w_x\alpha)$. As
\[s_{\alpha}(\alpha+w_x\alpha)=w_x\alpha-\alpha=-\sum_{i\in J}\langle\alpha, \beta_i^{\vee}\rangle \beta_i \]
is a root in $M$, it should have the same length as $\beta_i$ for any $i\in J$. We get a contradiction. Therefore $\alpha+w_x\alpha$ cannot be a root and this finishes the proof of the Claim.

Let $R = \bar k[y]$ and $\R$ the $R$-frame chosen in \ref{para:notnsetup}. We define $g(y)\in G(\R_L)/G(\R)$ by
$$g(y)=g_{x}U_{\alpha}(p^{-1}y)U_{\alpha^1}(p^{-1}c_1\sigma(y))\dotsm U_{\alpha^{n-1}}(p^{-1}c_{n-1}\sigma^{n-1}(y))$$
where the $c_i\in\mathcal{O}_L^{\times}$ are such that $\dot w_x \sigma U_{\alpha^i}(c_i\sigma^i(y))\dot w_x^{-1}=U_{\alpha^{i+1}}(c_{i+1}\sigma^{i+1}(y))$ and $c_0=1$. In type I, all of these root subgroups commute. Using the above equations to compute the conjugation action of $b_x=p^{\mu_x}\dot w_x$ on these root subgroups we obtain

\begin{eqnarray*}
g^{-1}b\sigma(g) & = & U_{\alpha}(-p^{-1}y)p^{\mu_x}\dot w_xU_{\alpha}(p^{-1}\sigma(c_{n-1})\sigma^{n}(y)) \\
& = & U_{\alpha}(-p^{-1}y)p^{\mu_x}U_{w_x\alpha}(p^{-1}c_{\alpha}^x\sigma(c_{n-1})\sigma^{n}(y)) \dot w_x \\
& = & U_{w_x\alpha}(c_{\alpha}^x\sigma(c_{n-1})\sigma^{n}(y))p^{\mu_x}U_{\alpha}(-y) \dot w_x
\end{eqnarray*}
and the final expression is in $G(\R)p^{\mu_x}G(\R).$
Here, in the last equality, we have used that $U_{\alpha}$ and $U_{w(\alpha)}$ commute and that
$\langle\alpha,\mu_x\rangle <0$ and $\langle w_x\alpha,\mu_x\rangle >0.$
 Thus
$$S_{\preceq\mu}(g(y)^{-1}b\sigma g(y))=\Spec R. $$
In the usual way (as for example in the proof of Lemma \ref{lempm3A}) we can extend this family to a ``projective line'' and use that the point $g(0)$ and the point $g'$ ``at infinity'' are in the same connected component of $X_{\preceq \mu}^G(b)$. Here one obtains $g'\in g_xp^{-\sum_{i=0}^{n-1}\alpha^{i\vee}}K,$ which finishes the proof in this case.

Now we will deal with case when $\Omega$ is of type II or III. After replacing $G_{\Omega}$ by the standard Levi subgroup corresponding to the Galois orbit of any connected component of the Dynkin diagram of $G_{\Omega}$ containing some element of $\Omega$, we may assume that $\Gamma$ acts transitively on the Dynkin diagram of $G_{\Omega}$. As we only use $G_{\Omega}$ to distinguish several cases, this modification does not change the following argument.

\smallskip
\noindent{\it Case 2: $\Omega$ is of type II.}

By assumption $\langle\alpha+\beta+\alpha^d,\mu_x\rangle \geq 0$, hence $\langle\beta+\alpha^d,\mu_x\rangle=1$.
 We have that $\mu_x$ and $w_x$ are central on all connected components of the Dynkin diagram of $G_{\Omega}$ except for the one containing $\alpha$ and $\alpha^d$.

\begin{lemma}\label{lem:Claim1}
We have
\begin{itemize}
\item $\langle w_x\alpha^d, \mu_x \rangle = 1$
\item $w_x\alpha^d = \alpha^d + \beta$
\item $\beta\neq 0$ if and only if $\langle\beta,\mu_x\rangle=1$
\item $\langle w_x\alpha^d,\alpha^{\vee}\rangle=-1$
\end{itemize}
In particular $\alpha+w_x\alpha^d$ is equal to the root $\tilde \alpha,$ and $\alpha,w_x\alpha^d$ do not commute.

\end{lemma}
\begin{proof}
As $\alpha^d$ is $M$-anti-dominant we have $$\langle w_x\alpha^d,\mu_x\rangle=\langle \alpha^d,w_{M,0}\mu_x\rangle \geq \langle\alpha^d+\beta,\mu_x\rangle=1.$$

If $\beta=0$ then $\langle\alpha^d,\mu_x\rangle=1$, and thus $\langle\alpha^d,\mu_x\rangle=\langle w_x\alpha^d,\mu_x\rangle$. Thus by Lemma \ref{lemma_w_x_alpha minimal}, $w_x\alpha^d=\alpha^d=\alpha^d+\beta.$ Suppose $\beta\neq 0.$ If $\langle \beta,\mu_x\rangle = 0$, then $\langle\alpha^d,\mu_x\rangle= 1$. This implies $x+\alpha^{\vee}-\alpha^{d\vee}\in \bar I_{\mu_x, b}^{M, G_{\Omega}}$
(use $\langle \alpha, \mu_x \rangle < 0$ and $\alpha$ $M$-antidominant),
which is impossible. Thus $\langle\beta,\mu_x\rangle=1$ and
 \[\langle w_x\alpha^d, \mu_x\rangle=1=\langle \alpha^d+\beta, \mu_x\rangle > 0=\langle \alpha^d, \mu_x\rangle.\]
 Hence by Lemma \ref{lemma_w_x_alpha minimal}, $w_x\alpha^d=\alpha^d+\beta$.
Now the formula $\langle w_x\alpha^d,\alpha^{\vee}\rangle=-1$ is clear, and the final claim follows.
\end{proof}


\noindent{\it Case 2.1: $w_x\tilde\alpha\neq \tilde \alpha$.}

\begin{lemma}\label{lem:Claim 2} $\langle w_x\tilde\alpha,\mu_x\rangle=1$ and $\langle w_x\tilde\alpha,\alpha^{\vee}\rangle\geq 0$.
\end{lemma}
\begin{proof}
We check the lemma according to the type of the Dynkin diagram of $G_{\Omega}$ which can be only $A_m$, $D_m$ or $E_6$.

Suppose the Dynkin diagram of $G_{\Omega}$ is of type $A_m.$ By Lemma \ref{lem:Claim1} we have
 $w_x\tilde\alpha=w_x(\alpha+\beta+\alpha^d)=w_x\alpha-\beta+\alpha^d+\beta$. By the assumption of Case 2.1 this implies that $w_x\alpha\neq\alpha+\beta$. Thus by Lemma \ref{lemma_w_x_alpha minimal}, $\langle w_x\alpha,\mu_x\rangle>\langle \alpha+\beta, \mu_x\rangle$. Combined with the fact that if $\beta=0$, $\langle w_x\alpha,\mu_x\rangle\leq \langle\alpha, \mu_x\rangle+1$, we have
$$\langle w_x\alpha,\mu_x\rangle=\begin{cases}
1&\text{if }m \text{ is odd, i.e.~}\beta\neq 0\\
0&\text{if }m \text{ is even, i.e.~}\beta=0.
\end{cases}$$
Furthermore, as $\langle \beta,\mu_x \rangle = 1$ if $\beta \neq 0,$
 $$\langle w_x\tilde\alpha,\mu_x\rangle
=\langle w_x\alpha+\alpha^d,\mu_x\rangle=1$$ and $$\langle w_x\tilde\alpha,\alpha^{\vee}\rangle=\langle w_x\alpha+\alpha^d,\alpha^{\vee}\rangle.$$ If $\beta\neq 0$, this sum is $\geq 0+0$, if $\beta=0$, it is $\geq 1-1$, thus in all cases non-negative.

If the Dynkin diagram of $G_{\Omega}$ is of type $D_m$, we denote the simple roots by $\beta_r,\dotsc,\beta_1,\beta, \alpha,\alpha^d$ where $\beta$ is the simple root with three neighbors $\beta_1,\alpha,\alpha^d$, and $\beta_i$ is a neighbor of $\beta_{i-1}$ for all $i>1$. By Lemma \ref{lem:Claim1} we have $\langle \beta,\mu_x\rangle =1$. As $\mu_x$ is $M$-dominant and minuscule this implies $\langle \beta_i,\mu_x\rangle =0$ for $i=1,\dotsc,r$. Then the explicit definition of $w_x$ implies that $w_x=s_{\beta}s_{\beta_1}\dotsm s_{\beta_r}$.
Thus $w_x\tilde\alpha=w_x(\alpha+\beta+\alpha^d)=\alpha+\beta_1+2\beta+\alpha^d$. Hence $$\langle w_x\tilde\alpha,\mu_x\rangle=\langle \alpha+\beta_1+2\beta+\alpha^d,\mu_x\rangle=-1+0+2+0=1$$ and $$\langle w_x\tilde\alpha,\alpha^{\vee}\rangle=\langle \alpha+\beta_1+2\beta+\alpha^d,\alpha^{\vee}\rangle=2+0-2+0=0.$$

If the Dynkin diagram of $G_{\Omega}$ is of type $E_6$, the simple root $\beta$ has again three neighbors in the Dynkin diagram denoted $\alpha,\alpha^d,$ and $\beta_{-1}$. Denote the other neighbors of $\alpha,\alpha^d$ by $\gamma,\gamma^d$, respectively. As $\mu_x$ is $G$-minuscule and $M$-dominant, and $\langle\beta,\mu_x\rangle=1$ we have $\langle\beta_{-1},\mu_x\rangle=0$ and likewise $\langle\gamma^d,\mu_x\rangle=0,$ as $\langle \alpha^d, \mu_x \rangle = 0.$ If $\langle \gamma, \mu_x\rangle=0$, then $w_x=s_{\beta}$ and hence $w_x\tilde\alpha= \alpha+w_x\alpha^d$ which contradicts the hypothesis. Therefore we have $\langle\gamma,\mu_x\rangle=1$. We have that $\gamma+\alpha+2\beta+\alpha^d+\beta_{-1}$ is a root, but $\langle \gamma+\alpha+2\beta+\alpha^d+\beta_{-1},\mu_x\rangle=2$, in contradiction to the fact that $\mu_x$ is minuscule. Thus this subcase may not occur, which finishes the proof of the Lemma.
\end{proof}

\begin{para}{\em Proof of Lemma \ref{lempm3C}, Case 2 continued.}
We remind the reader that, by Lemma \ref{lem:Claim1} we have
$$ w_x\sigma^d(\alpha+w_x\alpha^d) = w_x\sigma^d(\tilde \alpha) = w_x\tilde \alpha.$$

For $R$ as above we define $g(y)\in G(\R_L)/G(\R)$ as
\begin{align*}g(y)=&g_{x}U_{\alpha}(p^{-1}y)U_{\alpha^1}(p^{-1}c_1\sigma(y))\dotsm U_{\alpha^{d-1}}(p^{-1}c_{d-1}\sigma^{d-1}(y))\\
&\cdot~U_{\alpha+w_x\alpha^d}(-p^{-1}c'_0y\sigma^d(y))\dotsm U_{\sigma^{d-1}(\alpha+w_x\alpha^{d})}(-p^{-1}c'_{d-1}\sigma^{d-1}(y)\sigma^{2d-1}(y))\end{align*} where the $c_i\in\mathcal{O}_L^{\times}$ for $i=0,\dotsc, d$ are such that $$\dot w_x \sigma U_{\alpha^i}(c_i\sigma^i(y))\dot w_x^{-1}=U_{w_x\alpha^{i+1}}(c_{i+1}\sigma^{i+1}(y))$$ and $c_0=1$.
Furthermore $c'_0\in \mathcal{O}_L$ is such that
$$U_{\alpha}(y)U_{w_x\alpha^d}(z)=U_{\alpha+w_x\alpha^d}(c'_0yz)U_{w_x\alpha^d}(z)U_{\alpha}(y),$$
and the $c'_i\in\mathcal{O}_L$ for $i=1, \dotsc, d$ are such that $$\dot w_x \sigma U_{\sigma^{i-1}(\alpha+w_x\alpha^{d})}(c'_{i-1}\sigma^{i-1}(y))\dot w_x^{-1}=U_{w_x\sigma^{i}(\alpha+w_x\alpha^{d})}(c'_{i}\sigma^{i}(y)).$$

\end{para}

We remark that  $c'_i\in \mathcal{O}_L^{\times}$ for $i=0,\dotsc, d.$ Indeed, it suffices to check this for $i=0.$
If $c'_0$ is in $p\O_L,$ then the root groups $U_{\alpha}$ and $U_{w_x\alpha^d}$ commute in $G\otimes \bar k.$
Since all the roots of $G_{\Omega}$ have the same length, this is impossible, by \cite{SGA3} XXIII Prop. 6.5.

Now we can compute the conjugation action of $b_x=p^{\mu_x}\dot w_x$ on the root subgroups by using the above equations. We obtain
\begin{eqnarray*}
g^{-1}b\sigma(g) &=&U_{\alpha+w_x\alpha^d}(p^{-1}c_0'y\sigma^d(y))U_{\alpha}(-p^{-1}y)U_{w_x\alpha^d}(c_d\sigma^d(y))\\
&&\quad \cdot U_{w_x\sigma^d(\alpha+w_x\alpha^d)}(-c_d'\sigma^d(y)\sigma^{2d}(y))p^{\mu_x}\dot w_x\\
&=&U_{w_x\alpha^d}(c_d\sigma^d(y))U_{\alpha}(-p^{-1}y)U_{w_x\sigma^d(\alpha+w_x\alpha^d)}(-c_d'\sigma^d(y)\sigma^{2d}(y))p^{\mu_x}\dot w_x.
\end{eqnarray*}
As $\langle w_x\sigma^d(\alpha+w_x\alpha^d),\alpha^{\vee}\rangle\geq 0$, the corresponding root subgroups commute and the above expression is indeed in $Kp^{\mu_x}K$. Thus $S_{\preceq\mu}(g(y)^{-1}b\sigma g(y))=\Spec R.$ As before we can extend this family to a ``projective line'' and use that the point $g(0)$ and the point $g'$ ``at infinity'' are in the same connected component of $X_{\mu}^G(b)$. It remains to compute $g'$. Let $\R' = \O_L\langle y,y^{-1} \rangle$ be the frame introduced above.
We consider each connected component of the Dynkin diagram of $G_{\Omega},$ separately and for $0\leq i\leq d-1$
we compute in $G(\R'_L)/G(\R'):$
\begin{align*}
&U_{\alpha^i}(p^{-1}c_i\sigma^i(y))U_{\sigma^i(\alpha+w_x\alpha^d)}(-p^{-1}c_i'\sigma^i(y\sigma^d(y)))\\
=&U_{\alpha^i}(p^{-1}c_i\sigma^i(y))U_{-\sigma^i(\alpha+w_x\alpha^d)}(-p(c_i')^{-1}\sigma^i(y\sigma^d(y))^{-1})p^{-\sigma^i(\alpha+w_x\alpha^d)^{\vee}}\\
=&U_{-\sigma^i(\alpha+w_x\alpha^d)}(-p(c_i')^{-1}\sigma^i(y^{-1}\sigma^d(y^{-1})))U_{-\sigma^iw_x\alpha^{d}}(d_i\sigma^{d+i}(y^{-1}))p^{-\sigma^i(\alpha+w_x\alpha^d)^{\vee}}
\end{align*}
for some $d_i\in\mathcal{O}_L^{\times}$. Here in the last line we have used that the root groups
$U_{\alpha^i}$ and $U_{-\sigma^iw_x\alpha^d}$ commute, and that $\langle \alpha^i, \sigma^i(\alpha+w_x\alpha^d) \rangle = 1.$
Thus we define the second family $f(y)$ as
$$f(y)=\prod_{i=0}^{d-1}U_{-\sigma^i(\alpha+w_x\alpha^d)}(-p(c_i')^{-1}\sigma^i(y\sigma^d(y)))U_{-\sigma^iw_x\alpha^{d+i}}(d_i\sigma^{d+i}(y))p^{-\sigma^i(\alpha+w_x\alpha^d)^{\vee}}.$$ In particular $g'=f(0)=g_xp^{-\sum_{i=0}^{d-1}\sigma^i(\alpha+w_x\alpha^d)^{\vee}}$ is as claimed.

\smallskip
\noindent{\it Case 2.2: $w_x\tilde \alpha = \tilde \alpha.$}

Let $c_0=c_0'=1$ and let $c_i,c_i'$ be defined inductively by
\begin{eqnarray*}
\dot w_x \sigma U_{\alpha^i}(c_i\sigma^i(y))\dot w_x^{-1}&=&U_{w_x\alpha^{i+1}}(c_{i+1}\sigma^{i+1}(y))\\
\dot w_x \sigma U_{\sigma^{i}(\alpha+w_x\alpha^{d})}(c'_i\sigma^i(y))\dot w_x^{-1}&=&U_{\sigma^{i+1}(\alpha+w_x\alpha^{d})}(c'_{i+1}\sigma^{i+1}(y)).
\end{eqnarray*}
Furthermore let $\tilde{c}\in \mathcal{O}_L$ be such that
$$U_{\alpha}(y)U_{w_x\alpha^d}(z)=U_{\alpha+w_x\alpha^d}(\tilde c yz)U_{w_x\alpha^d}(z)U_{\alpha}(y).$$
We evidently have $c_i' \in \O_L^\times,$ and $\tilde c \in \O_L^\times$ by the same argument as in Case 2.1 above.
We now define the frame we will need.

\begin{lemma}\label{lem:curveprops} Let $h = z- \tilde cy^{qd+1} - c_d'z^{qd},$ and set $A = \O_L[y,z]/h.$
Then $\Spec A$ is a dense Zariski open in a smooth, proper curve $X$ over $\O_L$ having
geometrically connected fibres.
\end{lemma}
\begin{proof} Let $h' = z^{-qd} - \tilde c w^{qd+1} - c_d'z^{-1},$ and
$A' = \O_L[w,z^{-1}]/h'.$ Then $A' \subset A[z^{-1}]$ by sending $w$ to $yz^{-1},$
and $\Spec A,$ $\Spec A'$ glue along $\Spec A[z^{-1}]$ into a proper flat curve $X$ over $\O_L,$
which admits a finite map $\eta: X \rightarrow \mathbb P^1$ given by the function $z.$

Note that $\frac{\partial h}{\partial z} = 1$ in $A\otimes\bar k,$ and
$\frac{\partial h'}{\partial z^{-1}} = -c_d' \neq 0$ in $A'\otimes\bar k.$ Hence $X$ is a smooth curve.
Since $\eta$ is totally ramified over $z=0,$ $X$ has geometrically irreducible fibres.
\end{proof}

\begin{para} {\em Proof of Case 2.2 continued:} Let $x_0 \in X$ be the point given by $y=z=0,$ and $x_1 \in X$ the point
given by $z^{-1} = w = 0,$ using the covering of $X$ introduced in Lemma \ref{lem:curveprops} above.
Choose a map $\xi: X \rightarrow \mathbb P^1$ such that $\xi$ is \'etale above $\xi(x_0)$ and $\xi(x_1).$
To see that this is possible, choose points $x_2, \dots, x_r$ for $r$ big enough (e.g. $r\geq 2g$ with $g$ the genus of $X\otimes\bar k$) and such that the $x_i$ are distinct in $X\otimes \bar k$
for $i=0,\dots, r,$ 
By the Riemann-Roch theorem
there is a section $g_0 \in \Gamma(X\otimes \bar k, \O(\sum_j x_j))$ which does not vanish at any $x_i.$
Lift $g_0$ to $g \in \Gamma(X, \O(\sum_j x_j)).$ Then as a meromorphic function on $X,$ $g$ has a simple pole at each $x_i$
with a residue which is non-zero mod $p.$ Take $\xi$ to be given by $g^{-1}.$ Then $\xi$ is \'etale over $0,$ and $\xi(x_i) = 0$ for all $i.$

Let $\widehat X$ and $\widehat {\mathbb P}^1$ denote the $p$-adic completion of $X$ and $\mathbb P^1.$
Let $U_0 \subset \mathbb P^1\otimes \bar k$ be the open subset where $\xi\otimes \bar k$ is \'etale, and let
$ U \subset \widehat {\mathbb P}^1$
denote the corresponding formal open affine, and $Y = \xi^{-1}(U) \subset \widehat X.$ Since $U_0$ is stable by Frobenius on
$\mathbb P^1\otimes \bar k,$ $U$ is stable by any Frobenius lift on $\widehat {\mathbb P}^1.$ Fix such a lift.
Since $Y \rightarrow U$ is finite \'etale, by Lemma \ref{1.1.5}, the Frobenius lift on $W$ lifts uniquely to a Frobenius lift on $Y = \Spf \R.$
We denote by $\sigma$ the corresponding $q$-Frobenius on $\R.$

It will be convenient to denote by $\Spf \R_0$ and $\Spf \R_1$ the formal affine subsets of $Y,$ which are the complements
of the mod $p$ reductions of $x_1$ and $x_0$ respectively. Thus $z,y  \in \R_0$ and $z^{-1},w = yz^{-1} \in \R_1.$
Likewise, we denote by $\Spf \R'$ the complement of $\{x_0,x_1\}$ in $\Spf \R.$
Define an element $g \in G(\R_{0,L})$ by

\begin{eqnarray*}g &=&g_xU_{\alpha}(p^{-1}c_0y)\dotsm U_{\alpha^{d-1}}(p^{-1}c_{d-1}\sigma^{d-1}(y))U_{\alpha+w_x\alpha^d}(-p^{-1}c_0'z)\dotsm \\
&&\quad \cdot~ U_{\sigma^{d-1}(\alpha+w_x\alpha^d)}(-p^{-1}c_{d-1}'\sigma^{d-1}(z)),\end{eqnarray*}
 Recall that $\langle w_x\alpha_d,\mu_x\rangle=1$ and that $w_x\sigma^d(\alpha+w_x\alpha^d)=\alpha+w_x\alpha^d$. We obtain
\begin{eqnarray*}
&&g^{-1}b\sigma(g)\\&=& U_{\alpha+w_x\alpha^d}(p^{-1}z)U_{\alpha}(-p^{-1}y)U_{w_x\alpha^d}(c_d\sigma^d(y))U_{w_x\sigma^d(\alpha+w_x\alpha^d)}(-p^{-1}c_d'\sigma^d(z))p^{\mu_x}\dot w_x\\
&=&U_{\alpha+w_x\alpha^d}(p^{-1}z-p^{-1}\tilde cy\sigma^d(y)-p^{-1}c_d'\sigma^d(z))U_{w_x\alpha^d}(c_d\sigma^d(y))U_{\alpha}(-p^{-1}y)p^{\mu_x}\dot w_x
\end{eqnarray*}
Recall that $\langle \alpha+w_x\alpha^d,\alpha^{\vee}\rangle=2-1=1$, thus $\alpha$ and $\alpha+w_x\alpha$ commute. For the second equality above we use that
$$w_x\sigma^d(\alpha+w_x\alpha^d)= w_x\tilde \alpha = \tilde \alpha = \alpha+w_x\alpha^d$$
commutes with $w_x\alpha^d$ and $\alpha$, and the definition of $\tilde c$. Since $\langle \alpha, \mu_x \rangle = -1,$ and
$$ z-\tilde cy\sigma^d(y)-c_d'\sigma^d(z) = z-\tilde cy^{dq+1}-c_d'z^{dq} = h = 0 \quad \text{in $\R_0/p\R_0$} $$
we see that $g^{-1}b\sigma(g) \in G(\R_0)p^{\mu_x}G(\R_0).$

To define and compute a ``point at infinity'' of the above family we first compute for $0\leq i<d$
\begin{align*}
&U_{\alpha^i}(p^{-1}c_i\sigma^i(y))U_{\sigma^i(\alpha+w_x\alpha^d)}(-p^{-1}c_i'\sigma^i(z))G(\R')\\
=&U_{\alpha^i}(p^{-1}c_i\sigma^i(y))U_{-\sigma^i(\alpha+w_x\alpha^d)}(-p(c_i')^{-1}\sigma^i(z^{-1}))p^{-\sigma^i(\alpha+w_x\alpha^d)^{\vee}}G(\R')\\
\intertext{moving $U_{\alpha^i}$ to the right we obtain}
=&U_{-\sigma^i(w_x\alpha^d)}(-d_i\sigma^i(yz^{-1}))U_{-\sigma^i(\alpha+w_x\alpha^d)}(-p(c_i')^{-1}\sigma^i(z^{-1}))p^{-\sigma^i(\alpha+w_x\alpha^d)^{\vee}}G(\R')
\end{align*}
for some $d_i\in\mathcal{O}_L$.

Define an element $f \in G(\R_{1,L})$ by setting
$$f =g_x \prod_{i=0}^{d-1}U_{-\sigma^i(w_x\alpha^d)}(-d_i\sigma^i(w))U_{-\sigma^i(\alpha+w_x\alpha^d)}(-p(c_i')^{-1}\sigma^i(z))p^{-\sigma^i(\alpha+w_x\alpha^d)^{\vee}}.$$
Then we have $f = g$ in $G(\R'_L).$ By what we saw above, $S_{\preceq\mu}(f^{-1}b\sigma(f))$ contains the open and dense subset
$\Spec \R'/p\R'$ of $\Spec \R_1/p\R_1.$ By Lemma \ref{1.1.3}, $S_{\preceq\mu}(f^{-1}b\sigma(f))$ is Zariski closed.
Hence $f$ defines an element of $X_{\preceq\mu}(b)(\R')$. In particular $g(x_0) = g_x\in X^M_{\mu_x}(b)$ and
$g'=f(x_1)=g_x p^{-\sum_{i=0}^{d-1}\sigma^i(\alpha+w_x\alpha^d)^{\vee}}$ are in the same connected component of $X_{\preceq\mu}(b)$.
\qed
\end{para}

\noindent{\it Case 3: $\Omega$ is of type III.}
The same argument as for the preceeding cases shows that $\mu_x$ and thus $w_x$ are central on all connected components of the Dynkin diagram of $G_{\Omega}$ except for the one containing $\alpha,\alpha^d,\alpha^{2d}$.  As $x+\alpha^{\vee}-\alpha^{d\vee}\notin \bar I_{\mu_x, b}^{M, G_{\Omega}}$, $x+\alpha^{\vee}-\alpha^{2d\vee}\notin \bar I_{\mu_x, b}^{M, G_{\Omega}}$, we have $\langle\alpha^d, \mu_x\rangle\leq 0$ and $\langle\alpha^{2d}, \mu_x\rangle\leq 0$.  Combined with the fact that $\langle \tilde{\alpha}, \mu_x\rangle\geq 0$,   we obtain $\langle\beta,\mu_x\rangle=1$ and $\langle \alpha^d,\mu_x\rangle=\langle\alpha^{2d},\mu_x\rangle= 0$. Let $$\tilde U_{\alpha}^i(y)=b_{x}^{(i)}\sigma^i(U_{\alpha}(y))(b_{x}^{(i)})^{-1}.$$ Then
\begin{eqnarray*}
\tilde U_{\alpha}^i(p^{-1}y)&=&\begin{cases}
U_{\alpha^i}(p^{-1}c_i\sigma^i(y))&\text{if }0\leq i<d\\
U_{\alpha^i+\beta^i}(c_i\sigma^i(y))&\text{if }d\leq i<2d\\
U_{\alpha^i}(c_i\sigma^i(y))&\text{if }2d\leq i<3d\\
U_{\alpha+\beta}(c_n\sigma^n(y))&\text{if }i=3d=n.
\end{cases}
\end{eqnarray*}
Let $R = \bar k[y]$ and $\R$ the $R$-frame chosen in \ref{(1.3.2)}. We define $g(y)\in G(\R_L)$ by
$$g(y)=g_x\tilde U^{3d-1}_{\alpha}(p^{-1}y)\dotsm \tilde U^{0}_{\alpha}(p^{-1}y).$$
Then
\begin{eqnarray*}
&&g(y)^{-1}b\sigma(g(y))\\&=&U_{\alpha}(-p^{-1}y)U_{\alpha^d+\beta}(-c_d\sigma^d(y))\left(U_{\alpha^{2d}}(-c_{2d}\sigma^{2d}(y))U_{\alpha+\beta}(c_{3d}\sigma^{3d}(y))\right.\\
&&\left.\quad \cdot~ U_{\alpha^{2d}}(c_{2d}\sigma^{2d}(y))\right)U_{\alpha^d+\beta}(c_d\sigma^d(y))p^{\mu_x}\dot w_x\\
&=&U_{\alpha}(-p^{-1}y)U_{\alpha^d+\beta}(-c_d\sigma^d(y))U_{\alpha+\beta}(c_{3d}\sigma^{3d}(y))\\
&&\quad \cdot~ U_{\alpha+\alpha^{2d}+\beta}(-c'\sigma^{2d}(y)\sigma^{3d}(y))U_{\alpha^d+\beta}(c_d\sigma^d(y))p^{\mu_x}\dot w_x\\
&=&U_{\alpha}(-p^{-1}y)U_{\alpha+\beta}(c_{3d}\sigma^{3d}(y))U_{\alpha+\alpha^d+\alpha^{2d}+2\beta}(-c''\sigma^d(y)\sigma^{2d}(y)\sigma^{3d}(y))\\
&&\quad \cdot~ U_{\alpha+\alpha^{2d}+\beta}(c'\sigma^{2d}(y)\sigma^{3d}(y))p^{\mu_x}\dot w_x\\
\end{eqnarray*}
with $c',c''\in\mathcal{O}_L$. Now $U_{\alpha}$ commutes with the other factors and can be moved to the right. We obtain that $g(y)^{-1}b\sigma(g(y))\in Kp^{\mu_x}K$. A computation analogous to the above constructs a point $g'$ ``at infinity'' and shows that it has the required properties, which finishes the proof of Lemma \ref{lempm3C}.
\end{proof}

\begin{remark}Example \ref{ex_for_lempm3C} is in Case 2.1 of the proof of Lemma \ref{lempm3C}. Another interesting example is the following: Let $G$ be a unitary similitude group such that $G_{L}\simeq \mathrm{GL}_3\times \mathbb{G}_{m, L}$ with standard simple roots $\beta_i=e_i-e_{i+1}$ for $i=1, 2$. The group $\Gamma=\{\mathrm{Id}, \sigma\}$ acts on $G_{L}$  with $\sigma\beta_i= \beta_{3-i}$ for $i=1, 2$. Take $M=T$,  $\alpha=\beta_2$ and the cocharacter $\mu_{x_0}$ is defined as follows which determines $b$ and $\mu$:
\begin{eqnarray*}
\mu_{x_0}: \mathbb{G}_{m, L} &\rightarrow& G_{L} \simeq \mathrm{GL}_3\times \mathbb{G}_{m, L} \\ y&\mapsto& (\mathrm{diag}( y, 1, y), y)
\end{eqnarray*}
Then the datum $( M, G, \Gamma\alpha, b, \mu)$ still satisfies all the conditions of Lemma \ref{lempm3C} and corresponds to Case 2.2 in that proof.
\end{remark}

\begin{cor}\label{cor_pm3weak_x}Let $\Omega\in\Phi_{N, \Gamma}$ be adapted and of type I. Let $x\in \bar I^{M, G}_{\mu, b}$. Suppose there exists $\alpha\in\Omega$ such that $\langle\alpha, \mu_x\rangle<0$. Then there exist $g, g'\in X^M_{\mu_x}(b)(W(\bar k))$ such that
\begin{itemize}
\item $g$ and $g'$ are in the same connected component of $X_{\mu}^G(b)$;
\item $w_M(g')-w_M(g)=\sum_{\beta\in \Omega}\beta^{\vee}$ in $\pi_1(M)^{\Gamma}$.
\end{itemize}
\end{cor}
\begin{proof}As $\Omega$ is of type I, by Lemma \ref{lempm3B} and Lemma \ref{lempm3C},  Proposition \ref{pm3weak} holds for $\Omega$. In particular, the corollary holds for $x=x_0$. Moreover if we replace $x_0$ by $x$ in Proposition \ref{pm3weak}, it still holds once we replace correspondingly $b$ by $b_x$. This means that there exist $g_1, g'_1\in X^M_{\mu_x}(b_x)(W(\bar k))$ such that
\begin{itemize}
\item $g$ and $g'$ are in the same connected component of $X_{\mu}^G(b)$;
\item $w_M(g')-w_M(g)=\sum_{\beta\in \Omega}\beta^{\vee}$ in $\pi_1(M)^{\Gamma}$.
\end{itemize} By Remark \ref{lemma_b=b_x}, $[b_x]=[b]$ in $B(M)$. So there exists an element $h\in M(L)$ such that $g\mapsto hg$ gives an isomorphism between $X^M_{\mu_x}(b_x)$ and $X^M_{\mu_x}(b)$. Therefore $g=hg_1, g'=hg'_1\in X^M_{\mu_x}(b)(W(\bar k))$ are the desired elements.
\end{proof}

\begin{lemma}\label{lempm3D}
Proposition \ref{pm3weak} holds if $\Omega$ is of type II or III and $\langle\tilde\alpha,\mu_x\rangle < 0$ for some $\alpha\in\Omega$ and some $x\in \bar I_{\mu, b}^{M, G}$.
\end{lemma}
\begin{proof}
Let $\alpha\in\Omega$ be as in the Lemma. Let $\Omega':=\Gamma\tilde{\alpha}$. Then $\Omega'$ is adapted and of type I. Therefore we can apply Corollary \ref{cor_pm3weak_x} to $(\Omega', x)$ and obtain elements  $g, g'\in X^M_{\mu_x}(b)(W(\bar k))$ such that $g\sim g'$ and $w_M(g')-w_M(g)=\sum_{\beta\in \Omega'}\beta^{\vee}=\sum_{\beta\in\Omega}\beta^\vee$ in $\pi_1(M)^{\Gamma}$.
\end{proof}

Proposition \ref{pm3weak} then follows immediately from Lemma \ref{lempm3B}, Lemma \ref{lempm3C}, and Lemma \ref{lempm3D}.

\section{Application to Rapoport-Zink spaces} \label{sec_app_RZ}
In this section, we apply the main results of this paper to (simple) unramified Rapoport-Zink spaces of EL type or unitary/symplectique PEL
type.
\setcounter{subsection}{1}
\setcounter{equation}{0}

\begin{para}
From now on, suppose $F=\mathbb{Q}_p$. In the previous sections, we have studied the connected components of affine Deligne-Lusztig varietes $X^G_\mu(b)$ defined from the datum $(G, b,\mu)$.  Now we require that the datum $(G, b, \mu)$ satisfies the following additional conditions:
\begin{itemize}
\item $G$ belongs to one of the following three cases:
\begin{itemize}
\item \emph{EL case}: $G=\mathrm{Res}_{\mathcal{O}_{F_1}|\mathbb{Z}_p}\mathrm{GL}(\Lambda_0)$ where $F_1$ is a finite unramified extension of $\mathbb{Q}_p$, and where $V$ is a finite dimensional $F_1$-vector space with $\Lambda_0\subset V$ a lattice.
\item \emph{PEL symplectic case}: $G=\mathrm{GSp}(\Lambda_0, \langle\cdot, \cdot\rangle)$ where $F_1$, $V$ are as above and where $\langle\cdot, \cdot\rangle: V\times V\rightarrow \mathbb{Q}_p$ is a non-degenerate alternating $\mathbb{Q}_p$-bilinear form on $V$ such that $\langle \lambda x, y\rangle=\langle x, \lambda y\rangle$ for all $x, y\in V$, $\lambda\in F_1$, and $\Lambda_0\subset V$ is an autodual lattice in $V$ for this form.

\item \emph{PEL unitary case}: $G=\mathrm{GU}(\Lambda_0, \langle\cdot, \cdot\rangle)$
where $F_1$, $V$ as above, $*$ is a non-trivial involution on $F_1$, $\langle\cdot, \cdot\rangle: V\times V\rightarrow \mathbb{Q}_p$ is a non-degenerate alternating hermitian form on $V,$ and $\Lambda_0\subset V$ is a autodual lattice in $V$ for this form.
\end{itemize}

\item  The weight decomposition of $\mu$ in $V\otimes_{\mathbb{Q}_p} L$ has only slopes 0 and 1, where we consider $\mu\in X_*(T)$ as the representation $$\mu: \mathbb{G}_{m, L}\rightarrow T_L \hookrightarrow G_{L}\hookrightarrow (\mathrm{Res}_{F_1|\mathbb{Q}_p}\mathrm{GL}(V))_{L}.$$
\end{itemize}

A datum $(G, b, \mu)$ satisfying the above conditions is called a (simple) unramified Rapoport-Zink datum of EL type or unitary/symplectique PEL type. To this kind of datum we can associate a Rapoport-Zink space $\breve{\mathcal{M}}=\breve{\mathcal{M}}(G, b, \mu)$. These spaces are formal schemes locally formally of finite type over $\mathrm{Spf} \mathcal{O}_L,$ which are defined as moduli spaces
parametrizing certain families of $p$-divisible groups in a fixed isogeny class. They are equipped with a natural action of
$J_b(\mathbb{Q}_p)$. For the precise definition of these spaces we refer to \cite{RZ}. There exists a $J_b(\mathbb{Q}_p)$-equivariant locally constant morphism on $\breve{\mathcal{M}}$
$$\varkappa_{\breve{\mathcal{M}}}: \breve{\mathcal{M}}(G, b,\mu)\rightarrow \mathrm{Hom}(X^*_{\mathbb{Q}_p}(G), \mathbb{Z}),$$ where $X^*_{\mathbb{Q}_p}(G)$ is the group of $\mathbb{Q}_p$-rational characters of $G$. The classification of $p$-divisible groups over $\bar{\mathbb F}_p$ via Dieudonn\'e theory, induces a natural bijection $\breve{\mathcal{M}}(G, b, \mu)(\bar{\mathbb{F}}_p)\simeq X^G_\mu(b)(W(\bar{\mathbb{F}}_p))$ compatible with the $J_b(\mathbb{Q}_p)$-action.
\end{para}

\begin{prop}\label{prop_conn_comp_surj}Suppose that $(G, b, \mu)$ is HN-indecomposable. Then the natural bijection $\theta: X^G_{\mu}(b)(W(\bar{\mathbb{F}}_p))\simeq \breve{\mathcal{M}}(G,b, \mu)(\bar{\mathbb{F}}_p)$ induces a map on the sets of connected components
\[\pi_0(X^G_{\mu}(b))\rightarrow \pi_0(\breve{\mathcal{M}}(G,b, \mu)),\]
which is necessarily surjective.
\end{prop}
\begin{proof} Let $R$ be a smooth integral $\bar k$-algebra, and $\R$ a frame for $R.$ We have to show that if $g_0, g_1\in X^{G}_{\mu}(b)(W(\bar{\mathbb{F}}_p))$ are connected via a $g \in X_{\mu}^G(b)(\R)$ then
$\theta (g_0)$ and $\theta(g_1)$ are in the same connected component in $\breve{\mathcal{M}}$.
Let $s_0, s_1\in \mathrm{Spec}(R)(\bar{\mathbb{F}}_p)$ with $g(s_0) = g_0$ and $g(s_1) = g_1,$ as in (\ref{conn_comp_ADLV}).

By Proposition \ref{1.1.9}, there exists an \'etale covering $f:\mathrm{Spec}(R')\rightarrow \mathrm{Spec} (R)$ such that $g^{-1}b\sigma(g)\in G(\R')p^{\mu}G(\R')$ where $\R'$ is the canonical frame for $R'$. It suffices to prove the statement with $\R$ replaced by the affine
ring of one of the connected components of $\Spec \R'.$
(Indeed, we can find a chain of elements $(h_i)_{1\leq i\leq n}\in X^{G}_{\mu}(b)(W(\bar{\mathbb{F}}_p))$ such that $h_1=g_0$ and $h_n=g_1$ and there exists $s_i, s'_i\in \mathrm{Spec} (R')(\bar{\mathbb{F}}_p)$ in the same connected component with $g(s_i)=h_i$ and $g(s'_i)=h_{i+1}$ for $1\leq i\leq n-1$. We can then consider separately each pair $(h_i, h_{i+1})$ with the connected component of $\mathrm{Spec}(R')$ containing $s_i$).  Therefore we reduce to the case when $g^{-1}b\sigma(g)\in G(\R)p^{\mu}G(\R)$.

Now we will define an element in $\breve{\mathcal{M}}(R)$ corresponding to $g$ by using Dieudonn\'e theory. The proof is very similar to the proof of \cite{Ki4} Lemma 1.4.6. Here we only give a sketch. Let $\Lambda_0\subset V$ be as in the definition of $G$. Let $M:=g (\Lambda_0\otimes_{\mathbb{Z}_p}\R)\subset V\otimes _{\mathbb{Q}_p}\R_L$. The Frobenius map $F=b\sigma$ acts on $M$. As the weight decomposition of $\mu$ on $V\otimes L$ has only slopes 0 and 1 , we have $pM\subset FM\subset M$. Therefore $M$ is stable under Frobenius and Verschiebung.

We write $\R_n$ for the ring $\R$ considered as an $\R$-algebra via $\sigma^n: \R\rightarrow \R$. Similarly let $R_n:=\R_n/p\R_n$. As the action of $\sigma$ on $\Omega^1_{\R/\mathcal{O}_L}$ is topologically nilpotent, there exists $n\in \mathbb{N}$ sufficiently large such that $g^{-1}dg\in \mathrm{End}(\Lambda_0)\otimes_{\mathbb{Z}_p} \Omega^1_{\R_n/\mathcal{O}_L}$. Then we can check that $g(\Lambda_0\otimes_{\mathbb{Z}_p}\R_{n})$ is stable under the canonical connection $$\nabla=1\otimes d: \Lambda_0\otimes_{\mathbb{Z}_p}\R_{n, L}\rightarrow \Lambda_0\otimes_{\mathbb{Z}_p}\Omega^1_{\R_{n, L}/L}.$$ Therefore, $(M\otimes_{\R}\R_n, \nabla, F, V)$ gives rise to a Dieudonn\'e crystal on $\R_n$ with $G$-structures. This corresponds to a point in $\breve{\mathcal{M}}(G, b,\mu)(R_n)$ by \cite{dJ} Theorem 4.1.1.

\end{proof}

\begin{para}
Recall that $G^{\ab}=G/G^{\der}$ is the cocenter of $G$. Then $X^*_{\mathbb{Q}_p}(G)=
X^*(G^{\ab})^{\Gamma}$ and $\pi_1(G)=\pi_1(G^{\ab})=X_*(G^{\ab})$ by (\cite{Bor} Lemma 1.5) since $G^{\der}$ is simply connected. Then by comparing the definition of $w_G$ and $\varkappa_{\breve{\mathcal{M}}}$, we can check that the following diagram commutes

\begin{eqnarray}\label{eqn_kappa_compatible}\xymatrix{ X^G_\mu(b)(W(\bar{\mathbb{F}}_p)) \ar[r]^{w_G}\ar[d]_{\sim}
& c_{b, \mu}\pi_1(G)^{\Gamma}\ar@{=}[r] & c_{b, \mu}X_*(G^{\ab})^{\Gamma}\ar[d]\\
\breve{\mathcal{M}}(G, b, \mu)(\bar{\mathbb{F}}_p)\ar[r]^{\varkappa_{\breve{\mathcal{M}}}}
&\mathrm{Hom} (X^*_{\mathbb{Q}_p}(G), \mathbb{Z})\ar@{=}[r]  &\mathrm{Hom}(X^{*}(G^{\ab})^{\Gamma}, \mathbb{Z})
}\end{eqnarray}
where the vertical arrow on the right is induced by the natural $\Gamma$-equivariant pairing
$X_*(G^{\ab})\otimes X^*(G^{\ab})\rightarrow \mathbb{Z}$.
\end{para}

\begin{thm}\label{thm_conn_comp_RZ} \begin{enumerate}
\item $\theta: \pi_0(X^G_{\mu}(b))\rightarrow \pi_0(\breve{\mathcal{M}}(G,b, \mu))$ is a bijection;

\item If $(\mu, b)$ is HN-irreducible, then $\varkappa_{\breve{\mathcal{M}}}$
induces an injection on the connected components $$\varkappa_{\breve{\mathcal{M}}}: \pi_0(\breve{\mathcal{M}})
\rightarrow \mathrm{Hom}(X^*_{\mathbb{Q}_p}(G), \mathbb{Z}).$$\end{enumerate}
\end{thm}
\begin{proof} Suppose $(\mu, b)$ is HN-irreducible. By Prop. \ref{prop_conn_comp_surj}, the above diagram induces a
 commutative diagram on the connected components:

\begin{eqnarray}\xymatrix{ \pi_0( X^G_\mu(b)) \ar[r]^{w_G}_{\sim}\ar[d]
& c_{b, \mu}X_*(G^{\ab})^{\Gamma}\ar[d]\\
\pi_0(\breve{\mathcal{M}})\ar[r]^-{\varkappa_{\breve{\mathcal{M}}}}
&\mathrm{Hom}(X^{*}(G^{\ab})^{\Gamma}, \mathbb{Z})
}\end{eqnarray}
where the top horizontal morphism is a bijection by Theorem \ref{thmzshk} and Corollary \ref{cartesianIII}. In order to show (1) and (2),
it suffices to show that $c_{b,\mu}X_*(G^{\ab})^{\Gamma}\rightarrow \mathrm{Hom}(X^*(G^{\ab})^{\Gamma}, \mathbb{Z})$ is injective.
Since $X_*(G^{\ab})^{\Gamma}$ is torsion free, it suffices to prove the statement after $\otimes \Q,$ and then the map is an isomorphism,
as $\Gamma$ acts on $X_*(G^{\ab})$ through a finite quotient.


We now prove (1) in the general case. If $(G, b, \mu)$ is Hodge-Newton-indecomposable, by Theorem \ref{thmzshkII}, we only need to
deal with the case when $b$ is $\sigma$-conjugate to $p^{\mu}$ with $\mu$ central.
We may assume that $b=p^{\mu}$. For any algebraically closed
extension $k$ of $\bar{\mathbb{F}}_p$, one uses Dieudonn\'e theory and the same computation as in
Remark \ref{remspecialcase} to show that
\begin{eqnarray*}\breve{\mathcal{M}}(G, b,\mu)(k)&=&\{g\in G(W(k)[1/p])/ G(W(k))\mid g^{-1}b\sigma(g)\in G(W(k))p^{\mu}G(W(k))\}\\
&=&G(\mathbb{Q}_p)/G(\mathbb{Z}_p)
\end{eqnarray*} where the third equality follows from Lang's theorem $H^1(\langle\sigma\rangle, G(W(k)))=0.$
It follows that $\breve{\mathcal{M}}(G, b,\mu)$ is discrete and (1) follows from Theorem \ref{thmzshk}.
It remains the case when $(G, b, \mu)$ is Hodge-Newton decomposable.
In this case there exists a standard parabolic $P$ with Levi subgroup $M$ containing $T$
and a $b'\in M\cap [b]$ such that $(M, b', \mu)$ is Hodge-Newton indecomposable.
We may assume $b'=b$. With $(M, b, \mu)$ and $(P, b,\mu)$ one can also associate analogs of
Rapoport-Zink spaces $\breve{\mathcal{M}}(M, b,\mu)$ resp.~$\breve{\mathcal{M}}(P, b,\mu)$.
They are moduli spaces of $p$-divisible groups with additional structure of the same type as
for $\breve{\mathcal{M}}(G, b,\mu)$, but which are in addition equipped with a slope decomposition resp.
with a slope filtration corresponding to $M$ resp. to $P$ (see \cite{Ma} for the precise construction).
One obtains naturally defined morphisms
\[\breve{\mathcal{M}}(M, b,\mu)
\stackrel{s}{\rightarrow} \breve{\mathcal{M}}(P, b,\mu)
\stackrel{p_2}{\rightarrow}
 \breve{\mathcal{M}}(G,
b,\mu).\] Moreover Mantovan also constructed a morphism $p_1:
\breve{\mathcal{M}}(P, b,\mu)^{\an} \rightarrow
\breve{\mathcal{M}}(M, b,\mu)^{\an}$ satisfying $p_1\circ
s^{\an}=\mathrm{Id}$ by considering the graded pieces of the
filtration on the $p$-divisible groups, where $(-)^{\an}$ always
denote the generic fiber. Then $s^{\an}$ induces an injection on
the connected components. By \cite{Sh} Prop 6.3, $p^{\an}_2$
induces an isomorphism of analytic spaces on the generic fiber, we find that
$$ \pi_0(\breve{\mathcal{M}}(M, b,\mu)) \iso \pi_0(\breve{\mathcal{M}}(M, b,\mu))^{\an}
\hookrightarrow \pi_0(\breve{\mathcal{M}}(G, b,\mu))^{\an} \iso \pi_0(\breve{\mathcal{M}}(G, b,\mu)).$$
Here the two bijections follow from the fact that $\breve{\mathcal{M}}(M, b,\mu)$ and $\breve{\mathcal{M}}(G,
b,\mu)$ are both formally smooth by Grothendieck-Messing
deformation theory (cf. \cite{dJ} Theorem 7.4.1).
Thus $\pi_0(p_2\circ s)$ is an injection.
But we already know that $\theta$ induces a surjective map on connected components.
Hence, using Proposition \ref{prophndecomp}, $\pi_0(p_2\circ s)$ is also surjective.
Then (1) follows from the Hodge-Newton-indecomposable case.
\end{proof}

\begin{para}
Theorem \ref{thm_conn_comp_RZ} confirms Conjecture 6.1.1 of \cite{chen}. As the main results in \cite{chen} are proved after assuming that conjecture, we can now state all these results without this hypothesis.

Let $\breve{\mathcal{F}}$ be the flag variety of parabolic subgroups of type $\mu$ of $G_{/L}$. Let $\breve{\pi}: \breve{\mathcal{M}}^{\an}\rightarrow \breve{\mathcal{F}}^{\an}$ be the period morphism (cf. \cite{RZ} chapter 5), where $\breve{\mathcal{M}}^{\an}$ is the generic fiber of $\breve{\mathcal{M}}$ as Berkovich's analytic space, and $\breve{\mathcal{F}}^{\an}$
is Berkovich's analytic space associated to $\breve{\mathcal{F}}$. Let $\breve{\mathcal{F}}^a$ be the image of $\breve{\pi}$.
\end{para}

\begin{prop}[cf. \cite{chen} Lemma 6.1.3] If $(\mu, b)$ is HN-irreducible, then $\breve{\mathcal{F}}^a$ is connected.
\end{prop}

\begin{para} Recall that $(\breve{\mathcal{M}}_{\tilde K})_{\tilde K\subset G(\mathbb{Z}_p)}$  is a tower of finite \'etale covers over $\breve{\mathcal{M}}^{\an}$
parametrizing the $\tilde K$-level structures with $\tilde K\subset G(\mathbb{Z}_p)$ open compact. The group $J_b(\mathbb{Q}_p)$ acts on the left on each $\breve{\mathcal{M}}_{\tilde K}$
and the group $G(\mathbb{Q}_p)$ acts on the right on the tower $(\breve{\mathcal{M}}_{\tilde K})_{\tilde K}$ by Hecke correspondences.
As in the introduction we have the map
$$\delta = (\delta_{J_b},\delta_G,\chi_{\delta_G,\mu}): J_b(\Q_p)\times G(\Q_p) \times \Gal(\bar L/L) \rightarrow G^{\ab}(\Q_p).$$
\end{para}

\begin{thm}[cf. \cite{chen} Theorem 6.3.1]\label{geo_conn_comp} If $(\mu, b)$ is HN-irreducible, then the  action of
$J_b(\Q_p)\times G(\Q_p) \times \Gal(\bar L/L)$ on $\pi_0(\breve{\mathcal{M}}_{\tilde K}\hat{\otimes}\mathbb{C}_p)$ factors through
$\delta,$ and makes $\pi_0(\breve{\mathcal{M}}_{\tilde K}\hat{\otimes}\mathbb{C}_p)$ into a $G^{\ab}(\Q_p)/\delta_G(\tilde K)$-torsor.
In particular, we have bijections
\[ \pi_0(\breve{\mathcal{M}}_{\tilde K}\hat{\otimes}\mathbb{C}_p)\simeq  G^{\ab}(\mathbb{Q}_p)/\delta_G(\tilde K)\] which are compatible when $\tilde K$ varies.
\end{thm}

\begin{remark}
Write
$$\pi_0(\breve{\mathcal{M}}_{\infty}\hat{\otimes}\mathbb{C}_p) : = \ilim_{\tilde K}\pi_0(\breve{\mathcal{M}}_{\tilde K}\hat{\otimes}\mathbb{C}_p).$$
Then the theorem above is equivalent to the statement that when $(\mu, b)$ is HN-irreducible, the action of $J_b(\Q_p)\times G(\Q_p) \times \Gal(\bar L/L)$ on
$\pi_0(\breve{\mathcal{M}}_{\infty}\hat{\otimes}\mathbb{C}_p)$ makes this set a $G^{\ab}(\Q_p)$-torsor.

When $\breve{\mathcal{M}}$ is of EL type \footnote{This condition is presumably unnecessary},
then we can form the inverse limit $\breve{\mathcal{M}}_{\infty, \mathbb{C}_p}=\varprojlim_{\tilde K} \breve{\mathcal{M}}_{\tilde K}\hat\otimes\mathbb{C}_p$ as a perfectoid space as in \cite{SW}. In this case the set $\pi_0(\breve{\mathcal{M}}_{\infty}\hat{\otimes}\mathbb{C}_p),$ defined formally above coincides with the set of connected components of $\breve{\mathcal{M}}_{\infty, \mathbb{C}_p}.$
\end{remark}

\bibliographystyle{amsalpha}




\end{document}